\documentclass[12pt,a4paper]{amsart}
\usepackage[left=2.5cm,right=2.5cm,top=3.6cm,bottom=4cm]{geometry}

\usepackage{url}
\usepackage{ifthen}
\usepackage{amsmath}
\usepackage{amsfonts}
\usepackage{amssymb}
\usepackage{pstricks}
\usepackage{pst-plot}
\usepackage{epsfig}
\usepackage{comment}
\usepackage{graphicx,color}

\usepackage{lipsum}

 \theoremstyle{plain}
 \newtheorem{thm}{Theorem}[section]
 \newtheorem{cor}[thm]{Corollary}
 \newtheorem{lem}[thm]{Lemma}
 
 \newtheorem{propo}[thm]{Proposition}
 \theoremstyle {definition}
 \newtheorem{Def}[thm]{Definition}
 
 \theoremstyle{remark}
 \newtheorem{rem}[thm]{Remark}

\renewcommand{\u}{\cup} 
\newcommand{\set}[1]{ \{ #1 \}} 

\newcommand{\ci}{\subseteq} 

\newcommand{\tif}{\text{if }}

\newcommand{\tand}{\text{ and }}

\newcommand{\tfor}{\text{for }}

\newcommand{\twith}{\text{with }}

\newcommand{\ga}{\alpha}
\newcommand{\gb}{\beta}

\newcommand{\gd}{\delta}
\newcommand{\gep}{\varepsilon}
\newcommand{\gf}{\varphi}
\newcommand{\gga}{\gamma}
\newcommand{\gh}{\eta}

\newcommand{\gt}{\theta}
\newcommand{\gk}{\kappa}
\newcommand{\gl}{\lambda}
\newcommand{\gm}{\mu}
\newcommand{\gn}{\nu}

\newcommand{\gx}{\xi}
\newcommand{\gs}{\sigma}
\newcommand{\gr}{\rho}
\newcommand{\gta}{\tau}

\newcommand{\gv}{\vartheta}
\newcommand{\gp}{\psi}
\newcommand{\gz}{\zeta}

\newcommand{\gD}{\Delta}
\newcommand{\gF}{\Phi}

\newcommand{\gL}{\Lambda}
\newcommand{\gO}{\Omega}
\newcommand{\gT}{\Theta}
\newcommand{\gS}{\Sigma}
\newcommand{\gU}{\Upsilon}

\newcommand{\gP}{\Psi}
 
\def\B#1{\textnormal{\textbf{#1}}}
\newcommand{\C}[1]{{\mathcal{#1}}} 
\newcommand{\BB}[1]{{\mathbb{#1}}} 

\newcommand{\BF}[1]{{\textbf#1}}   

\newcommand{\refT}[1]{Theorem ~\ref{#1}}
\newcommand{\refL}[1]{Lemma ~\ref{#1}}

\newcommand{\refP}[1]{Proposition ~\ref{#1}}

{\par \samepage}%
{\par}

\newcommand{\ol}{\overline}

\newcommand{\qq}{\qquad}

\renewcommand{\Im}{\operatorname{Im}}
\renewcommand{\Re}{\operatorname{Re}}
 
\newcommand{\qg}{\C{QG}_N} 
\newcommand{\Dom}{\operatorname{Dom}\, } 
\newcommand{\diam}{\operatorname{diam}\,} 
\newcommand{\inte}{\operatorname{int}\,}   
\newcommand{\IS}{\mathcal{I}\hspace{-1pt}\mathcal{S}} 
\newcommand{\ex}{\operatorname{\mathbb{E}xp}} 
\newcommand{\cv}{\textup{cv}} 
\newcommand{\cp}{\textup{cp}} 
\newcommand{\co}[1]{^{\circ {#1}}} 
\newcommand{\PC}[1]{{#1}\ltimes \IS}
\newcommand{\Td}{\operatorname{d_{\textnormal{Teich}}}}
\newcommand{\Sd}{\operatorname{D_{\textnormal{S}}}}
\newcommand{\Dil}{\operatorname{Dil }}

\newcommand{\inv}{\operatorname{inv}} 
\newcommand{\saw}{\operatorname{saw}} 

\newcommand\npr[1]{\ifthenelse{\equal{#1}{1}}%
                     {\operatorname{\mathcal{R}_{\scriptscriptstyle NP}}}{\operatorname{\mathcal{R}^{\circ #1}_{\scriptscriptstyle NP}}}}
                   
\newcommand\nprt[1]{\ifthenelse{\equal{#1}{1}}%
                     {\operatorname{\mathcal{R}_{\scriptscriptstyle NP-t}}}{\operatorname{\mathcal{R}^{\circ #1}_{\scriptscriptstyle NP-t}}}}

\newcommand\nprb[1]{\ifthenelse{\equal{#1}{1}}%
                     {\operatorname{\mathcal{R}_{\scriptscriptstyle NP-b}}}{\operatorname{\mathcal{R}^{\circ #1}_{\scriptscriptstyle NP-b}}}}

\newcommand\plr[1]{\ifthenelse{\equal{#1}{1}}%
                     {\operatorname{\mathcal{R}_{\scriptscriptstyle PL}}}{\operatorname{\mathcal{R}^{\circ #1}_{\scriptscriptstyle PL}}}}                     

\title{Satellite renormalization of  quadratic polynomials}
\author{Davoud Cheraghi}
\address{Department of Mathematics, Imperial College London, London SW7 2AZ, UK}
\email{d.cheraghi@imperial.ac.uk}
\author{Mitsuhiro Shishikura}
\address{Department of Mathematics, Kyoto University, Kyoto 606-8502, JAPAN}
\email{mitsu@math.kyoto-u.ac.jp}
\date{\today}

\begin{document}

\begin{abstract}
We prove the uniform \textit{hyperbolicity} of the \textit{near-parabolic renormalization} operators 
acting on an infinite-dimensional space of holomorphic transformations. 
This implies the \textit{universality of the scaling laws}, conjectured by physicists in the 70's, 
for a combinatorial class of bifurcations. 
Through near-parabolic renormalizations the \textit{polynomial-like} renormalizations of \textit{satellite} 
type are successfully studied here for the first time, and new techniques are introduced to analyze the 
fine-scale dynamical features of maps with such \textit{infinite renormalization structures}. 
In particular, we confirm the \textit{rigidity conjecture} under a \textit{quadratic growth} condition 
on the combinatorics.
The class of maps addressed in the paper includes infinitely-renormalizable maps with degenerating 
geometries at small scales (lack of \textit{a priori} bounds). 
\end{abstract}
\maketitle

\renewcommand{\thethm}{\Alph{thm}}
\section{Introduction}\label{S:intro} 
\subsection{Renormalization conjecture}
In the 1970's, physicists Feigenbaum \cite{Fei78} and independently Coullet-Tresser \cite{TrCo78}, 
working numerically, observed \textit{universal scaling laws} in the cascades of \textit{doubling bifurcations} 
in generic families of one-dimensional real analytic transformations.
To explain this phenomena, they conjectured that a \textit{renormalization operator} 
acting on an infinite-dimensional function space is \textit{hyperbolic} with a one-dimensional 
unstable direction and a co-dimension-one stable direction.
Subsequently, this remarkable feature was observed in other \textit{bifurcation combinatorics} 
(besides the doubling one) in generic families of real and complex analytic transformations \cite{DGP79}.
A conceptual explanation for this phenomena has been the focus of research ever since. 

By the seminal works of Sullivan, McMullen, and Lyubich in the 90's, there is a proof of the renormalization 
conjecture for combinatorial types arising for real and some complex analytic transformations, \cite{Sul92,Mc2,Lyu02}, 
see also Avila-Lyubich~\cite{AvLy11}. 
A central concept in these works is the \textit{pre-compactness} of the \textit{polynomial-like renormalization};
a non-linear operator introduced by Douady and Hubbard in the 80's \cite{DH84}. 
While this provides the first conceptual proof of the renormalization conjecture for a class of combinatorial 
types, lack of the pre-compactness of this renormalization operator with arbitrary combinatorics 
is a major obstacle to establishing the renormalization conjecture for arbitrary combinatorics.

Inou and Shishikura in 2006 \cite{IS06} introduced a sophisticated pair of renormalizations, 
called \textit{near-parabolic renormalizations}, acting on an infinite-dimensional class of complex analytic 
transformations near \textit{parabolic maps}. 
Using a new analytic technique introduced by the first author \cite{Ch10-I,Ch10-II} to control the dependence 
of these nonlinear operators  on the data, we prove the hyperbolicity of these renormalization operators in this paper, 
Theorem~\ref{T:hyperbolicity}.
This implies the universality of the scaling laws for a (combinatorial) class of bifurcations.
Our result covers some combinatorial types where the polynomial-like renormalizations are 
not pre-compact. 

\subsubsection*{Hyperbolicity versus rigidity} 
The proof of the expansion part of the hyperbolicity by Lyubich relies on a major result  
on the \textit{combinatorial rigidity} of the underlying maps.
The latter involves a detailed combinatorial and analytic study of the dynamics of the underlying maps, 
successfully accomplished through a decade of intense studies \cite{Ly97,GrSw97}, 
see also \cite{Hub93,Mc1,LyYa97,LV98,Hi00} and the references therein.
As a result of this, it is slightly short of providing the rates of expansions. 
In contrast, the expansion part of the hyperbolicity stated here comes from the relations between the conformal data 
on the large-scale and the small-scale, related via the renormalizations, see Theorem~\ref{T:Lipschitz}.
This provides basic formulas for the rates of expansions, and in turn yields an elementary proof of the rigidity conjecture
for a class of combinatorial types, see Theorem~\ref{T:rigidity}.

\subsubsection*{Tame and wild dynamics} 
\textit{A priori bounds}, a notion of pre-compactness on the non-linearities of long return maps 
to small scales, is a key concept that has been widely used since the 90's to analyze the fine-scale 
structure of the dynamics of real and complex analytic transformations (tame dynamics). 
This has also been at the center of the arguments by Sullivan-McMullen-Lyubich.
The hyperbolicity result in this paper applies to classes of transformations that do not enjoy 
the \textit{a priori bounds}. 
It also treats maps (of bounded type) that are conjectured to enjoy the \textit{a priori} bounds, but 
remained mysterious to date.
Our approach provides a strong set of tools to describe the fine-scale dynamics of these maps 
using \textit{towers} of near-parabolic renormalizations, see Theorem~\ref{T:unique-parameter-in-class}.  
In particular, in forthcoming papers we shall construct the first examples of analytic 
transformations with some pathological phenomena. 
 
Below, we state the above notions and results more precisely.

\subsection{Near-parabolic renormalization operators}
In mathematics, renormalization is a strong tool to study fine-scale structures in low-dimensional dynamics. 
It is a procedure to control the divergence of large iterates of a map through \textit{regularizations}. 
Starting with a class of maps, to each $f$ in the class, one often identifies an appropriate iterate of $f$ 
on a region in its domain of definition, which, once viewed in a suitable coordinate on the region 
(the regularization), belongs to the same class of maps.
Remarkably, iterating a renormalization operator on a class of maps provides significant information 
about the behavior of individual maps in the class. 

Inou and Shishikura in \cite{IS06} introduced a renormalization scheme to study the local dynamics of 
\textit{near-parabolic} holomorphic transformations.
More precisely, there is an infinite-dimensional class of maps $\mathcal{F}_0$ (\textit{the non-linearities}), 
consisting  of holomorphic maps $h$ defined on a neighborhood of $0$, with $h(0)=0$, $h'(0)=1$, 
and $h$ has a particular covering property from its domain onto its range. 
For $\gr>0$, let (\textit{the set of linearities}) 
\[A(\gr)= \{\ga\in \BB{C} \mid 0 < |\ga| \leq \gr ,  |\Re \ga| \geq |\Im \ga|\}.\]
Define the class of maps 
\[A(\gr) \ltimes \mathcal{F}_0 
= \{h(e^{2\pi \B{i} \ga} z) \mid \ga \in A(\gr), h\in \C{F}_0\}.\] 
For $\gr$ small enough, every $\ga \ltimes h$ in the above class has two distinct fixed points 
$0$ and $\gs=\gs(\ga\ltimes h)$, with derivatives $(\ga \ltimes h)'(0)=e^{2\pi \B{i} \ga}$ and 
$(\ga \ltimes h)'(\gs)=e^{2\pi \B{i} \gb}$, where $\gb=\gb(\ga\ltimes h)\in \BB{C}$ and $-1/2 < \Re \gb \leq 1/2$. 
There are two renormalization operators, called the \textit{top near-parabolic renormalization} and 
the \textit{bottom near-parabolic renormalization} acting on the class $A(\gr) \ltimes \mathcal{F}_0$. 
They are defined as some sophisticated notions of return maps of $\ga\ltimes h$ near $0$ and $\gs$, respectively,
viewed in some canonically defined coordinates. 
We denote these by $\nprt{1}$ and $\nprb{1}$, respectively, and refer to them as NP-renormalizations. 
According to Inou and Shishikura, the non-linearities of $\nprt{1}(\ga \ltimes h)$ and 
$\nprb{1}(\ga\ltimes h)$ belong to the same class $\mathcal{F}_0$, that is, 
\begin{equation*}
\nprt{1}(\ga \ltimes h)= (\hat{\ga}(\ga,h) \ltimes \hat{h}(\ga,h)),\quad 
\nprb{1}(\ga \ltimes h)= (\check{\ga}(\ga,h) \ltimes \check{h}(\ga,h)),
\end{equation*}
where $\hat{h}(\ga,h)$ and $\check{h}(\ga,h)$ belong to $\C{F}_0$. 
It follows from the construction that $\hat{\ga}(\ga, h)= -1/\ga \mod \BB{Z}$ and $\check{\ga}(\ga, h)= -1/\gb \mod \BB{Z}$ 
(so $\hat{\ga}$ and $\check{\ga}$ are not necessarily in $A(\gr)$).


A crucial step here is to understand the dependence of these renormalization operators on the data.
In \cite{IS06} $\C{F}_0$ is identified with a Teichm\"uller metric  in order to 
establish the contractions of the maps $h \mapsto \hat{h}(\ga \ltimes h)$ and 
$h \mapsto \check{h}(\ga \ltimes h)$ on $\C{F}_0$, for each fixed $\ga$. 
On the other hand, to control the maps $\ga \mapsto \hat{h}(\ga \ltimes h)$ and 
$\ga \mapsto \check{h}(\ga \ltimes h)$, from $A(\gr)$ to $\C{F}_0$, one faces the canonic 
transcendental mappings with highly distorting nature that appear as the regularizations in 
the definitions of these renormalization operators.  

An analytic approach has been introduced by the first author in \cite{Ch10-I,Ch10-II} to control 
the geometric quantities, and their dependence on the data,  involved in these renormalization schemes. 
That is, to discard the distortions via certain model maps, and study the differences in the framework 
of nonlinear elliptic partial differential equations.
We extend this approach here to prove an upper bound on the dependence of these 
regularizations (and the renormalizations) on the data. 
A key step here is to study the variations of (the hyperbolic norm of) the \textit{Schwarzian derivatives} of 
$\hat{h}(\ga \ltimes h)$ and $\check{h}(\ga \ltimes h)$ as a function of $\ga$. 

\begin{thm}\label{T:Lipschitz}
There exists a constant $L$ such that for every $h\in \mathcal{F}_0$, the maps 
$\ga \mapsto \hat{h}(\ga, h), $ and $\ga \mapsto \check{h}(\ga,h)$ are L-Lipschitz with respect to 
the Euclidean metric on $A(\gr)$ and the Teichm\"uller metric on $\mathcal{F}_0$.
\end{thm}

The Gauss maps $\hat{\ga}(\ga, h)= -1/\ga \mod \BB{Z}$ and $\check{\ga}(\ga, h)= -1/\gb \mod \BB{Z}$ make 
$\nprt{1}$ and $\nprb{1}$ expanding in the first coordinates. 
Combining these bounds, we build \textit{cone fields} in $A(\gr) \ltimes \mathcal{F}_0$ that are respected 
by the maps $\nprt{1}$ and $\nprb{1}$. 

\begin{thm}\label{T:hyperbolicity}
The renormalizations operators $\nprt{1}$ and $\nprb{1}$ are uniformly hyperbolic on 
$A(\gr) \ltimes \C{F}_0$. 
Moreover, the derivatives of these operators at each point in $A(\gr) \ltimes \C{F}_0$ have an invariant 
one-dimensional expanding direction and an invariant uniformly contracting co-dimension-one direction.
\end{thm}

In the above theorem, the rates of expansions along unstable directions are given in terms 
of the Gauss map and a holomorphic index formula.


\subsection{Polynomial-like renormalizable versus near-parabolic renormalizable}
\label{SS:I-NP-renormalizable}
To repeat applying $\nprt{1}$ or $\nprb{1}$ to the map $\nprt{1}(\ga\ltimes h)= \hat{\ga} \ltimes \hat{h}$ 
one requires the complex rotation $\hat{\ga}= -1/\ga \mod \BB{Z}$ belong to $A(\gr)$. 
Similarly, to apply them to $\nprb{1}(\ga \ltimes h)= \check{\ga} \ltimes \check{h}$ one requires 
$\check{\ga}= -1/\gb \mod \BB{Z}$ belong to $A(\gr)$. 
In general, to iterate some arbitrary composition of these operators at some $\ga\ltimes h$, one needs 
the inductively defined \textit{complex rotations} at $0$ belong to $A(\gr)$. 
For instance, to apply $\nprt{1}$ infinitely often, we require $\ga$ be real and the continued fraction expansion of 
$\ga=[a_1,a_2,a_3, \dots]$ consist of entries $a_i\geq 1/\gr $.  
It follows from Theorem~\ref{T:hyperbolicity} that the set of $\ga \ltimes h$, where an 
infinite mix of NP-renormalizations may be applied at, consists of a bundle over a 
Cantor set in $A(\gr)$, with fibers isomorphic to the class $\C{F}_0$, see Theorem~\ref{T:Cantor-structure}.

To employ the theory, one faces the problem of whether a given map lies on the (implicitly defined) 
set of infinitely NP-renormalizable maps. 
We discuss two strategies here: one is based on successive perturbations, and the other is 
based on somehow knowing the complex rotations of a sequence of periodic points of the given map 
beforehand. 
Below we discuss an instance of the first strategy and in Section~\ref{SS:combinatorial-rigidity-conjecture} we discuss 
an instance of the second strategy. 

Let $P_c(z)=z^2+c$, $c\in \BB{C}$. 
The Mandelbrot set 
\[M= \{c\in \BB{C}\mid 
\text{the orbit } \langle P_c\co{n}(0) \rangle_{n=0}^\infty \text{ remains bounded}\}\]
is the set of parameters $c\in \BB{C}$ where $P_c$ has a connected Julia set.  
To explain the appearance of many homeomorphic copies of $M$ within $M$, 
Douady and Hubbard in the 80's \cite{DH85} introduced the foundational notion of  
\textit{polynomial-like} (abbreviated by PL) \textit{renormalization}. 
A map $P_c: \BB{C} \to \BB{C}$ is PL-renormalizable if there exist an integer $q\geq 2$ and simply 
connected domains $0 \in U \Subset V \subset \BB{C}$ such that $P_c\co{q}: U \to V$ is a proper branched 
covering of degree two and the orbit of $0$ under the map $P_c\co{q}: U \to V$ remains in $U$. 
Moreover, when there is a fixed point of $P_c$ in all the domains $P_c\co{i}(U)$, for $0\leq i \leq q-1$, 
the PL-renormalization is said of \textit{satellite} type. 
In turn, if $P_c \co {q}: U \to V$ is PL-renormalizable of satellite type, $P_c$ is called two times 
PL-renormalizable of satellite type. 
\textit{Infinitely PL-renormalizable of satellite type} 
is naturally defined as when this scenario occurs infinitely often.
These are complex analogues of the period doubling bifurcations (Feigenbaum phenomena), which were 
successfully studied in the 90's, while these complex analogues remained widely out of reach.

When a $P_c$ is PL-renormalizable of satellite type, the permutation of the domains $P_c \co {i}(U)$ 
about the fixed point, for $0\leq i \leq q-1$, under the action of $P_c$ may be described by a non-zero 
rational number $p/q \in (-1/2,1/2]$.
Naturally, an infinitely PL-renormalizable of satellite type gives rise to a sequence of non-zero 
rational numbers $\langle p_i/q_i \rangle_{i=1}^\infty$ in the interval $(-1/2, 1/2]$. 
This describes the \textit{combinatorial behavior} of the map. 
We use the notation
\[\langle m_i : b_{i,j} : \gep_{i,j}\rangle_{i=1}^\infty,\]  
with integers $m_i\geq 1$, $b_{i,j}\geq 2$, and $\gep_{i,j}=\pm1$, for $i\geq 1$ and $1\leq j \leq m_i$, 
to denote the sequence of rational numbers defined by the (modified) continued fractions 
\[\frac{p_i}{q_i}= 
\cfrac{\gep_{i,1}}{b_{i,1}+ \cfrac{\gep_{i,2}}{\ddots +\cfrac{\gep_{i,m_i}}{b_{i,m_i}}}}, i\geq 1.\]

\begin{thm}\label{T:unique-parameter-in-class}
There exists $N\geq 2$ such that for every sequence of rational numbers 
$\langle m_i : b_{i,j}: \gep_{i,j} \rangle_{i=1}^\infty$ with all $b_{i,j}\geq N$ 
there is $c \in M$ such that $P_{c}$ is infinitely PL-renormalizable of satellite type with combinatorics 
$\langle m_i : b_{i,j}: \gep_{i,j} \rangle_{i=1}^\infty$ and it is also infinitely NP-renormalizable. 
Moreover, the successive types of NP-renormalizations is given by 
\begin{equation*}\label{E:NP-type}
\dots \circ (\nprt{(m_n-1)}  \circ \nprb{1} ) \circ \dots \circ (\nprt{(m_2-1)} \circ \nprb{1}) 
\circ (\nprt{m_1-1} \circ \nprb{1}). 
\end{equation*}
\end{thm}

The parameter $c$ in the above theorem is obtained by an infinite perturbation procedure.
That is, by successively following the boundaries of the hyperbolic components of $M$ 
bifurcating one from the previous one. 
To this end we introduce a continued fraction type of algorithm (with correction terms 
satisfying universal laws) that produces the successive complex rotations at $0$ 
along the infinite NP renormalizations of $P_c$.

Successively applying NP-renormalizations produces a chain of maps linked via the regularizations, that is, 
the renormalizarion tower. 
This allows one to study fine-scale dynamical features of the original map.
For instance, being infinitely $\nprt{1}$ renormalizable has already led to a breakthrough on the 
dynamics of maps tangent to irrational rotations. 
It was used by Buff and Ch\'eritat~\cite{BC12} to complete a remarkable program to construct 
quadratic polynomials with Julia sets of positive area, see also \cite{Ya08}.
It is used in a series of papers \cite{Ch10-I,Blo10,Ch10-II,AC12,CC13} to confirm a number of conjectures on the 
dynamics of those maps, and is still being harvested.
When $\nprb{1}$ appears infinitely often in the chain of NP-renormalizations, we are dealing with 
the complex analogues of the real Feigenbaum maps. 
We shall use Theorem~\ref{T:unique-parameter-in-class} to describe fine dynamical features of these 
maps in a series of papers to appear in future.  


\subsection{Rigidity conjecture}\label{SS:combinatorial-rigidity-conjecture}
The \textit{combinatorial rigidity} conjecture in the quadratic family
suggests that the ``combinatorial behavior'' of a quadratic polynomial $P_c$ uniquely determines $c$, provided $c\in M$ and all 
periodic points of $P_c$ are repelling. 
This remarkable feature is equivalent to the \textit{local connectivity of the Mandelbrot set}
and implies the \textit{density of hyperbolic maps} within this family; a special case of a 
conjecture attributed to Fatou \cite{Fat20a}. 

Yoccoz in the 80's proved the rigidity conjecture for quadratic polynomials that are not PL-renormalizable, 
see \cite{Hub93}. 
In \cite{Sul92}, Sullivan proposed a program, based on \textit{a priori} bounds and pull-back methods, to 
study the rigidity conjecture for infinitely renormalizable quadratic polynomials. 
This has been the subject of intense studies for real values of $c$ in the 90's, \cite{GrSw97,Ly97,LyYa97,LV98,McM98}. 
The symmetry of the map with respect to the real line plays an important role in these studies.
When PL-renormalizations are not of satellite type (called primitive type), the pre-compactness 
is established for a wide class of combinatorial types \cite{Ly97,K2,KL2}.
However, when all PL-renormalizations are of satellite type, there is not a single combinatorial class for which 
the \textit{a priori bounds} is known. 
But, it is known that for some combinatorial types this property may not hold \cite{Sor00}. 
 
We study the rigidity problem for PL-renormalizations of satellite type via their near-parabolic renormalizations. 
To this end, we need to know when all PL-renormalizable maps of a given combinatorial type 
are infinitely NP-renormalizable. 
That is, to know, beforehand, the location of the complex rotations of the cycles associated with the 
PL-renormalizations.
We gain this information using a control on the geometry of the boundaries of the hyperbolic components 
of $M$ that is proved in this paper, as well as the combinatorial-analytic multiplier inequality of Pommerenke-Levin-Yoccoz 
\cite{Hub93}. 

For $N\geq 1$, define the class of sequences of rational numbers 
\[ \qg=\Big\{ \big \langle \frac{p_i}{q_i} \big \rangle_{i=1}^\infty 
=\langle m_i : b_{i,j} : \gep_{i,j} \rangle_{i=1}^\infty \Big|  
\begin{array}{l}
b_{1,1}\geq N, b_{i,j+1} \geq b_{i,j}^2, b_{i+1,1} \geq q_i^2 \\
 \forall i\geq 1, 1\leq j \leq m_i-1.
\end{array}
\Big\}.\]
For a sequence of non-zero rational numbers $\langle p_i/q_i \rangle_{i=1}^n$ in $(-1/2, 1/2]$, let 
$M(\langle p_i/q_i \rangle_{i=1}^n)$ denote the set of $c$ in $M$ such that $P_c$ is $n$ times 
PL-renormalizable of satellite type $\langle p_i/q_i \rangle_{i=1}^n$.

\begin{thm}\label{T:rigidity}
There are constants $N$, $C$, and $\gl\in (0,1)$ such that for every 
$\langle p_i/q_i \rangle_{i=1}^\infty$ in $\qg$,  we have 
\[\diam M(\langle p_i/q_i \rangle_{i=1}^n) \leq C\gl^n.\] 
In particular, if $P_c$ is infinitely PL-renormalizable of satellite type $\langle p_i/q_i \rangle_{i=1}^\infty$ 
in $\qg$, it is combinatorially rigid, and the Mandelbrot set is locally connected at $c$. 
\end{thm}

If one chooses $m_i=1$, for all $i\geq 1$, then a sequence $\langle p_i/q_i \rangle_{i=1}^\infty$ 
belongs to $\qg$ provided $q_1\geq N$,  $p_i=\pm 1$, and $q_{i+1} \geq q_i^2$, for all $i\geq 1$. 
Choosing a larger value for some $m_i$ allows us to have a rational number $p_i/q_i$ 
of mixed type, but this requires the later denominators become large because of the 
condition $b_{i+1,1} \geq q_i$. 
G.~Levin had already proved the combinatorial rigidity under the relative growth 
conditions $\limsup_{n} q_{n+1}/(q_1 q_2 \dots q_n)^2 >0$ and 
$\sup |p_n/q_n| q_0 q_1 \dots q_{n-1}\leq \infty$, \cite{Lev11,Lev14}. 
This is a faster growth condition on the denominators, but it covers rational numbers with certain numerators 
that are not covered in the above theorem. 
For parameters satisfying these growth conditions, he controls the location of the sequence of periodic 
cycles that consecutively bifurcate one from another; quantifying a construction due to Douady and 
Hubbard \cite{Sor00}, to obtain non-locally connected Julia sets. 
His approach is different from the one presented in this paper. 

We note that the \textit{post-critical} set (i.e.\ the closure of the orbit of the critical point) of 
the maps in the above theorem do not enjoy the a priori bounds (bounded geometry) proposed 
in the program of Sullivan.
The geometry of the post-critical set highly depends on the successive complex rotations at $0$ 
produced by successive NP-renormalizations. 
Moreover, the Pommerenke-Levin-Yoccoz inequality does not provide the kind of estimates 
to deduce that the post-critical sets of all maps with the same combinatorial behavior have comparable 
geometries.
Besides overcoming this issue, our approach allows us to treat all types of geometries at once, 
rather than dealing with fine geometric considerations dependent on the combinatorics, 
investigated in part two of \cite{Ly97} and in \cite{Che10}.

The combinatorial rigidity conjecture is meaningful for higher degree maps, and indeed it has 
been successfully established for a number of classes of maps through the program of Sullivan.  
Rational maps with all critical points periodic or pre-periodic are studied in \cite{DH85}. 
See \cite{LV98,KSS07} for real infinitely PL-renormalizable polynomials of higher degree, and  
\cite{CST13,CvS15} (and the references therein) for a broader result for real maps. 
The papers \cite{AKLS,KvS09,PT15} treat higher degree complex polynomials that are not 
PL-renormalizable. 
On the other hand, a near-parabolic renormalization scheme for uni-critical maps 
(i.e.\ maps with a single critical point of higher degree), similar to the one studied here, has been 
announced by Ch\'eritat in \cite{C14}.
The analysis of this paper may be carried out in that setting to obtain the corresponding results 
for the higher degree uni-critical maps. 

One may refer to the book by de Melo-van Strien \cite{dMvS93} for historical notes on early stages of the developments, 
and also for the real analysis tools that played a crucial role in the pioneering work of Sullivan on the subject. 
An alternative approach to the existence of the fixed point of doubling renormalization was given by Martens in \cite{Mar98}.  
A unified approach to the uniform contraction of polynomial-like renormalization for uni-singular maps is presented 
in \cite{AvLy11}. 
A number of renormalization schemes in low-dimensional dynamics have been studied in parallel to the one for holomorphic 
maps on plane-domains discussed here. 
However, the issue of the pre-compactness addressed here does not arise in those cases. 
One may refer to  
\cite{Lan82,EE86,dF92,Yam02,dFdM00,LeSw05,GudM13,KhKoSa14} \nocite{dF99} for critical circle maps; 
\cite{McM98} for linearizable maps of bounded type;
\cite{LeSw02,LeSw12}\nocite{LvS00} for critical circle covers;
\cite{GvST89,CLM05} for Henon maps; 
\cite{FMP06} for $C^r$ uni-modal maps. 

\renewcommand{\thethm}{\thesection.\arabic{thm}}

\tableofcontents

\section{Near-parabolic renormalization scheme}\label{S:Near-Parabolic}
\subsection{IS class of maps and their perturbations}
Consider the ellipse
\[E=\Big \{x+\B{i}y \in \BB{C}  \, \Big |\,  (\frac{x+0.18}{1.24})^2+(\frac{y}{1.04})^2 \leq 1\Big \}\]
and define the domain 
\[V=g(\hat{\BB{C}}\setminus E), \text{ where } g(z)= \frac{-4z}{(1+z)^2}.\]
The ellipse $E$ is contained in the ball $|z|< 2$, and thus, the ball $|z|< 8/9$ is contained in $V$.
Consider the cubic polynomial 
\[P(z)=z(1+z)^2, z\in \BB{C}.\] 
The polynomial $P$ has a fixed point at $0$ with multiplier $P'(0)=1$, and it has two critical 
points $-1\in \BB{C}\setminus V$ and $-1/3\in V$, where $P(-1)=0$ and $P(-1/3)=-4/27$. 
See Figure~\ref{F:domain-V} and \ref{F:covering-structure}. 

\begin{figure}\label{F:domain-V}
\begin{center}
  \begin{pspicture}(5.5,5) 
\epsfxsize=5cm
  \rput(2.5,2.5){\epsfbox{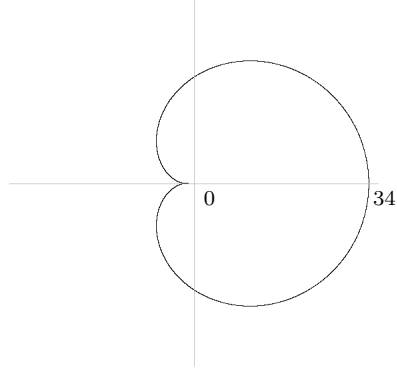}}
\rput(2.7,2.3){\tiny $0$}  
\rput(5,2.3){\tiny $34$}
  
\end{pspicture}
\caption{The domain $V$ is the region bounded by the back curve (resembling a cardioid). 
It contains the points $0$ and $-1/3$, but not $-1$.}
\end{center}
\end{figure}

Following \cite{IS06} we consider the class of maps 
\[\IS= \left\{f=P\circ \gf^{-1}:\gf(V)\to \BB{C} \,  \Big |
\begin{array}{l}
 \gf:V\to \BB{C} \text{ is univalent\footnotemark},\\
 \gf(0)=0, \gf'(0)=1, \tand \gf \\
 \text{has quasi-conformal extension onto } \BB{C}. 
\end{array}
\right\}.\]
\footnotetext{Univalent is a standard terminology used for one-to-one holomorphic maps.}
Every map in $\IS$ has a parabolic fixed point at $0$ and a unique critical point at $\gf^{-1}(-1/3)$ 
that is mapped to $-4/27$. 
Indeed, every element of $\IS$ has the same covering structure as the one of $P: V\to P(V)$.

We use the one-to-one correspondence between the class $\IS$ and the quasi-conformal mappings on 
$\BB{C}\setminus \overline{V}$ to define a metric on $\IS$. 
This corresponds to the Teichm\"uller metric on the Teichmuller space of $\BB{C}\setminus \ol{V}$.
One may refer to \cite{Leh76} for the definition of quasi-conformal mappings, and to 
\cite{GL00}, \cite{IT92}, and/or \cite{Leh87} for the theory of Teichm\"uller spaces. 
Recall that the \textit{dilatation quotient} of a quasi-conformal mapping $h$ is defined as 
\[\Dil (h)=\sup_{z\in \Dom h} \frac{|h_z|+|h_{\ol{z}}|}{|h_z|-|h_{\ol{z}}|}.\] 
The Teichm\"uller distance  between any two elements $f=P\circ \gf_f^{-1}$ and $g= P\circ \gf_g^{-1}$ 
in $\IS$ is defined as 
\begin{align*}
\Td(f,g) 
=\inf  &\Big\{\log \Dil (\hat{\gf}_g\circ \hat{\gf}^{-1}_f)\;\Big|%
\begin{array}{l} 
\hat{\gf}_f \tand \hat{\gf}_g \text{ are quasi-conformal extensions}\\
\text{of }\gf_f \tand \gf_g \text{ onto } \BB{C} \text{, respectively}.
\end{array}
\Big\}
.\end{align*}
It is known that the Teichm\"uller space of $\BB{C}\setminus \ol{V}$ equipped with the 
Teichm\"uller distance is a complete metric space, and so is $\IS$ equipped with $\Td$.

Let $R_a(z)= e^{2\pi i a} \cdot z $ denote the complex rotation about $0$ in $\BB{C}$. 
Given a set $A \ci \BB{C}$, define the class of maps 
\[\PC{A}=\{f \circ R_a: R_{-a}(\Dom (f)) \to \Dom (f) \mid f\in \IS, a\in A\}. 
\footnote{$\Dom(f)$ denotes the domain of definition of a given map $f$, and is always assumed to be 
an open set.}\]
For $r>0$, we define 
\begin{gather*}
A^+(r)=\{ \ga\in \BB{C} \mid  0< |\ga| \leq r , \Re \ga \geq |\Im \ga|\}, \\
A^-(r)=\{\ga \in \BB{C} \mid  0< |\ga| \leq r , \Re \ga \leq -|\Im \ga|\}, \\
A(r)=A^+(r)\cup A^-(r). 
\end{gather*}
See Figure~\ref{F:cantor-sets-of-rotations}.
We shall work on the class of maps $\PC{A(r)}$, for an appropriate constant $r$ that is 
determined in Section~\ref{SS:renormalization}.

Every $f\in \PC{A(\infty)}$ has a unique critical point, denoted by $\cp_f$. 
That is, 
\[f'(\cp_f)=0, \quad f(\cp_f)=-4/27.\] 
For our convenience, we normalize the quadratic polynomials into the form  
\[Q_\ga(z)=e^{2\pi \B{i} \ga} z+\frac{27}{16}e^{4\pi \B{i} \ga}z^2,\]
so that their critical values lie at $-4/27$ and $Q_\ga= Q_0 \circ R_\ga$. 
\begin{propo}\label{P:sigma-fixed-point}
There exist a simply connected neighborhood $W$ of $0$ bounded by a smooth curve, and 
a constant $r_1>0$ such that every map in $\PC{A(r_1)}$ has exactly two distinct fixed points in 
the closure of $W$.  
\end{propo} 

We shall prove the above proposition in Section~\ref{SS:basic-properties-lift-a}.
The non-zero fixed point of $f\in \PC{A(r_1)}$ contained in $W$ is denoted by $\gs_f$.
There are complex numbers $\ga(f)$ and $\gb(f)$ with real part in $(-1/2, 1/2]$ such that 
\[f'(0)=e^{2\pi \BF{i} \ga(f)} \tand f'(\gs_f)=e^{2\pi \BF{i} \gb(f)}.\] 
These values are related by the \textit{holomorphic index formula} 
\begin{equation}\label{E:holomorphic-index-formula}
\frac{1}{2\pi \B{i}}\int_{\partial W} \frac{1}{z-f(z)}\, \mathrm{d} z
=\frac{1}{1-e^{2\pi \B{i} \ga(f)}}+ \frac{1}{1-e^{2\pi \B{i} \gb(f)}}.
\end{equation}

We consider the topology of uniform convergence on compact sets on the space of holomorphic 
maps $g: \Dom (g) \to \BB{C}$, where $\Dom (g)$ is an open subset of $\BB{C}$. 
A basis for this topology is defined by 
\[N(h;K,\gep)
=\Big \{g:\Dom (g) \to \BB{C} \, \Big | \,  K\subset \Dom (g) \tand \sup_{z\in K}
|g(z)-h(z)|<\gep\Big \},\]
where $h : \Dom (h) \mapsto \BB{C}$ is a holomorphic map, $K \subset \Dom (h)$ is compact,  
and an $\gep>0$. 
In this topology, a sequence $h_n : \Dom (h_n) \mapsto \BB{C}$ converges to $h$ provided 
$h_n$ is contained in any given neighborhood of $h$ defined as above, for large enough $n$. 
Note that the maps $h_n$ are not necessarily defined on the same domains. 
We note that the convergence with respect to $\Td$ on $\IS$ implies the uniform convergence 
on compact sets.

Let $f_\gl:\Dom(f_\gl)\subset \BB{C}\to \BB{C}$ be a family of holomorphic maps parameterized by 
$\gl$ in a finite dimensional complex manifold $\gL$. 
We say that the family $f_\gl$ is a \textit{holomorphic family} of maps, if for every $\gl_0 \in \gL$ 
and every $z_0 \in \Dom f_{\gl_0}$, the map $(z,\gl) \mapsto f_\gl(z)$ is defined and holomorphic 
in $z$ and $\gl$, for $(z,\gl)$ sufficiently close to $(z_0, \gl_0)$. 
Let $\gU: X \to Y$ be a mapping where $X$ and $Y$ are some classes of holomorphic maps. 
We say that the mapping $f \mapsto \gU(f)$ has \textit{holomorphic dependence} on $f$, 
if for every holomorphic family of maps $f_\gl$ in $X$, the family $\gU(f_\gl)$ is a holomorphic 
family of maps. 
 
\begin{propo}\label{P:Fatou-coordinates}
There exists $r_2>0$ such that for every $f\colon V_f \to \BB{C}$ in $\PC{A^+(r_2)}$ there exist 
a domain $\C{P}_f \subset \Dom(f)$ and a univalent map $\Phi_f\colon \C{P}_f \to \BB{C}$ 
satisfying the following properties:
\begin{itemize}  
\item[a)] The domain $\C{P}_f$ is bounded by piecewise smooth curves and is compactly contained in $V_f$. 
Moreover, $\C{P}_f$ contains $\cp_f$, $0$, and $\sigma_f$ on its boundary.
\item[b)] $\Im \Phi_f(z) \to +\infty$ when $z\in \C{P}_f\to 0$, and $\Im \Phi_f(z)\to-\infty$ 
when $z \in \C{P}_f \to \sigma_f$.
\item[c)] The image of $\gF_f$ contains the vertical strip $\Re w \in (0,2)$.
\item[d)] The map $\Phi_f$ satisfies
\[\Phi_f(f(z))=\Phi_f(z)+1, \text{ whenever $z$ and $f(z)$ belong to $\C{P}_f$}.\] 
\item[e)] The map $\Phi_f$ is uniquely determined by the above conditions together with the 
normalization $\Phi_f(\cp_f)=0$. 
Moreover, the normalized map $f\mapsto \Phi_f$ has holomorphic dependence on $f$.
\end{itemize}
\end{propo}

When $f=Q_\ga:\BB{C}\to\BB{C}$, with $\ga\in A^+(r_2)$, the existence of a domain $\C{P}_f$ and 
a coordinate $\gF_f: \C{P}_f  \to \BB{C}$ satisfying the properties in the above proposition is 
rather classical. 
Indeed, these are among the basic tools in complex dynamics for over a century now. 
However, their existence for the maps in the class $A(r) \ltimes \IS$ is proved in \cite{IS06}. 
The univalent map $\gF_f$ with the above properties is called the \textit{Fatou coordinate} of $f$ 
on $\C{P}_f$, see Figure~\ref{F:petal}. 

\begin{figure}[ht]
\begin{center}
\begin{pspicture}(0,.3)(9,8.1)
\rput(4,4){\includegraphics[width=8cm]{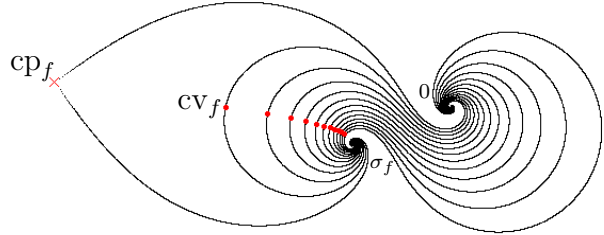}}
\rput(.3,4.5){$\cp_f$}
\rput(2.5,4){$\cv_f$}
\rput(5.45,4.2){\tiny $0$}
\rput(4.9,3.2){\tiny $\sigma_f$}
\end{pspicture}
\end{center}
\caption{The figure shows the pre-images of the vertical lines, with integer real parts,  
under the Fatou coordinate. 
The curves land at $0$ and $\gs_f$.
Here, $\Im \ga \neq 0$, and the curves spiral about these points at well-defined speeds.
The red cross shows the critical point of $f$, while the small red disks show a few iterates of the critical point.} 
\label{F:petal}
\end{figure}

\begin{propo}\label{P:wide-petals}
There are constants $r_3\in (0, r_2)$ and $\B{k}$ such that for all $f\in \PC{A^+(r_3)}$, 
or $f=Q_\ga$ with $\ga\in A^+(r_3)$, the domain $\C{P}_f$ in Proposition~\ref{P:Fatou-coordinates} 
may be chosen (wide enough) to satisfy the additional property
\[\Phi_f(\C{P}_f)=\{w \in \BB{C} \mid 0 < \Re(w) < \Re \frac{1}{\ga(f)} -\B{k}\}.\]
\end{propo}

The above proposition is proved in Sections~\ref{SS:petal-geometry}.

In \cite{Ch10-I} it is proved that when $\ga$ is real, $\gF_{f}^{-1}$ of every vertical line in the image of 
$\gF_f$ is a curve that lands at $0$ and $\gs_f$ at some well-defined angle. 
Thats is, there is a tangent line to the curves at $0$ and $\gs_f$. 
However, this is not the case when $\Im \ga\neq 0$.  
The pre-images of the vertical lines spiral about $0$ and $\gs_f$, and the corresponding speeds of spirals depend on 
on $\Im \ga$ and $\Im \gb$, respectively. 
The precise statement is sated in the next proposition.
 
\begin{propo}\label{P:bounded-spirals}
There exists a constant $k'$ such that for all $f\in \PC{A^+(r_3)}$, or $f=Q_\ga$ with $\ga\in A^+(r_3)$, 
there exists a continuous branch of argument defined on $\C{P}_f$ such that 
\begin{itemize}
\item[a)] For all $\gx_1$ in $(0, \Re \frac{1}{\ga(f)} -\B{k})$ and $\gx_2 \geq 0$, we have 
\begin{align*}
\lim_{\gx_2 \to +\infty} \big (\arg \gF_f^{-1}(\gx_1+\B{i} \gx_2)& + 2 \pi \gx_2 \Im \ga \big) 
= \arg \gs_f + 2 \pi \gx_1 \Re \ga + c_f, 
\end{align*}
where $c_f$ is a real constant depending on $f$ with $|c_f| \leq k' (1- \log |\ga|)$. 
\item[b)] For all $\gx_1$ in $(0, \Re \frac{1}{\ga(f)} -\B{k})$ and $\gx_2 \leq 0$, we have 
\begin{align*}
\lim_{\gx_2 \to +\infty} \big (\arg (\gF_f^{-1}(\gx_1+\B{i} \gx_2)-\gs_f) & + 2 \pi \gx_2 \Im \gb \big) 
= \arg \gs_f + 2 \pi \gx_1 \Re \gb + c'_f, 
\end{align*}
where $c'_f$ is a real constant depending on $f$ with $|c'_f| \leq k' (1- \log |\ga|)$. 
\end{itemize}
\end{propo}

Part a of the above proposition is proved in Section~\ref{SS:petal-geometry}, while its part b is 
proved in Section~\ref{SS:2-1-derivative-b}.
Our analysis also allows one to control the dependence of the spirals on the non-linearity of the map $f$.  

\begin{rem}
When $f \in A(r) \ltimes \IS$ tends to a map $f_0\in \IS$, the fixed point $\gs_f$ tends to $0$, and becomes a 
parabolic fixed point.  
Although it is not used in this paper, it may be useful to point out that as $f$ tends to $f_0\in \IS$, 
appropriately normalized Fatou coordinates $\gF_f$ tend to some conformal mappings, called attracting and 
repelling Fatou coordinates, that still satisfy the remarkable functional equation in 
Proposition~\ref{P:Fatou-coordinates}-d). 
One may refer to \cite{Sh00} for further details on this. 
\end{rem}


\subsection{Top and bottom near-parabolic renormalizations}\label{SS:renormalization}
Let $f\colon V_f \to \BB{C}$ either be in $\PC{A^+(r_2)}$ or be the quadratic polynomial $Q_\ga$ with $\ga\in A^+(r_2)$. 
Let $\gF_f$ be the Fatou coordinate of $f$ introduced in the previous section. 
Define the sets 
\begin{equation}\label{E:sector-def}
\begin{gathered}
A_f=\{z\in \C{P}_f : 1/2 \leq \Re(\Phi_f(z)) \leq 3/2 \: , \: 2 \leq  \Im \Phi_f(z) \}, \\ 
C_f=\{z\in \C{P}_f : 1/2 \leq \Re(\Phi_f(z)) \leq 3/2 \: ,\: -2 \leq \Im \Phi_f(z) \leq 2 \},\\
B_f=\{z\in \C{P}_f : 1/2 \leq \Re(\Phi_f(z)) \leq 3/2 \: , \:    \Im \Phi_f(z)\leq -2 \}.
\end{gathered}
\end{equation}
By Proposition~\ref{P:Fatou-coordinates}, $\gF_f$ maps the critical value of $f$, denoted by $\cv_f$, to 
$+1$, and hence $\cv_f$ belongs to $\inte(C_f)$ \footnote{The notation $\inte(C)$ denotes the (topological) 
interior of a given set $C$.}. 
Moreover, $0\in \partial A_f$ and $\gs_f \in \partial B_f$. 

\begin{propo}\label{P:renormalization-top}
For every $f\in \PC{A^+(r_3)}$ or $f=Q_\ga$ with $\ga\in A^+(r_3)$, there is a positive integer $k_f^t$ 
satisfying the following: 
\begin{itemize}
\item[a)] For every integer $k$, with $0\leq k \leq k_f^t$, there exists a unique connected component 
of $f^{-k}(A_f)$ which is compactly contained in $\Dom (f)$  and contains $0$ on its boundary. 
We denote this component by $A_f^{-k}$. 
\item[b)] For every integer $k$, with $0\leq k \leq k_f^t$, there exists a unique connected component of 
$f^{-k}(C_f)$ which has non-empty intersection with $A_f^{-k}$, and is compactly contained in 
$\Dom (f)$. 
This component is denoted by $C_{f,t}^{-k}$. 
\item[c)] We have 
\[A_f^{-k_f^t}, C_{f,t}^{-k_f^t} 
\ci \big \{z\in\C{P}_f \mid  1/2< \Re \Phi_f(z) < \Re \frac{1}{\ga(f)} -\B{k}\big \}.\] 
\item[d)] The map $f:C_{f,t}^{-k}\to C_{f,t}^{-k+1}$, for $2\leq k \leq k_f^t$, and 
$f: A_f^{-k}\to A_f^{-k+1}$, for $1\leq k \leq k_f^t$, are univalent. 
On the other hand, the map $f: C_{f,t}^{-1}\to C_{f}$ is a proper branched covering of degree two.
\end{itemize}
\end{propo}

\begin{figure}
\begin{center}
\begin{pspicture}(9,9) 
\epsfxsize=9cm
  \rput(4.5,4.5){\epsfbox{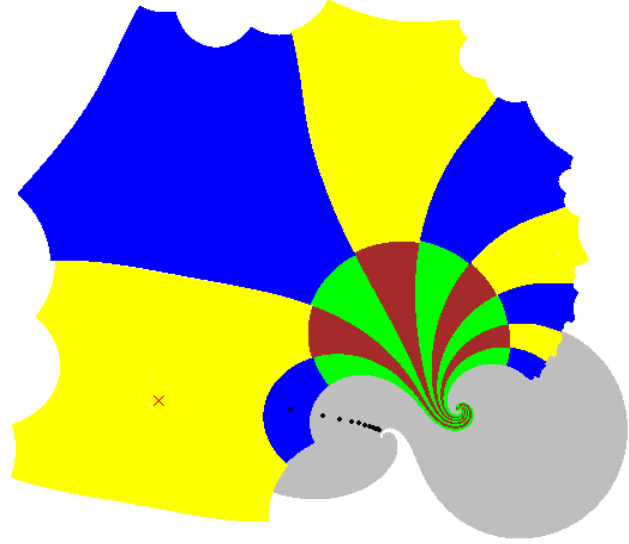}}
 \rput(2.8,3.1){\tiny \textcolor{red}{$\times$}} 
\end{pspicture}
\caption{A schematic presentation of the regions associated to the top renormalization of $f$. 
The alternating green and brown shades denote the sets $A_f^{-j}$, and the alternating blue and yellow
shades are the sets $C_f^{-j}$. 
The grey region is the petal $\C{P}_f$ drawn in Fig~\ref{F:petal}. 
Here $f=Q_\ga$ with $\ga= 0.01-\B{i}0.02$. 
The red cross is the critical point and the black dots show the orbit of the critical value.}
\label{F:renormalization-bottom}
\end{center}
\end{figure}

\begin{propo}\label{P:renormalization-bottom}
For every $f\in \PC{A^+(r_3)}$ or $f=Q_\ga$ with $\ga\in A^+(r_3)$, there is a positive integer 
$k_{f, b}$ satisfying the following:
\begin{itemize}
\item[a)] For every integer $k$, with $0\leq k \leq k_f^b$, there exists a unique connected component of $f^{-k}(B_f)$ 
which is compactly contained in $\Dom f$, and contains $\gs_f$ on its boundary. 
We denote this component by $B_f^{-k}$. 
\item[b)] For every integer $k$, with $0\leq k \leq k_f^b$, there exists a unique connected component of 
$f^{-k}(C_f)$ which has non-empty intersection with $B_f^{-k}$, and is compactly contained in $\Dom f$. 
This component is denoted by $C_{f,b}^{-k}$. 
\item[c)] We have 
\[B_f^{-k_f^b}, C_{f,b}^{-k_f^b}
\ci  \big \{z\in\C{P}_f \mid  1/2< \Re \Phi_f(z) <\frac{1}{\Re \ga(f)} -\B{k}\big \}.\] 
\item[d)] The map $f:C_{f,b}^{-k}\to C_{f,b}^{-k+1}$, for $2\leq k \leq k_f^b$, and 
$f: B_f^{-k}\to B_f^{-k+1}$, for $1\leq k \leq k_f^b$, are univalent. 
On the other hand, the map $f:C_{f,b}^{-1}\to C_{f}$ is a proper branched covering of degree two.
\end{itemize}
\end{propo}

\begin{figure}
\begin{center}
\begin{pspicture}(9,9)
\epsfxsize=9cm
  \rput(4.5,4.5){\epsfbox{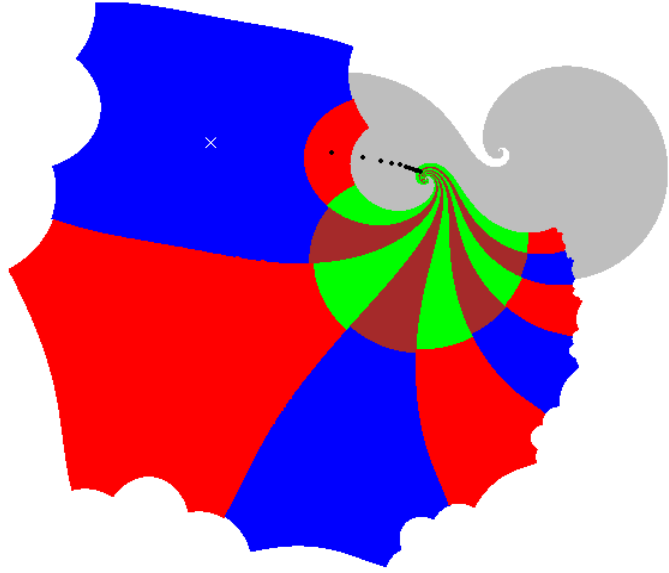}}
 \rput(2.8,6.1){\tiny \textcolor{white}{$\times$}} 
\end{pspicture}
\caption{A schematic presentation of the regions associated to the bottom renormalization of $f$. 
The alternating green and brown shades show the sets $B_f^{-j}$, and the alternating red and blue 
shades show the sets $C_f^{-j}$.}
\label{F:renormalization-top}
\end{center}
\end{figure}

The above two propositions are appropriately adjusted and reformulated versions of several statements 
that appear in Section 5.A in \cite{IS06}. 
See in particular Propositions 5.6 and 5.7 in that paper. 
Note that here we are working with the value $+2$ in place of the parameter $\gh$ in that paper. 

Let $k_f^t$ and $k_f^b$ be the smallest positive integers satisfying the above propositions. 

\begin{propo}\label{P:bounded-iterates-remaining}
There exists a constant $\B{k}''$ such that for every $f\in \PC{A^+(r_3)}$, or $f=Q_\ga$ with $\ga\in A^+(r_3)$, 
we have 
\begin{itemize}
\item[a)] $k_f^t \leq k''$,
\item[b)] $k_f^b\leq k''$.
\end{itemize}
\end{propo}

Part a of the above proposition is proved in Section~\ref{SS:petal-geometry} and its Part b is proved in 
Section~\ref{SS:2-1-derivative-b}.  

Define
\[S_f^t=A_{f}^{-k_f^t} \cup C_{f,t}^{-k_f^t}, S_f^b=B_{f}^{-k_f^b} \cup C_{f,b}^{-k_f^b}.\]
See Figures~\ref{F:renormalization-top} and \ref{F:renormalization-bottom}.

Consider the induced maps 
\begin{gather}
E_f^t=\Phi_f \circ f\co{k_f^t} \circ \Phi_f^{-1}:\Phi_f(S_f^t) \to \BB{C}, \label{E:horn-top}\\
E_f^b=\Phi_f \circ f\co{k_f^b} \circ \Phi_f^{-1}:\Phi_f(S_f^b)\to \BB{C}, \label{E:horn-bottom}
\end{gather}
By the functional equation $\gF_f(f(z))=\gF_f(z)+1$, we have $E_f^t(w+1)= E_f^t(w)+1$ whenever 
both $w$ and $w+1$ are contained in $\Phi_f(S_f^t)$. 
Similarly, $E_f^b$ also commutes with the translation by one on the boundary of $\gF_{f}(S_f^b)$. 

Let us define the covering maps 
\begin{equation}\label{E:ex} 
\ex^t(w)=\frac{-4}{27}e^{2\pi \B{i}w},  \ex^b(w)=\frac{-4}{27}e^{-2\pi \B{i}w}.   
\end{equation}
The map $E_f^t :\Phi_f(S_f^t) \to \BB{C}$ projects via $\ex^t$ to a well-defined holomorphic map 
defined on a set containing a punctured neighborhood of $0$. 
We denote this map by $\operatorname{\mathcal{R}_{\scriptscriptstyle NP-t}}(f)$. 
Similarly, $E_f^b :\Phi_f(S_f^b)\to \BB{C}$ projects via $\ex^b$ to a well-defined holomorphic map
defined on a set containing a punctured neighborhood of $0$. 
This map is denoted by $\operatorname{\mathcal{R}_{\scriptscriptstyle NP-b}}(f)$. 
Both of these maps have a removable singularity at $0$ with asymptotic expansions
\[\nprt{1}(f)(z)=e^{-2 \pi \B{i} /\ga(f)}z+ O(z^2), \; 
\ \nprb{1}(f)(z)=e^{-2 \pi \B{i} /\gb(f)}z+ O(z^2),\]
near $0$, where $f'(0)= e^{2\pi \B{i} \ga(f)}$ and $f'(\gs_f)= e^{2\pi \B{i} \gb(f)}$.
The above asymptotic expansions are obtained from comparing $f$ near $0$ and $\gs_f$ to the linear 
maps $z\mapsto e^{2\pi \B{i}\ga(f)} z$ and $z \mapsto \gs_f+e^{2\pi \B{i}\gb(f)}(z-\gs_f)$, respectively. 

By Propositions~\ref{P:renormalization-top} and \ref{P:renormalization-bottom}, each of the maps 
$E_f^t$ and $E_f^b$ has a unique critical value. As $\gF_f(\cv_f)=1$, the critical values of  $E_f^t$ 
and $E_f^b$ lie at $+1$. 
On the other hand, as $\ex^t(+1)=\ex^b(+1)= -4/27$, each of 
$\nprt{1}(f)$ and $\nprb{1}(f)$ must have a unique critical value at $-4/27$.

The main result of \cite{IS06} is formulated in the next theorem.

\begin{thm}[Inou-Shishikura]\label{T:Ino-Shi2} 
There exist a Jordan domain $U\supset \ol{V}$ satisfying the following. 
For all $f \in \PC{A^+(r_3)}$, or $f=Q_\ga$ with $\ga\in A^+(r_3)$, there are appropriate 
restrictions (to smaller domains about $0$) of the maps $\nprt{1}(f)$ and $\nprb{1}(f)$ 
that belong to the classes $\PC{\{\frac{-1}{\ga(f)}\}}$ and $\PC{\{ \frac{-1}{\gb(f)}\}}$, respectively. 
That is, there exist quasi-conformal homeomorphisms $\gp, \gf: \BB{C}\to \BB{C}$ that are 
holomorphic on $V$, $\gp(0)=\gf(0)=0$, $\gp'(0)=\gf'(0)=1$, and  
\begin{gather*}
\nprt{1}(f)(z)
=P\circ \gp^{-1}(e^{-2\pi \B{i}/\ga(f)} \cdot z), \forall z\in e^{2\pi \B{i}/\ga(f)}\cdot\gp(V) \\
\nprb{1}(f)(z)
=P\circ \gf^{-1}(e^{-2\pi \B{i}/\gb(f)} \cdot z), \forall z\in e^{2\pi \B{i}/\gb(f)}\cdot\gf(V) .
\end{gather*}
Moreover, when $f \in \PC{A^+(r_3)}$, both $\gp: V\to \BB{C}$ and $\gf: V\to \BB{C}$ extend 
to univalent maps on $U$. 
\end{thm}

See the figure below for the covering structure of the polynomial $P$ on the set $V$. 
\begin{figure}[ht]
\begin{center}
  \begin{pspicture}(8,3.2)
  \rput(4.5,1.6){\epsfbox{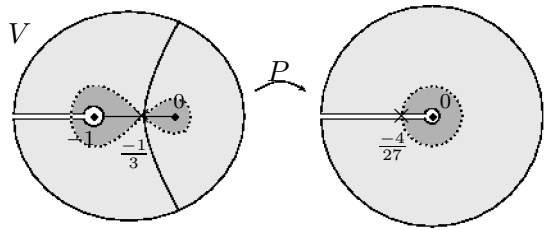}}
      \rput(6.12,1.6){{\small $\times$}}
      \rput(6,1.2){{\tiny $\frac{-4}{27}$}}
      \rput(6.7,1.8){{\tiny $0$}} 
      \rput(2.7,1.6){{\small $\times$}}
      \rput(2.6,1.1){{\tiny $\frac{-1}{3}$}}
      \rput(3.2,1.8){{\tiny $0$}}
      \rput(1.9,1.3){{\tiny $-1$}}
      \rput(4.5,2.2){{\small $P$}}
      \rput(1.1,2.7){{\small $V$}} 
\end{pspicture}
\caption{A schematic presentation of the polynomial $P$; its domain and its range. 
Similar colors and line styles are mapped on one another.}
\label{F:covering-structure}
\end{center}
\end{figure}


\begin{rem}
The renormalizations $\nprt{1}(f)$ and $\nprb{1}(f)$ are not obtained from the return maps (iterates of $f$) 
to a region, in contrast to other notions of renormalization in holomorphic dynamics, such as the 
PL-renormalization~\cite{DH85} or the sector renormalization of Yoccoz~\cite{Yoc95}. 
Near $0$ or $\gs_f$, these renormalizations may be interpreted as return maps since  
$E_f^t$ and all its integer translations project to the same map $\nprt{1}(f)$. 
For $w\in \gF_f(S_f^t)$ with $\Im w$ large enough, there is $i_w \in \BB{N}$ such that 
$E_f^t(w)+i_w \in  \gF_f(S_f^t)$; hence a return map. 
But, this may not happen for every $w$ in $\gF_f(S_f^t)$. 
For example, when $|\ga|$ is small with $\arg \ga=-\pi/4$, the $\gs_f$ is attracting and it may  
attract the orbit of the critical point.
Then, the orbit of the critical point may not visit $S_f^t$, (and ``go around'' $0$ to return back to $\C{P}_f$. 
However, this does not contradict the above theorem and the top renormalization is still defined.
Here, $\nprt{1}(f)$ has a critical value, but it does not belong to the domain of $\nprt{1}(f)$. 
For such values of $\ga$, $|e^{-2\pi \B{i}/\ga}|$ is large, and hence by 
Theorem~\ref{T:Ino-Shi2}, $\Dom \nprt{1}(f)$ may be small and not contain the critical value at $-4/27$. 
As we shall see in Sections~\ref{SS:pairs} and \ref{SS:multi-complex-rotation}, in the interesting cases 
where both multipliers at $0$ and $\gs_f$ are repelling, $\ga$ belongs to a substantially smaller region and 
this scenario does not occur. 
\end{rem}

\begin{Def}\label{D:extended-renormalization}
For every $f\in \PC{A^-(r_3)}$, the conjugate map $s \circ f \circ s$, where $s(z)=\ol{z}$ denotes the 
complex conjugation, belongs to $\PC{A^+(r_3)}$. 
We may extend the above definitions of renormalizations onto $\PC{A^-(r_3)}$ by 
letting 
\[\nprt{1}(f)= \nprt{1}(s\circ f \circ s),  \; \nprb{1}(f)= \nprb{1}(s \circ f \circ s), 
\quad \forall f\in \PC{A^-(r_3)}.\]  
In particular, the Fatou coordinates are also defined for maps in $\PC{A^-(r_3)}$. 
Similarly, one defines the Fatou coordinates and the renormalizations for $Q_\ga$ with 
$\ga \in A^-(r_3)$. 
\end{Def}

The following proposition is a consequence of the holomorphic dependence of the Fatou coordinate 
on the map, Proposition~\ref{P:Fatou-coordinates}-e), and the definitions of the operators $\nprt{1}$ 
and $\nprb{1}$.

\begin{propo}\label{P:holomorphic-dependence} 
The operators $f \mapsto \nprt{1}(f)$ and $f\mapsto \nprb{1}(f)$ have holomorphic dependence on 
$f \in \PC{A(r_3)}$. 
Similarly, $\ga \mapsto \nprt{1}(Q_\ga)$ and $\ga \mapsto \nprb{1}(Q_\ga)$ 
are holomorphic families of maps, parametrized on $A(r_3)$. 
\end{propo}

The restrictions of the maps $\nprt{1}(f)$ and $\nprb{1}(f)$ to the smaller domain such that they 
belong to $\PC{A(\infty)}$, are called the \textit{top} and \textit{bottom}
 \textit{near-parabolic renormalization} of $f$, respectively.
We use the notation $\nprt{1}(f)$ and $\nprb{1}(f)$ to denote these (domain restricted) maps. 
Note that in this definition, $\nprt{1}(f)$ and $\nprb{1}(f)$ have extension onto the larger domain $U$, 
by the above theorem. 
Also, note that although the renormalization of a map has extension onto a larger domain ($U$), 
the renormalizations are defined only using the iterates of the map on the smaller domain ($V$).
\section{Analytic properties of the near-parabolic renormalizations}\label{S:analytic-properties}
\subsection{K-horizontal curves}\label{SS:hyperbolicity}
In this section we study the dependence of the renormalizations $f \mapsto \nprt{1}(f)$ and 
$f \mapsto \nprb{1}(f)$ on the linearity of $f$ (that is, $\ga(f)$), and the non-linearity of $f$ 
(that is, the higher order terms of $f$). 
Let us introduce some notations in order to simplify the statements that will follow. 

For $h\in \IS$ and $\ga\in \BB{C}$ we use the notation $\ga \ltimes h$ to represent the map 
defined as 
\begin{equation}\label{E:notation-ltimes}
(\ga \ltimes h) (z)= h(e^{2\pi \B{i} \ga} \cdot z), z\in e^{-2\pi \B{i} \ga} \cdot \Dom (h).
\end{equation} 
That is, $\ga\ltimes h$ defines a product structure on the set of maps $A \ltimes \IS$, for $A\ci \BB{C}$. 
Note that $\ga \ltimes Q_0= Q_\ga$. 
With this notation, we write  
\[\nprt{1}(\ga \ltimes h)=(\hat{\ga}(\ga, h) \ltimes \hat{h}(\ga, h)),\]  
\[\nprb{1}(\ga \ltimes h)=(\check{\ga}(\ga, h) \ltimes \check{h}(\ga, h)),\] 
where $\hat{\ga}(\ga, h)$ and $\check{\ga}(\ga, h)$ are complex numbers, which depend on 
$\ga$ and $h$, as well as $\hat{h}(\ga, h)$ and $\check{h}(\ga, h)$ are elements of $\IS$, which
depend on $\ga$ and $h$. 
Indeed, by the definitions of the renormalizations, $\hat{\ga}(\ga,h) = -1/\ga$ but 
$\check{\ga}$, $\hat{h}$ and $\check{h}$ depend on both $\ga$ and $h$.
Recall that for a map $\ga \ltimes h$, $\gs_{\ga\ltimes h}$ denotes the preferred non-zero fixed point of 
this map introduced in the previous section and $\gb=\gb(\ga,h)$ is a complex number with 
$(\ga \ltimes h)' (\gs_{\ga\ltimes h})= e^{2\pi \B{i} \gb}$.

Let $s \mapsto \gU(s)=(\ga(s) \ltimes h(s))$ be a map defined for $s$ in a connected set $\gD\ci \BB{C}$, 
and with values in the set $\PC{A(+\infty)}$.
For $k>0$, we say that $\gU$ is $k$-\textit{horizontal}, if $\gU$ is continuous on $\gD$, and for all 
$s_1, s_2 \in \gD$ we have 
\[\Td(h(s_1), h(s_2))\leq k |\ga(s_1)-\ga(s_2)|.\]
We call the image of any such curve a $k$-\textit{horizontal curve}. 
\begin{propo}\label{P:k-horizontal}
There are constants $r_4\in (0,r_3]$ and $k_1>0$ such that 
\begin{itemize}
\item[a)] for every $k_1$-horizontal curve $\gU$ in $\PC{A(r_4)}$, $\nprt{1}(\gU)$ and 
$\nprb{1}(\gU)$ are $k_1$-horizontal curves in $A(+\infty) \ltimes \IS$. 
\item[b)] the curves $\ga \mapsto \nprt{1}(Q_\ga)$ and $\ga \mapsto \nprb{1}(Q_\ga)$, for 
$\ga \in A(r_4)$, are $k_1$-horizontal curves in $A(+\infty) \ltimes \IS$. 
\end{itemize}
\end{propo}

It follows from the proof that by making $k_1$ large enough and $r_4$ small enough, we may make sure that the 
image of a $k_1$-horizontal curve under the renormalizations are $k_1/2$-horizontal. 
That means, the renormalization operators map the cone field of $k_1$-horizontal curves passing through a point in 
$\PC{A(r_4)}$ well inside the cone field of $k_1$-horizontal curves passing through the image point.   
In order to prove the above proposition, we need to control the dependence of the linearities and the 
non-linearities of $\nprt{1}$ and $\nprb{1}$ on $\ga$ and $h$. 
These are formulated in the following four propositions. 
Recall the constant $r_3$ obtained in Section~\ref{S:Near-Parabolic}. 

\begin{propo}\label{P:2-1-derivative}
There exists a constant $c_{2,1}$ such that for all $h$ in $\IS \cup \{Q_0\}$, as well as all $\ga_1$ and 
$\ga_2$ in $A^+(r_3)$ or all $\ga_1$ and $\ga_2$ in $A^-(r_3)$, we have  
\begin{itemize}
\item[a)] \[\Td(\hat{h}(\ga_1 \ltimes h),\hat{h}(\ga_2 \ltimes h)) \leq c_{2,1} |\ga_1-\ga_2|,\]
\item[b)] \[\Td(\check{h}(\ga_1 \ltimes h),\check{h}(\ga_2 \ltimes h)) \leq c_{2,1} |\ga_1-\ga_2|.\]
\end{itemize}
\end{propo}

Note that the above proposition is stated in the Introduction as Theorem~\ref{T:hyperbolicity}.
The Lipschitz property of $\ga \mapsto \hat{h}(\ga, h)$ on the intervals $(-\gr, 0) \cup (0, \gr)$ is proved 
in \cite{CC13}, and a similar proof is extended to the complex neighborhood $A(\gr)$ here. 
In that proof, it is crucial that $\nprt{1}$ is fiber preserving, that is, $\hat{\ga}(\ga\ltimes h)$ 
depends only on $\ga$.
On the other hand, $\nprb{1}$ is not fiber preserving and the proof of Theorem~\ref{T:hyperbolicity} 
for $\check{h}(\ga,h)$ involves further analysis. 
Also, when $\ga$ is real, we did not need to deal with the spiraling effects of the Fatou coordinates 
in \ref{P:bounded-spirals}.

\begin{propo}\label{P:2-2-derivative}
There exists a constant $c_{2,2}\in (0,1]$ such that for all $\ga$ in $A(r_3)$ as well as all 
$h_1$ and $h_2$ in $\IS$ we have 
\begin{gather*}
\Td(\hat{h}(\ga_1, h_1),\hat{h}(\ga_1,h_2)) \leq c_{2,2} \Td(h_1, h_2), \\
\Td(\check{h}(\ga_1, h_1),\check{h}(\ga_1, h_2)) \leq c_{2,2} \Td(h_1, h_2), 
\end{gather*}
\end{propo}

For a fixed $h\in \IS \cup \{Q_0\}$, $\ga \mapsto \check{\ga}(\ga \ltimes h)$ is a holomorphic 
mapping from $A(r_3)$ into the complex plane. 
In particular, the partial derivative of this map with respect to $\ga$ exists at every 
$h\in \IS \cup \{Q_0\}$ and $\ga$ in $A(r_3)$. 

\begin{propo}\label{P:1-1-derivative}
There exists a constant $c_{1,1}> 0$ such that for all $h$ in $\IS \cup \{Q_0\}$, and all 
$\ga_1$ and $\ga_2$ in $A^+(r_1)$ or all $\ga_1$ and $\ga_2$ in $A^-(r_1)$, we have 
\[\frac{1}{c_{1,1}|\ga_1 \ga_2|} |\ga_1 - \ga_2|
\leq \Big | \frac{1}{\gb(\ga_1\ltimes h)} - \frac{1}{\gb(\ga_2 \ltimes h)} \Big | 
\leq \frac{c_{1,1}}{|\ga_1 \ga_2|} |\ga_1 - \ga_2|.\]
\end{propo}

\begin{propo}\label{P:1-2-derivative}
There exists a constant $c_{1,2}>0$ such that for all $\ga$ in $A(r_3)$ as well as all $h_1$ and $h_2$ 
in $ \IS$, we have 
\[|\check{\ga}(\ga \ltimes h_1)-\check{\ga}(\ga \ltimes h_2)| \leq c_{1,2} \Td(h_1, h_2).\]
\end{propo}

\begin{rem}
By the definitions of $\nprt{1}$ and $\nprb{1}$, we only need to prove the 
Propositions~\ref{P:2-1-derivative}-\ref{P:1-2-derivative} for complex rotations in $A^+(r_3)$. 
The statements on $A^-(r_3)$ follow from the ones on $A^-(r_3)$. 
Thus, with in the rest of this section we assume that all complex rotations have positive real part, 
unless otherwise stated.
\end{rem}

\begin{proof}[Proof of Proposition~\ref{P:k-horizontal} assuming 
Propositions~\ref{P:2-1-derivative}, \ref{P:2-2-derivative}, \ref{P:1-1-derivative}, and \ref{P:1-2-derivative}] 
\hfill

Define $k_1= c_{2,1}/3$,  and choose $r_4>0$ such that  
\[12 + c_{1,2}c_{2,1} \leq \frac{3}{c_{1,1}} \frac{1}{r_4^2}.\]

Let $\gU:\gD \to \PC{A(r_3)}$ be a $k_1$-horizontal curve defined on a connected set $\gD$. 
Fix two (distinct) points $\ga_1\ltimes h_1$ and $\ga_2\ltimes h_2$ on the image of the curve $\gU$, 
and consider the third point $\ga_1\ltimes h_2$ in $\PC{A(r_3)}$. 
Consider the images of these points under $\nprt{1}$ and $\nprb{1}$, represented by  
\begin{align*}
\hat{\ga}_1\ltimes \hat{h}_1= 
\nprt{1}(\ga_1\ltimes h_1), \qq  &  \check{\ga}_1\ltimes \check{h}_1= \nprb{1}(\ga_1\ltimes h_1) \\
\hat{\ga}_2 \ltimes \hat{h}_2= 
\nprt{1}(\ga_2\ltimes h_2), \qq  &  \check{\ga}_2\ltimes \check{h}_2= \nprb{1}(\ga_2\ltimes h_2) \\
\hat{\ga}_3 \ltimes \hat{h}_3= 
\nprt{1}(\ga_1\ltimes h_2), \qq &  \check{\ga}_3\ltimes \check{h}_3= \nprb{1}(\ga_1\ltimes h_2)
\end{align*}

First we deal with the top renormalization. 
For the non-linearities of the maps we have 
\begin{equation}\label{E:bound-on-nonlinearities-top}
\begin{aligned}
\Td(\hat{h}_1, \hat{h}_2) 
& \leq \Td(\hat{h}_1, \hat{h}_3) +\Td (\hat{h}_3, \hat{h}_2) 
&& (\text{Triangle Inequality}) \\
& \leq c_{2,2} \Td(h_1, h_2) + c_{2,1} |\ga_1-\ga_2|, 
&& (\text{Propositions~\ref{P:2-2-derivative} and \ref{P:2-1-derivative}})  \\
& \leq c_{2,2} k_1 |\ga_1 - \ga_2| + c_{2,1} |\ga_1 -\ga_2|  
&& (\text{$\gU$ is $k_1$-horizontal}) \\
& = (c_{2,2} k_1 + c_{2,1}) |\ga_1 - \ga_2|. 
&& 
\end{aligned}
\end{equation}
On the other hand, as $\hat{\ga}_1= -1/\ga_1$ and $\hat{\ga}_2= -1/\ga_2$, we obtain  
\begin{align*}
|\hat{\ga}_1-\hat{\ga}_2| = \Big | \frac{1}{\ga_1 \ga_2} \Big | |\ga_1 - \ga_2|.  
\end{align*}
Recall that $c_{2,2}\leq 1$, $|\ga_1|\leq 1/2$, and $|\ga_2|\leq 1/2$.  
Combining the above two equations, we obtain 
\begin{multline*}
\Td(\hat{h}_1, \hat{h}_2) 
\leq (c_{2,2} k_1 + c_{2,1}) |\ga_1 - \ga_2|
= (c_{2,2} k_1 + c_{2,1}) |\ga_1 \ga_2| |\hat{\ga}_1 - \hat{\ga}_2| \\
\leq (k_1 + c_{2,1}) \frac{1}{4} |\hat{\ga}_1 - \hat{\ga}_2|  
\leq 4 k_1 \frac{1}{4} |\hat{\ga}_1 - \hat{\ga}_2|  
\leq k_1 |\hat{\ga}_1 - \hat{\ga}_2|.
\end{multline*}
The above inequality means that $\nprt{1}(\gU)$ is a $k_1$-horizontal curve, 
since $\hat{\ga}_1 \ltimes \hat{h}_1$ and $\hat{\ga}_2 \ltimes \hat{h}_2$ are arbitrary points on $\nprt{1}(\gU)$. 

\medskip

Now we deal with the bottom near-parabolic renormalization. 
The same lines of arguments presented in Equation~\ref{E:bound-on-nonlinearities-top} may be repeated to 
obtain the inequality 
\begin{align*}
\Td(\check{h}_1, \check{h}_2)  \leq  (c_{2,2} k_1 + c_{2,1}) |\ga_1 - \ga_2|
\end{align*}
On the other hand, we have $\check{\ga}(\ga\ltimes h)= -1/\gb(\ga\ltimes h)$. 
Hence,  
\begin{align*}
|\check{\ga}_1-\check{\ga}_2| 
& \geq |\check{\ga}_2- \check{\ga}_3| - |\check{\ga}_1 - \check{\ga}_3| 
&& (\text{Triangle Inequality}) \\
& \geq \frac{1}{c_{1,1}|\ga_1 \ga_2|} |\ga_1 - \ga_2|- c_{1,2} \Td(h_1, h_2)
&& (\text{Props.~\ref{P:1-1-derivative} and \ref{P:1-2-derivative}}) \\
& \geq \frac{1}{c_{1,1}|\ga_1 \ga_2|} |\ga_1 - \ga_2| - c_{1,2} k_1 |\ga_1 - \ga_2| 
&& (\text{as $\gU$ is $k_1$-horizontal}) \\
& \geq \Big (\frac{1}{c_{1,1}|\ga_1 \ga_2|} - c_{1,2} k_1 \Big) |\ga_1 - \ga_2|.
&& 
\end{align*}
Combining the above two inequalities, the choices of $k_1$ and $r_4$, as well as the inequality $c_{2,2}\leq 1$, 
we obtain, 
\begin{multline*} 
\Td(\check{h}_1, \check{h}_2) 
\leq  (c_{2,2} k_1 + c_{2,1}) |\ga_1 - \ga_2|  \\
\leq \dfrac{c_{2,2} k_1 + c_{2,1}}{\frac{1}{c_{1,1}|\ga_1 \ga_2|} - c_{1,2} k_1} |\ga_1 - \ga_2| 
=     \dfrac{c_{2,2} c_{2,1}+3 c_{2,1}}{\frac{3}{c_{1,1} r_4^2}- c_{1,2} c_{2,1}} |\ga_1 - \ga_2| \\
\leq  \dfrac{ 4 c_{2,1}}{12} |\ga_1 - \ga_2| 
=  k_1 |\ga_1 - \ga_2|. 
\end{multline*}
This finishes the proof of the desired statement; $\nprb{1}(\gU)$ is $k_1$-horizontal. 

The proof of the second part of the proposition is a special case of the above arguments. 
\end{proof}

In Section~\ref{SS:Schwarzian-derivative}, we use the notion of Schwarzian derivative to 
reduce Proposition~\ref{P:2-1-derivative} to the Euclidean variation of the functions, 
stated in Proposition~\ref{P:changes-non-linearity}.
The proof of the latter statement constitutes a rather long series of calculations. 
To make the idea clearer, we first present a proof for the top renormalization, in 
Sections~\ref{SS:basic-properties-lift-a} to \ref{SS:ecale-maps}, along which we also 
prove Propositions~\ref{P:sigma-fixed-point} and \ref{P:wide-petals}. 
Then, we prove the second part of Proposition~\ref{P:2-1-derivative} in Section~\ref{SS:2-1-derivative-b}. 
Proposition~\ref{P:2-2-derivative} follows from Theorem~\ref{T:Ino-Shi2} and an infinite dimensional 
Schwartz lemma of Royden-Gardiner (further details appear later). 
The proofs of Propositions \ref{P:1-1-derivative} and \ref{P:1-2-derivative} appear in 
Section~\ref{SS:pairs}. 


\subsection{Schwarzian derivative}\label{SS:Schwarzian-derivative}
When the maps $\nprt{1}(\ga \ltimes h)$ and $\nprb{1}(\ga \ltimes h)$ belong to $\BB{C} \ltimes \IS$, 
there are univalent mappings $\hat{\gp}_{\ga, h}: V \to \BB{C}$ and $\check{\gp}_{\ga,h}: V \to \BB{C}$, 
with $\hat{\gp}_{\ga,h}(0)=\check{\gp}_{\ga,h}(0)=0$ and $\hat{\gp}'_{\ga,h}(0)=\check{\gp}'_{\ga,h}(0)=1$, 
such that  
\begin{equation}\label{E:renormalization-notation}
\begin{gathered}
\nprt{1}(\ga \ltimes h)(z)=P \circ \hat{\gp}_{\ga,h}^{-1} (e^{2 \pi \B{i} \hat{\ga}}z), \;
\forall z \in e^{- 2\pi \B{i} \hat{\ga}} \cdot \hat{\gp}_{\ga,h}(V), \\
\nprb{1}(\ga \ltimes h)(z)=P \circ \check{\gp}_{\ga,h}^{-1} (e^{2 \pi \B{i} \check{\ga}}z), \;
\forall z \in e^{-2\pi \B{i} \check{\ga}} \cdot \check{\gp}_{\ga,h}(V).
\end{gathered}
\end{equation}
Combining with our earlier notations $\hat{h}$ and $\check{h}$ we have 
\begin{gather}
\hat{h}= P \circ \hat{\gp}_{\ga,h}^{-1}, \quad \check{h}= P \circ \check{\gp}_{\ga,h}^{-1}. 
\end{gather}

The \textit{Schwarzian derivative} of a univalent map $f$ is defined as (and denoted by) 
\[\Sd f = \left (\frac{f ''}{f '} \right)' - \frac{1}{2} \left(\frac{f''}{f}\right)^2,\]
where $'$ denotes the complex differentiation of an analytic map. 

Consider the mappings 
\begin{gather*}
\hat{\gO}_{\ga,\ga',h}= \hat{\gp}_{\ga,h}  \circ \hat{\gp}_{\ga',h}^{-1}: \hat{\gp}_{\ga',h}(V) \to \hat{\gp}_{\ga,h}(V), \\
\check{\gO}_{\ga,\ga',h}= \check{\gp}_{\ga,h}  \circ \check{\gp}_{\ga',h}^{-1}: \check{\gp}_{\ga',h}(V) \to \check{\gp}_{\ga,h}(V).
\end{gather*} 
The Schwarzian derivative of the above maps allows us to study the quantities 
$\Td(\hat{h}(\ga \ltimes h), \hat{h}(\ga' \ltimes h))$ and $\Td(\check{h}(\ga \ltimes h), \hat{h}(\ga' \ltimes h))$,
respectively. 
Let $\hat{\gh}_{\ga,h} |dz|$ and $\check{\gh}_{\ga,h} |dz|$ denote the hyperbolic metrics of constant 
curvature $-1$ on $\hat{\gp}_{\ga,h}(V)$ and $\check{\gp}_{\ga,h}(V)$, respectively. 
Then the hyperbolic norms of the Schwarzian derivatives $\Sd \hat{\gO}_{\ga, \ga',h}$ and 
$\Sd \check{\gO}_{\ga,\ga',h}$ are defined as 
\begin{gather*}
\big \| \Sd \hat{\gO}_{\ga, \ga',h} \big \|_{\hat{\gp}_{\ga',h}(V)}
= \sup_{z\in \hat{\gp}_{\ga',h}(V)} \frac{| \Sd \hat{\gO}_{\ga, \ga',h}(z) |}{\hat{\gh}_{\ga,h}(z)^{2}},\\
\big \|\Sd \check{\gO}_{\ga, \ga',h} \big \|_{\check{\gp}_{\ga',h}(V)}
= \sup_{z\in \check{\gp}_{\ga',h}(V)} \frac{| \Sd \check{\gO}_{\ga, \ga',h}(z) |}{\check{\gh}_{\ga,h}(z)^{2}}.
\end{gather*}

\begin{propo}\label{P:Schwarzian-norm}
There exists a constant $D_1$ such that for all $\ga,\ga'$ in $A^+(r_3)$, and $h$ in 
$\IS \cup  \{Q_0\}$, 
we have 
\[\| \Sd \hat{\gO}_{\ga, \ga',h} \|_{\hat{\gp}_{\ga',h}(V)} \leq D_1|\ga -\ga'|, \;
\| \Sd \check{\gO}_{\ga, \ga',h} \|_{\check{\gp}_{\ga',h}(V)} \leq D_1|\ga -\ga'|.\]
\end{propo}

\begin{proof}[Proof of Prop.~\ref{P:2-1-derivative} assuming Prop.~\ref{P:Schwarzian-norm}]
The domain $V$ is bounded by a smooth curve,  and hence is a quasi-circle. 
On the other hand, by Theorem~\ref{T:Ino-Shi2} the mappings $\hat{\gp}_{\ga,h}: V \to \BB{C}$ 
and $\check{\gp}_{\ga,h}: V \to \BB{C}$ have univalent extension onto the domains $U$ which 
contains the closure of $V$ in its interior. 
This implies that there exists a constant $K$, depending only on $V$ and $\mod (U\setminus V)$, such that $\hat{\gp}_{\ga,h}(V)$ and $\check{\gp}_{\ga,h}(V)$ are $K$-quasi-circles.  

By a rather classical result on role of the Schwarzian derivative for the univalence and quasi-conformal  extension, 
see \cite[Chapter 2, Thm 4.1]{Leh87} or \cite{Ahl63}, there exists a constant $\gep(K)$ such that 
$\hat{\gO}_{\ga, \ga',h}$ and $\check{\gO}_{\ga,\ga',h}$ can be extended to quasi-conformal mapping of the 
plane whose complex dilatations $\hat{\gm}$ and $\check{\gm}$, respectively, satisfy
\[\| \hat{\gm} \|_\infty \leq \frac{\| \Sd \hat{\gO}_{\ga, \ga',h} \|_{\hat{\gp}_{\ga',h}(V)}}{\gep(K)},   \quad 
\| \check{\gm} \|_\infty \leq  \frac{\| \Sd \check{\gO}_{\ga, \ga',h} \|_{\check{\gp}_{\ga',h}(V)}}{\gep(K)}.\]
 
By the definition of $\Td$ on $\IS$, and Proposition~\ref{P:Schwarzian-norm}, we conclude that  
Proposition~\ref{P:2-1-derivative} holds with the constant $c_{2,1}=D_1/\gep(K)$. 

\end{proof}

The Schwarzian derivative satisfies the chain rule 
\begin{gather*}
\Big \| \Sd (\hat{\gp}_{\ga,h}  \circ \hat{\gp}_{\ga',h}^{-1}) \Big \|_{\hat{\gp}_{\ga',h}(V)} = 
\Big \| \Sd \hat{\gp}_{\ga,h}  - \Sd \hat{\gp}_{\ga',h} \Big \|_{\hat{\gp}_{\ga',h}(V)}. 
\end{gather*}
A similar relation holds for the corresponding check maps. 
By virtue of these relation, Proposition~\ref{P:Schwarzian-norm} boils down to the following statement. 

\begin{propo}\label{P:changes-non-linearity}
For every Jordan domain $V'$ with $V \Subset V' \Subset U$, there exists a constant $D_2$ 
such that for all $\ga\in A(r_3)$, all $h$ in $\IS \cup \{Q_0\}$, and all $z$ in $V'$, we have
\begin{itemize}
\item[a)] $| \partial \hat{\gp}_{\ga,h} (z)/\partial \ga | \leq  D_2$,
\item[b)] $| \partial \check{\gp}_{\ga,h} (z)/\partial \ga| \leq  D_2$.
\end{itemize}
\end{propo}

\begin{proof}[Proof of Prop.~\ref{P:Schwarzian-norm} assuming Prop.~\ref{P:changes-non-linearity}] 
By the Cauchy integral formula and the estimates in Proposition~\ref{P:changes-non-linearity}, 
there is a constant $D_2'$ such that for all $z\in V$ we have 
\[\Big | \frac{\partial }{\partial \ga} \hat{\gp}'_{\ga,h} (z) \Big | \leq  D_2',  
\Big | \frac{\partial }{\partial \ga} \hat{\gp}''_{\ga,h} (z) \Big | \leq  D_2',  
\Big | \frac{\partial }{\partial \ga} \hat{\gp}'''_{\ga,h} (z) \Big | \leq  D_2'.\]
A similar set of estimates holds for the check maps. 
Also, by the Koebe distortion theorem, $|\hat{\gp}'_{\ga,h}|$ and $|\check{\gp}'_{\ga,h}|$ are uniformly 
bounded from above and away from zero, independent of $h$ and $\ga\in A(r_3)$. 
Combining these bounds together, one obtains the uniform bounds in Proposition~\ref{P:Schwarzian-norm}. 
\end{proof}

The proof of Proposition~\ref{P:changes-non-linearity} constitutes a series of calculations that will be 
presented in Sections~\ref{SS:basic-properties-lift-a} to \ref{SS:ecale-maps}. 


\subsection{Preliminary estimates on the maps in $\IS$} 
In particular, we use the following basic property of the maps in the class $\IS$. 

\begin{lem}\label{L:univalence-locus-h}
We have, 
\begin{itemize}
\item[a)] $\forall h \in \IS$ and $ \forall \ga \in A(1/2)$ the map $\ga \ltimes h$ 
is defined on the ball $B(0, 2 e^{-\pi \sqrt{2}}/9)$ and is univalent on the ball $B(0, 4 e^{-\pi \sqrt{2}}/27)$;
\item[b)] $\forall h \in \IS$ and $ \forall \ga \in A(1/2)$ the critical point of the map 
$\ga \ltimes h$, $\cp_{\ga\ltimes h}$, satisfies 
\[4 e^{-\pi \sqrt{2}}/27 \leq |\cp_{\ga \ltimes h}| \leq 4 e^{\pi \sqrt{2}}/3;\] 
\item[c)] $\forall \ga \in A(1/2)$, $Q_\ga$ is univalent on the ball $B(0,8 e^{-\pi \sqrt{2}}/27)$, 
and its critical point $\cp_\ga$ satisfies $8 e^{-\pi \sqrt{2}}/27 \leq |\cp_\ga | \leq 8 e^{\pi \sqrt{2}}/27$.   
\end{itemize}
\end{lem}

\begin{proof}
a) First note that the ellipse $E$ is contained in the ball $B(0,2)$, and therefore, the domain $V$ 
is contained in the ball $B(0,8/9)$. 
Applying the classical $1/4$-Theorem to the map $z \mapsto \frac{9}{8} \gf (\frac{8}{9} \cdot z)$, one 
concludes that $\gf(V)$ must contain the ball $B(0,2/9)$.  
That is, every $h\in \IS$ is defined on the ball $B(0,2/9)$. 

On the other hand, the polynomial $P$ is univalent on the ball $B(0,1/3)$. 
To see this, one may first write 
\[P(z_1)-P(z_2)= (z_1-z_2) ((1+z_1+z_2)^2 -z_1 z_2)\]
and note that for all $z_1, z_2$ in $B(0,1/3)$, $\Re (1+z_1+z_2)^2 > 1/9 $ and $ -1/9 <\Re (z_1 z_2)<1/9 $. 
The classical Koebe distortion theorem applied to the map $z\mapsto \frac{3}{2} \gf (\frac{2}{3} \cdot z)$ 
implies that $\gf(B(0,1/3))$ contains $B(0,4/27)$. 
For the simplicity of calculations we have applied the Koebe Theorem to the map 
$z\mapsto \frac{3}{2} \gf (\frac{2}{3} \cdot z)$ rather than the map $z \mapsto \frac{9}{8} \gf (\frac{8}{9} \cdot z)$.)
This means that every map $h\in \IS$ is univalent on the ball $B(0,4/27)$. 

For $\ga$ in $A(1/2)$, $ |\Im \ga| \leq \sqrt{2}/2$.
Composing with the rescalings at $0$, the above paragraphs imply that $\ga \ltimes h$ must be defined on the ball 
$(2/9) \cdot e^{-\pi \sqrt{2}}$, and must be univalent on the ball $(4/27) e^{-\pi \sqrt{2}}$. 

\medskip

b) The polynomial $P$ has a unique critical point at $-1/3$ within $V$. 
By the Koebe distortion theorem (see the above calculations), $|\gf (-1/3)|  \in [4/27, 4/3]$.  
Recall that the critical point of $h \in \IS$ is equal to $\gf(-1/3)$. 
Composing with the complex rotations $z \mapsto e^{2\pi i \ga } \cdot z$, we conclude the bounds in Part b). 
 
\medskip

c) The unique critical point of $Q_\ga$ lies at $-8 e^{2\pi \B{i} \ga}/27$.  
Further details are left to the reader. 
\end{proof}

\begin{lem}\label{L:preliminary-h''} 
We have 
\begin{itemize}
\item[a)] For every $h\in \IS$, $h\co{n}(\cp_h)$ tends to $0$ as $n$ tends to $+\infty$;
\item[b)] For every $h\in \IS$, $2\leq |h''(0)| \leq 7$. 
\end{itemize}
\end{lem}
The above lemma is stated in \cite[main theorem 1]{IS06}. 
We note that the uniform bound in Part b) of the lemma also follows from the Area Theorem and that the 
conformal radius of the set $V$ is strictly larger than one.  

\begin{proof}[Proof of Proposition~\ref{P:sigma-fixed-point}]
Although the class of maps $\IS$ is not compact (their domain of definitions are quasi-circles), 
every sequence of maps in $\IS$ converges, in the compact-open topology, to a holomorphic 
map with a non-degenerate parabolic fixed point at $0$. 
Indeed, by Lemma~\ref{L:univalence-locus-h}, the limiting map is defined on the ball $B(0, e^{\pi \sqrt{2}}/9)$, 
and the absolute value of its second derivative in the interval $[2,7]$. 

Every map $h$ in the closure of $\IS$, $\ol{\IS}$, has a non-degenerate parabolic fixed point at $0$. 
That is, a fixed point of order two.
Every such $h$ has an attracting and a repelling petal covering a punctured neighborhood of $0$.  
Hence, $h$ may not have any fixed point on the union of the petals. 
Using Lemmas~\ref{L:univalence-locus-h}-a and $\ref{L:preliminary-h''}$, one may find a neighborhood $W$ 
of $0$, bounded by a smooth curve, such that every $h$ in $\ol{\IS}$ has a unique fixed point at $0$ on 
the closure of $W$. 

By the Argument principle, there is $r_1 \in (0, 1/2)$ such that for all $\ga \in A(r_1)$ and all $h\in \IS$, 
$\ga \ltimes h$ has two fixed points in $W$, counted with multiplicity. 
As $\ga\neq 0$, $0$ is a simple fixed point of $\ga \ltimes h$, and hence, there must be another simple fixed point 
of $\ga \ltimes h$ within $W$.  
\end{proof}

In order to analyze the dependence of $\hat{h}$ on $\ga$ we need to study the definitions of $\nprt{1}$ and 
$\nprb{1}$ in detail. 
For $h\in \IS$ and $\ga\in A^+(+\infty)$, we denote the map $(\ga \ltimes h)$ by $h_\ga$, that is,   
\[h_\ga(z)= h(e^{2\pi \B{i} \ga} z), \; z \in e^{-2\pi \B{i} \ga}\cdot \Dom (h).\] 
This is consistent with the notation $Q_\ga(z)=Q_0(e^{2\pi \B{i}\ga} \cdot z)$. 
For $h$ in $\IS\u \{Q_0\}$ and $\ga$ in $A(r_2)$, Proposition~\ref{P:Fatou-coordinates} guarantees the 
existence of a Jordan domain, denoted by $\C{P}_{\ga, h}$ here, and a conformal change of coordinate 
(denoted by)
\[\gF_{\ga,h}: \C{P}_{\ga,h} \to \BB{C},\] 
which conjugates the dynamics of $h_\ga$ on 
$\C{P}_{\ga,h}$ to the translation by one. $r$
We need to control the dependence of $\gF_{\ga,h}$ on $\ga$, with uniform bounds independent of 
$\ga$ and $h$. 
It is convenient to work out this in a certain coordinate called the pre-Fatou coordinate.  


\subsection{The top pre-Fatou coordinate}\label{SS:basic-properties-lift-a}
Let $\gs_{\ga,h}$ denote the non-zero fixed point of $h_\ga$ that lies in $W$ 
(see Proposition~\ref{P:sigma-fixed-point}). 
Every map $h_\ga$ in $\PC{A(r_1)}$ or in $A(r_1) \ltimes \{Q_0\}$, may be written of the form 
\begin{equation}\label{E:expression-u}
h_\ga (z)= z + z(z-\gs_{\ga,h}) u_{\ga, h}(z),
\end{equation}
where $u_{\ga, h}$ is a holomorphic function defined on $\Dom h_\ga$ which is non-zero at $0$ 
and $\gs_{\ga,h}$.
As $\ga \to 0$, $\gs_{\ga, h} \to 0$, and we may identify a holomorphic function $u_{0,h}$ such that 
\begin{equation}\label{E:h-new-form}
h_0(z)= z+ z^2 u_{0,h}(z)),
\end{equation}
with $u_{0,h}(0)\neq 0$. 
By the pre-compactness of the class $\IS$, and the uniform bound in Lemma~\ref{L:preliminary-h''}-b, 
$|u_{\ga, h}(0)|$ is uniformly bounded from above and away from $0$, for $\ga \in A(r_1)\cup\{0\}$.  
That is, there is a constant $D_3$, independent of $\ga$ in $A(r_1)\cup \{0\}$ and $h$ in $\IS \cup \{Q_0\}$, 
such that 
\begin{equation}\label{E:u(0)}
D_3^{-1}  \leq u_{\ga,h}(0) \leq D_3.
\end{equation}
Differentiating Equation \eqref{E:h-new-form} at $0$ and $\gs_{\ga,h}$ provides us with the formulas:    
\begin{equation}\label{E:sigma-fixed-point}
\gs_{\ga,h}= (1-e^{2\pi \B{i} \ga})/ u_{\ga,h}(0), \; h_\ga'(\gs_{\ga,h})= 1+ \gs_{\ga,h} u_{\ga,h}(\gs_{\ga,h}).
\end{equation}
In particular, there is a constant $D_4$ such that for all $\ga\in A(r_1)$ and $h\in \IS \cup \{Q_0\}$, 
we have 
\begin{equation}\label{E:size-of-sigma}
\frac{1}{D_4} |\ga|  \leq |\gs_{\ga,h} |  \leq  D_4 |\ga|.
\end{equation}

Consider the covering map $\gta_{\ga,h}:\BB{C} \to \hat{\BB{C}} \setminus \{0, \gs_{\ga,h}\}$, 
where $\hat{\BB{C}}$ denotes the Riemann sphere, defined as
\nocite{Sh98, Sh00} 
\begin{equation}\label{E:tau-covering}
\gta_{\ga,h}(w) =  \frac{\gs_{\ga,h}}{1-e^{-2\pi \B{i} \ga w}}.
\end{equation} 
We have, 
\[\gta_{\ga,h}(w+\ga^{-1})= \gta_{\ga,h}(w), \;
\lim_{\Im (\ga w) \to +\infty} \gta_{\ga,h}(w)= 0, \;
\lim_{\Im (\ga w) \to -\infty} \gta_{\ga,h}(w)= \gs_{\ga,h}\]
Also, $\gta_{\ga,h}$ maps the points in $\BB{Z}/\ga$ to the point at infinity in $\hat{\BB{C}}$. 
\footnote{Add a figure for the covering map $\gta_{\ga,h}$, if there is spiral in the final result.}

\begin{lem}\label{L:covering-asymptotes-a}
For all $h\in \IS \cup \{Q_0\}$ and all $\ga$ in $A^+(r_1)$, we have the following estimates
\begin{itemize}
\item[a)] if  $\Im (\ga w) >0$, then
\[|\gta_{\ga,h}(w)| \leq D_4 \frac{|\ga|}{e^{2\pi \Im (\ga w)} -1};\]
\item[b)] if $\Im (\ga w)<0$, then 
\[|\gta_{\ga,h}(w) - \gs_{\ga,h}| \leq D_4 \frac{|\ga|  e^{2\pi \Im (\ga w)} }{1- e^{2\pi \Im (\ga w)}}.\]
\end{itemize}
\end{lem}

\begin{proof}
By Equation~\eqref{E:size-of-sigma}, for $w$ with $\Im (\ga w) >0$, 
\[|\gta_{\ga,h}(w)| \leq D_4 |\ga| \frac{1}{|1- e^{-2\pi\B{i}\ga w}|} \leq D_4 |\ga| \frac{1}{e^{2\pi \Im (\ga w)}-1}.\]

Similarly, for $w$ with $\Im (\ga w) < 0$, 
\[|\gta_{\ga,h}(w) - \gs_{\ga,h}| \leq |\gs_{\ga,h}|  \frac{|e^{-2\pi \B{i} \ga w}|}{|1- e^{-2\pi \B{i} \ga w}|} 
\leq D_4 |\ga| \frac{|\ga|  e^{2\pi \Im (\ga w)} }{1- e^{2\pi \Im (\ga w)}}. \qedhere\] 
\end{proof}

The map $h_\ga : \C{P}_{\ga, h} \to \BB{C}$ may be lifted via $\gta_{\ga,h}$ to a meromorphic map 
\[F_{\ga,h}: \gta_{\ga,h}^{-1}(\C{P}_{\ga,h}) \to \BB{C},\]
which is determined upto an additive constant in $\BB{Z}/\ga$. 
However, since there is no pre-image of $0$ or $\gs_{\ga,h}$ in $\C{P}_{\ga,h}$, $F_{\ga,h}$ is finite at every 
point in $ \gta_{\ga,h}^{-1}(\C{P}_{\ga,h})$. 
That is, $F_{\ga,h}$ is a holomorphic map on $ \gta_{\ga,h}^{-1}(\C{P}_{\ga,h})$.
Being a lift, for any choice of the additive constant in $\BB{Z}/\ga$, we must have 
\begin{equation}\label{E:lift-F}
h_\ga \circ \gta_{\ga,h}(w) = \gta_{\ga,h} \circ F_{\ga,h}(w), \; F_{\ga,h} (w+ 1/\ga) = F_{\ga,h} (w) +1/\ga, \; 
w \in  \gta_{\ga,h}^{-1}(\C{P}_{\ga,h}).
\end{equation} 
Indeed, there is a formula for $F_{\ga,h}$, in terms of the function $u_{\ga,h}$ in \eqref{E:expression-u},
\begin{equation}\label{E:lift-formula}
F_{\ga,h}(w)= w + \frac{1}{2\pi \B{i} \ga} \log \Big (1- \frac{\gs_{\ga,h} u_{\ga,h}(z)}{1+ z u_{\ga,h}(z)}\Big ), \;
\twith z= \gta_{\ga,h}(w).   
\end{equation}
A choice of the branch of $\log$ in the above formula corresponds to a choice of the additive constant 
in $\BB{Z}/\ga$. 
In this paper, we work with the branch satisfying $\Im \log (\cdot) \ci (-\pi, +\pi)$, in order that 
\begin{equation}\label{E:lift-asymptote}
\lim_{\Im (\ga w)\to +\infty} |F_{\ga,h}(w)-(w+1)|=0.
\end{equation}
Let $\hat{\C{P}}_{\ga,h}$ denote the connected component of the set $\gta_{\ga,h}^{-1}(\C{P}_{\ga,h})$ 
that separates $0$ from $1/\ga$. 
The unique critical point of $h_\ga$, which lies on the boundary of $\C{P}_{\ga,h}$, lifts under 
$\gta_{\ga,h}$ to a $1/\ga$-periodic set of points. 
There is a unique point in this set that lies on the boundary of $\hat{\C{P}}_{\ga,h}$. 
This is denoted by $\hat{\cp}_{\ga,h}$.

For $r\in (0, +\infty)$, define the set\footnote{The operator $\inte$ denotes the (topological) 
interior of a given set.} 
\[\gT_\ga(r) =\inte\big ( \BB{C}\setminus \cup_{n\in \BB{Z}} B(n/\ga, r) \big ).\]

\begin{lem}\label{L:lift-asymptotes-a}
There are constants $r_3'>0$, $D_5$, and $D_6$ such that for all $h\in \IS \cup \{Q_0\}$ and 
$\ga\in A^+(r_3')$, $F_{\ga,h}$ is defined and univalent on $\gT_\ga(D_5)$ and satisfies the 
following properties: 
\begin{itemize}\setlength{\itemsep}{1em}
\item[a)]for all $w\in \gT_\ga(D_5)$, 
\[|F_{\ga,h} (w) - (w+1)|\leq 1/4, \;  |F'_{\ga,h} (w) - 1| \leq 1/4.\]
\item[b)] for all $w\in \gT_\ga(D_5)$ with $\Im (\ga w) >0$, we have 
\begin{gather*}
|F_{\ga,h} (w) - (w+1)|\leq D_6  |\gta_{\ga,h}(w)|, \; |F'_{\ga,h} (w) - 1 | \leq D_6 |\gta_{\ga,h}(w)|.
\end{gather*}
\item[c)] for all $w\in \gT_\ga(D_5)$ with $\Im (\ga w) < 0$, we have 
\begin{gather*}
|F_{\ga,h} (w) - w + \frac{1}{2 \pi \B{i} \ga} \log h_\ga'(\gs_{\ga,h})| 
\leq D_6  |\gta_{\ga,h}(w) -\gs_{\ga,h}|,  \\
|F'_{\ga,h} (w) - 1 | \leq D_6 |\gta_{\ga,h}(w) -\gs_{\ga,h}|.
\end{gather*}
\end{itemize}
\end{lem}
\begin{proof}
Recall that by Proposition~\ref{P:sigma-fixed-point} for all $\ga \in A(r_1)$, $\gs_{\ga,h}$ belongs to
the Jordan neighborhood $W$ containing $0$.
There is a constant $C_1>0$ such that for all $\ga$ in $A(r_1)$ and $h$ in $\IS \cup \{Q_0\}$, 
$\gta_{\ga,h}(\gT_\ga(C_1))$ is contained in $B(0, 4 e^{-\pi \sqrt{2}} /27) \cap W$. 
By Lemma~\ref{L:univalence-locus-h}, $h_\ga$ is univalent on $B(0, 4 e^{-\pi \sqrt{2}} /27)$.  
Thus, there are no pre-image of $0$ and $\gs_{\ga,h}$ within $B(0, 4 e^{-\pi \sqrt{2}} /27) \cap W$, 
except $0$ and $\gs_{\ga,h}$. 
This implies that there is a lift of $h_\ga$ defined on $\gT_\ga(C_1)$,
which is holomorphic, one-to-one, and agrees with the function in Equation~\eqref{E:lift-formula} 
on the set $\hat{\C{P}}_{\ga,h} \cap \gT_\ga(C_1)$. 

By the pre-compactness of the class $\IS$, there is a constant $C_2$, independent of $\ga$ and $h$, 
such that for all $z \in \gta_{\ga,h}(\gT_\ga(C_1))$, $|u_{\ga,h}(z)| \leq C_2$. 
Choose $\gd_1 >0$ such that $\gd_1 D_4 < 1/(2C_2)$, where $D_4$ is the constant in 
Equation~\eqref{E:size-of-sigma}.
This implies that for all $\ga\in A(\gd_1)$ and all $h\in \IS \cup \{Q_0\}$, 
$|\gs_{\ga,h}| \leq \gd_1 D_4 < 1/ (2C_2)$. 
Now, there is $C_3 \leq C_1$ such that for all $\ga\in A(\gd_1)$, all $h\in \IS \cup \{Q_0\}$, and all 
$w \in \gT_\ga(C_3)$, $|\gta_{\ga,h}(w)| \leq 1/(2C_2)$. 
Putting these together, we have 
\[\Big | \frac{\gs_{\ga,h} u_{\ga,h}(z)}{1+ z u_{\ga,h}(z)}\Big | \leq \frac{D_4 C_2 |\ga|}{1/2}, \; 
\forall z \in \gta_{\ga,h}(\gT_\ga(C_3)).\]
Choose $\gd_2\leq \gd_1$ so that $D_4 C_2 \gd_2 < 1/2$.
This guarantees that for all $\ga \in A(\gd_2)$, all $h\in \IS \cup \{Q_0\}$, and all 
$z \in \gta_{\ga,h}(\gT_\ga(C_3))$, $1- \gs_{\ga,h} u_{\ga,h}(z)/ (1+ z u_{\ga,h}(z))$ is uniformly 
away from the negative real axis $(-\infty, 0]$ . 
In particular, the branch of $\log$ with $\Im \log (\cdot) \ci (-\pi, \pi)$
is defined in formula \eqref{E:lift-formula}. 
Similarly, we have 
\[|\gs_{\ga,h}u_{\ga,h}(\gs_{\ga,h})| \leq D_4 \gd_2 C_2 < 1.\]
Thus, by Equation~\eqref{E:sigma-fixed-point}, $h_\ga'(\gs_{\ga,h})$ is uniformly away from 
the negative axis. 
In particular, the same branch of $\log$ is defined at $h_\ga'(\gs_{\ga,h})$. 

Let us define $C_4 <+\infty$ as the maximum of $|\log '(x)|= |1/x|$, where $x$ belongs to the set 
\[B(1, 2 D_4 C_2 \gd_2) \u \{ 1/x \mid x\in B(1, C_2 D_4 \gd_2) \u 
\{e^{2\pi \ga \B{i}} \mid \ga \in A(\gd_2)\}.\] 

To prove the inequalities in the lemma, we may now use the formula in \eqref{E:lift-formula} for $F_{\ga,h}$. 
With $z=\gta_{\ga,h}(w)$, $|F_{\ga,h}(w) -w-1|$ is estimated through the following inequalities

\begin{align*}
\Big | \frac{1}{2\pi \B{i} \ga} \log \Big (1 - \frac{\gs_{\ga,h} u_{\ga,h}(z)}{1+ z u_{\ga,h}(z)}& \Big ) -1 \Big | \\
&\leq \frac{1}{|2\pi \B{i} \ga|} \Big |\log \Big (1- \frac{\gs_{\ga,h} u_{\ga,h}(z)}{1+ z u_{\ga,h}(z)}\Big ) -  
\log e^{2\pi \B{i}\ga} \Big |  \\
&\leq \frac{C_4}{2 \pi |\ga|} \Big | \big (1-\frac{\gs_{\ga,h} u_{\ga,h}(z)}{1+ z u_{\ga,h}(z)}\big ) - 
e^{2\pi \B{i}\ga}\Big| \\
&\leq \frac{C_4}{2 \pi |\ga|} \Big | \big (1-\frac{\gs_{\ga,h} u_{\ga,h}(z)}{1+ z u_{\ga,h}(z)}\big ) - 
(1-\gs_{\ga,h}u_{\ga,h}(0))\Big| \\
&\leq \frac{C_4}{2\pi |\ga|} |\gs_{\ga,h}|  \Big |\frac{u_{\ga, h}(z)}{(1+ zu_{\ga,h}(z))} - u_{\ga,h}(0) \Big | \\
&\leq \frac{C_4 D_4}{2\pi} C_5 |z|. 
\end{align*}

In the above inequalities, we have replaced $e^{2\pi \B{i} \ga}$ by $1-\gs_{\ga,h} u_{\ga,h}(0)$, 
because of Equation~\eqref{E:sigma-fixed-point}. 
In the last inequality, the uniform constant $C_5$ exists because of the pre-compactness of the class $\IS$. 

We may choose $r_3'\leq \gd_2$, and then choose $C_6\leq C_3$ such that for all $\ga\in A(r_3')$, all $h$ in 
$\IS \u \{Q_0\}$, and all $z\in \gta_{\ga,h}(\gT_\ga(C_6))$, we have $C_4 D_4 C_5 |z|/(2\pi) \leq 1/4$. 
This implies the first inequalities of Parts a and b.  

With $z= \gta_{\ga,h}(w)$ and using the relation $h_\ga'(\gs_{\ga,h})=1+ \gs_{\ga,h} u_{\ga,h}(\gs_{\ga,h})$, 
we have 
\begin{align*}
\Big | F_{\ga,h}(w)- w+ & \frac{1}{2\pi \ga\B{i}} \log h_\ga'(\gs_{\ga,h})\Big | \\
&=\Big |\frac{1}{2\pi \ga \B{i}} \log \big (1-\frac{\gs_{\ga,h} u_{\ga,h}(z)}{1+z u_{\ga,h}(z)})- 
\frac{1}{2\pi \B{i} \ga} \log \frac{1}{1+ \gs_{\ga,h} u_{\ga,h}(\gs_{\ga,h}))} \Big |\\
& \leq \frac{C_4}{2\pi |\ga|} \Big|(1-\frac{\gs_{\ga,h} u_{\ga,h}(z)}{1+z u_{\ga,h}(z)})- 
\frac{1}{1+\gs_{\ga,h} u_{\ga,h}(\gs_{\ga,h})}\Big| \\
&\leq \frac{C_4}{2\pi |\ga|} \sup_{t\in (0,1)}\Big|(\frac{1+(y-\gs_f)u_{\ga,h}(y)}{1+y u_{\ga,h}(y)})' \Big |_{y=t\gs_{\ga,h}+(1-t)z}\Big |
\cdot |z-\gs_{\ga,h}|\\
&\leq \frac{C_4 |\gs_{\ga,h}|}{2\pi |\ga|} \sup_{t\in (0,1)}\Big|(\frac{u_{\ga,h}(y)^2-u_{\ga,h}'(y)}{(1+y u_{\ga,h}(y))^2})\Big |_{y=t\gs_{\ga,h}+(1-t)z}\Big|
\cdot |z-\gs_{\ga,h}|\\
&\leq \frac{C_4 D_4}{2\pi} C_5' |z-\gs_{\ga,h}|. 
\end{align*}
The constant $C_5'$ in the last inequality above depends only on the class $\IS$. 
The above uniform estimate implies the first inequality in Part c. 

To prove the uniform bounds for the derivatives in Parts a, b, and c, one may use the Cauchy Integral formula 
for the first derivatives, at points in $\gT_\ga (C_6+1)$. 
This finishes the proof of the proposition by introducing $D_5=C_6+1$ and $D_6$ as the maximum of 
$C_4 D_4 C_5/(2\pi)$ and $C_4 D_4 C'_5/(2\pi)$.  
\end{proof}

\begin{lem}\label{L:lift-dependence-a}
There exists a constant $D_7$ such that for all $\ga, \ga' \in A^+(r_3')$ and all $h\in \IS \cup \{Q_0\}$ 
we have the following inequalities: 
\begin{itemize}
\item[a)] for all $w\in \gT_\ga(D_5)$ with $\Im (\ga w) >0$, 
\[|F_{\ga,h} (w) - F_{\ga',h}(w)| \leq D_7 |\ga-\ga'| \cdot  |\gta_{\ga,h}(w)|; \]
\item[b)] for all $w\in \gT _ \ga(D_5)$ with $\Im (\ga w) < 0$, 
\[ |F_{\ga,h} (w) -  F_{\ga',h}(w)| \leq D_7 |\ga -\ga'|. \]
\end{itemize}
\end{lem}

\begin{proof}
Let us define the function $B_1: \BB{C}\setminus (-\infty, -1]\to \BB{C}$ through  
$\log (1+x)=x B_1(x)$, and the function $B_2(\ga, w)$, for $\ga\in A(r_3')$ and $w\in \gT_\ga(D_5)$, 
by the formula 
\[B_2(\ga, w)=\frac{(1-e^{-2\pi \ga\B{i}})}{\ga}  
\Big (\frac{u_{\ga,h}(z)}{u_{\ga,h}(0) (1+z u_{\ga,h}(z))} -1 \Big ), z=\gta_{\ga,h}(w).\]
In Lemma~\ref{L:lift-asymptotes-a} we chose $r_3'$ and $D_5$ so that for $z\in \gta_{\ga,h}(\gT_\ga(D_5))$, 
$1 + z u_{\ga,h}(z)$ is uniformly away from $0$.  
Combining this with the pre-compactness of the class $\IS$, we have 
\[|B_2(\ga, w)| = O(|\gta_{\ga,h}(w)|), \; 
\Big |\frac{B_2(\ga, w)}{\gta_{\ga,h}(w)} - \frac{B_2(\ga',w)}{\gta_{\ga',h}(w)}\Big |= O(|\ga-\ga'|), \] 
for some uniform constants in $O$. 

Let $w\in \gT_\ga(D_5)$. 
Using the formulas \eqref{E:sigma-fixed-point} and  \eqref{E:lift-formula}, 
\begin{align*}
F_{\ga,h}(w)-w-1&=
\frac{1}{2\pi\ga \B{i}} \log\big(1-\frac{\gs_{\ga,h} u_{\ga,h}(z)}{1+z u_{\ga,h} (z)}\big) 
- \frac{1}{2\pi \ga \B{i}} \log e^{2\pi \ga \B{i}}\\
&=\frac{1}{2\pi \ga \B{i}}\log \Big(\big(1- \big (\frac{1-e^{2\pi \B{i} \ga}}{u_{\ga,h}(0)} \big )
\frac{u_{\ga,h}(z)}{1+z u_{\ga,h}(z)}\big )e^{-2\pi\ga\B{i}}\Big)\\
&=\frac{1}{2\pi \ga \B{i}}\log \Big( e^{-2\pi \B{i} \ga} + \big (\frac{1-e^{-2\pi \B{i} \ga}}{u_{\ga,h}(0)} \big )
\frac{u_{\ga,h}(z)}{1+z u_{\ga,h}(z)}\Big)\\
&=\frac{1}{2\pi\ga \B{i}}\log\Big(1+(1-e^{-2\pi \ga\B{i}})
\big(\frac{u_{\ga,h}(z)}{u_{\ga,h}(0)(1+z u_{\ga,h}(z))} -1 \big)\Big).\\
&= \frac{1}{2\pi \B{i}} B_2(\ga,w) B_1 (\ga B_2(\ga, w)). 
\end{align*}

Define the set 
\[O=\partial \gT_\ga(D_5) \u \{w \in \gT_\ga (D_5)\mid \Im (\ga w) = 0\}.\]
For $\ga'$ sufficiently close to $\ga$, $F_{\ga,h}(w)$ and $F_{\ga',h}(w)$ are defined for all $w \in O$. 
Moreover, for $\ga'$ sufficiently close to $\ga$, $|\gta_{\ga,h}(w)|/ |\gta_{\ga',h}(w)|$ is uniformly bounded 
from above and away from $0$, independent of $w \in O$ and $h\in \IS$.
For $w\in O$, using $B_1$ and $B_2$, we have 
\begin{align*}
| F_{\ga,h}(w) - F_{\ga',h}(w) | 
&= \frac{1}{2\pi} \big |B_2(\ga,w)B_1(\ga B_2(\ga,w))-B_2(\ga',w) B_1(\ga' B_2(\ga',w))\big| \\
& \leq \frac{1}{2\pi} |B_2(\ga, w)| \cdot  \big |B_1(\ga B_2(\ga,w))-B_1(\ga' B_2(\ga',w))\big|  \\
& \qquad + \frac{1}{2\pi} |B_1(\ga' B_2(\ga',w))|  \cdot \big |B_2(\ga,w)-B_2(\ga',w)\big|\\
&\leq C_1 | \gta_{\ga,h}(w)| \cdot  |\ga- \ga'|+ C_2 |\ga-\ga'| \cdot |\gta_{\ga,h}(w)|
\end{align*}
The constants $C_1$ and $C_2$ only depend on the class $\IS$. 
We now use the maximum principle in the $w$ variable in order to prove the estimates in Parts a and b. 

By the above equation, the estimate in Part a holds on the boundary of the set 
$\{w\in \gT_\ga(D_5) \mid \Im (\ga w)>0\}$. 
It also holds as $\Im (\ga w) \to + \infty$, since $|F_{\ga,h}(w) - F_{\ga',h}(w)| \to 0$ by 
Equation~\eqref{E:lift-asymptote}.  
This implies that the uniform bound must hold for all $w$ in $\gT_\ga(D_5)$ with $\Im (\ga w)>0$. 

On the other hand, for $w$ in $\gT_\ga(D_5)$ with $\Im (\ga w)< 0$, $|\gta_{\ga,h}(w)|$ is uniformly 
bounded from above. 
Hence, by the above equation, $|F_{\ga,h}(w) - F_{\ga',h}(w)|$ is bounded by a uniform constant times 
$|\ga-\ga'|$. 
We need to look at the asymptotic behavior of this difference as $\Im (\ga w)\to -\infty$. 
For this, we have 
\begin{align*}
\lim_{\Im w \to -\infty}&|F_{\ga,h}(w)- F_{\ga',h}(w)| \\
&=\Big|\frac{1}{2\pi\ga \B{i}}\log\big(1-\frac{\gs_{\ga,h} u_{\ga,h}(\gs_{\ga,h})}{1+\gs_{\ga,h} u_{\ga,h}(\gs_{\ga,h})}\big)-
\frac{1}{2\pi\ga' \B{i}} \log\big(1-\frac{\gs_{\ga',h} u_{\ga',h}(\gs_{\ga',h})}{1+\gs_{\ga',h} u_{\ga',h}(\gs_{\ga',h})}\big)\Big|\\
&= \Big|\frac{1}{2\pi\ga \B{i}}\log (1+\gs_{\ga,h} u_{\ga,h}(\gs_{\ga,h})) - 
\frac{1}{2\pi \ga' \B{i}}\log(1+\gs_{\ga',h} u_{\ga',h}(\gs_{\ga',h}))\Big|\\
&\leq C_3|\ga-\ga'|,
\end{align*}  
for some constant $C_3$ depending only on the class $\IS$. 
By the maximum principle, the uniform bound in Part b must hold for all 
$w\in \gT_\ga(D_5)$ with $\Im (\ga w) < 0$.  
\end{proof}

Recall that the map $\ga \ltimes h$ has a unique critical point in its domain of definition. 
This points lifts under $\gta_{\ga,h}$ to a critical point for $F_{\ga,h}$ that lies on $\hat{\C{P}}_{\ga,h}$. 
We denote this point by $\hat{\cp}_{\ga,h}$. 
By Proposition~\ref{P:Fatou-coordinates}, several iterates of this point under $F_{\ga,h}$ remain in $\hat{\C{P}}_{\ga,h}$, and there is the first moment when the orbit exits $\hat{\C{P}}_{\ga,h}$. 
Once the orbit exits this set, it falls in the connected component of $\BB{C} \setminus \hat{\C{P}}_{\ga,h}$
containing $1/\ga$. 
In particular, there is the smallest $i_{\ga,h} \in \BB{N}$ such that 
\[\Re F_{\ga,h} \co{i_{\ga,h}}(\hat{\cp}_{\ga,h}) \in (D_5, D_5+2).\] 
Indeed, by the pre-compactness of $\PC{A(r_3')}$, $i_{\ga,h}$ is uniformly bounded from above independent of 
$\ga$ and $h$.  
The integer $i_{\ga,h}$ may be chosen so that it is locally constant near a given $\ga \in A(r_3')$ 
and $h\in \IS$. 
We set the notation 
\[v_{\ga,h}=F_{\ga,h} \co{i_{\ga,h}}(\hat{\cp}_{\ga,h}).\]
The above point shall be used as a reference point for the normalization of the auxiliary mappings we introduce 
to study the dependence of $\gF_{\ga,h}$ on $\ga$. 

\begin{lem}\label{L:v-dependence}
There exists a constant $D_8$ such that for all $\ga, \ga' \in A^+(r_3')$ and all $h\in \IS \cup \{Q_0\}$, we have 
\[| v_{\ga,h} - v_{\ga',h}| \leq D_8 |\ga-\ga'|.\]
\end{lem}
\begin{proof}
By the definition of $F_{\ga,h}$, $v_{\ga,h}$ is the pre-image of 
$h_\ga\co{i_{\ga,h}}(\cp_{\ga,h})= h_\ga \co{(i_{\ga,h}-1)} (-4/27)$.  
Also, by the pre-compactness of the class $\IS$, $|h_\ga \co{(i_{\ga,h}-1)} (-4/27)|$ is uniformly bounded 
from above and away from $0$. 
Then, the uniform bound in the lemma may be obtained by a similar argument as in the proof of 
Lemma~\ref{L:lift-dependence-a}.  
\end{proof}


\subsection{Quasi-conformal Fatou coordinates}\label{SS:quasi-conformal-coordinates}
Recall the map $\gF_{\ga,h}: \C{P}_{\ga,h}  \to \BB{C}$. 
We may consider the univalent map  
\begin{equation}\label{E:equivariant-L}
L_{\ga,h} = \gta_{\ga,h} ^{-1} \circ \gF_{\ga,h}^{-1}: \gF_{\ga,h} (\C{P}_{\ga,h}) \to \hat{\C{P}}_{\ga,h},
\end{equation}
where $\gta_{\ga,h}^{-1}$ is the inverse of the map $\gta_{\ga,h}: \hat{\C{P}}_{\ga,h} \to \C{P}_{\ga,h}$. 
By the functional equation for the Fatou coordinate in Proposition~\ref{P:Fatou-coordinates}, 
for every $\gx \in \gF_{\ga,h}(\C{P}_{\ga,h})$ with $\gx+1 \in \gF_{\ga,h}(\C{P}_{\ga,h})$, we have 
\[L_{\ga,h}(\gx+1) = F_{\ga,h}(L_{\ga,h}(\gx)).\]
The map $L_{\ga,h}$ may be extended onto the boundary of $\gF_{\ga,h}(\C{P}_{\ga,h})$, that is onto the two 
vertical limes. 
By this extension, $L_{\ga,h}(0)=\hat{cp}_{\ga,h}$. 
Then, by the above functional equation, it follows that $L_{\ga,h}(i_{\ga,h})=v_{\ga,h}$, where $i_{\ga,h}$ and 
$v_{\ga,h}$ are defined just before Section~\ref{SS:quasi-conformal-coordinates}. 

Using the functional equation~\eqref{E:equivariant-L}  and the uniform bound in 
Lemma~\ref{L:lift-asymptotes-a}-a, one may extend the map $L_{\ga,h}^{-1}$ onto the set $\gS_2$ of 
points $w$ with 
\[\arg (w - \sqrt{2}D_5 ) \in [\frac{-3\pi}{4}, \frac{3\pi}{4}] \bmod 2\pi, \]
and 
\[\arg (w-\frac{1}{\ga} + \sqrt{2}D_5) \in [\frac{\pi}{4}, \frac{7\pi}{4}] \bmod 2\pi.\] 
The extended map is univalent on $\gS_2$. 

\begin{lem}\label{L:bounded-derivative-L} 
There is a constant $D_9$, independent of $\ga$ and $h$, such that for all $w$ in $\gS_2$ we have 
\[1/D_9  \leq  |(L_{\ga,h}^{-1})'(w)| \leq D_9.\] 
Moreover, as $\Im w \to + \infty$ within $\gS_2$, $(L_{\ga,h}^{-1})'(w) \to +1$. 
\end{lem}

\begin{proof}
First assume $w \in \gS_2$ so that $B(w, 3/2) \ci \gS_2$. 
By the Koebe Distortion theorem applied to $L_{\ga,h}^{-1}$ on $B(w,3/2)$ we know that 
$L_{\ga,h}^{-1}$ is uniformly close to a (complex) linear map on the strictly smaller ball $B(w,5/4)$.  
By the uniform estimate in Lemma~\ref{L:lift-asymptotes-a}-a, $|F_{\ga,h}(w)-w-1| \leq 1/4$. 
Hence these two points lie in $B(w, 5/4)$, and are mapped by $L_{\ga,h}^{-1}$ to a pair of points apart by one.
This implies that $(L_{\ga,h}^{-1})(w)$ must be uniformly bounded from above and away from zero. 

For $w$ near the vertical line $\Re w = \Re (2\ga)^{-1}$ and with $\Im w$ large, there is a ball of 
radius comparable to $\Im w$, centered about  $w$, which is contained in $\gS_2$. 
By the Koebe distortion theorem, $L_{\ga,h}^{-1}$ tends to a (complex) linear map on $B(w,3/2)$, as 
$\Im w \to +\infty$. 
Meanwhile, the points $w$ and $F_{\ga,h}(w)$ that are nearly apart by one, are mapped to two points 
exactly apart by one. 
This implies that $(L_{\ga,h}^{-1})'(w) \to 1$ as $\Im w \to +\infty$. 

Finally, an arbitrary $w \in \gS_2$ marches under the iterates of $F_{\ga,h}$ to a point near the vertical line 
$\Re w = \Re (2\ga)^{-1}$, where $(L_{\ga,h}^{-1})'$ is uniformly bounded from above and away from zero. 
Moreover, this derivative tends to $+1$ as the imaginary part tends to $+\infty$.  
By the uniform estimate in Lemma~\ref{L:lift-asymptotes-a}-a, the number of forward or backward 
iterates required to reach the proximity of the vertical line is linear in $\Im w$. 
Moreover, $|F_{\ga,h}'|$ is uniformly bounded from above and away from $0$ on $\gS_2$, and also $F_{\ga,h}'$ 
tends to $+1$ exponentially fast as $\Im w\to +\infty$. 
Using this and the functional equation~\eqref{E:equivariant-L}, we conclude that  
$(L_{\ga,h}^{-1})'(w)$ is uniformly bounded from above and away from zero, and must tend to $+1$ 
as $\Im w \to +\infty$ within $\gS_2$. 
\end{proof}

\begin{rem}
The proof of the above lemma provides us with a uniform bound on $|(L_{\ga,h}^{-1})'-1|$ of order $1/\Im w$.   
Indeed, an exponentially decaying bound on $|(L_{\ga,h}^{-1})'-1|$ may be proved using an alternative approach 
introduced in \cite{Ch10-II}. 
We do not need that estimate here. 
\end{rem}

We shall analyze the dependence of the map $L_{\ga,h}$ on $\ga$ by comparing it to 
two quasi-conformal changes of coordinates denoted by $H^1_{\ga,h}$ and $H^2_{\ga,h}$, 
For $\ga \in A(r_3')$ and $h$ in $\IS \cup \{Q_0\}$, define the map 
\[H^1_{\ga,h}:\{ \gz\in \BB{C} \mid \Re \gz \in [0,1]\} \to \gT_\ga (D_5) \ci  \Dom F_{\ga,h}\] 
as 
\begin{gather*}
H^1_{\ga,h} (\gz) = (1-\Re \gz)(v_{\ga,h} + \B{i} \Im \gz ) + (\Re \gz) F_{\ga,h}(v_{\ga,h} + \B{i} \Im \gz ).
\end{gather*}
When $\Re \gz=0$, we have 
\[H^1_{\ga,h}(\gz+1)=F_{\ga,h}(H^1_{\ga,h}(\gz)).\]
Also, we have normalized the map $H^1_{\ga,h}$ by $H^1_{\ga,h}(0)=v_{\ga,h}$. 

\begin{lem}
The map $H^1_{\ga,h}$ is a quasi-conformal mapping whose complex dilatation satisfies 
\[|\partial_{\ol{\gz}} H^1_{\ga,h}(\gz) / \partial_{\gz} H^1_{\ga,h}(\gz) | \leq 4/3, \; \forall \gz \in \Dom H^1_{\ga,h}.\] 
\end{lem}

\begin{proof}
The first partial derivatives of $H^1_{\ga,h}$ exist and are given by 
\begin{equation}\label{E:first-partials-H}
\begin{gathered}
\partial_\gz H^1_{\ga,h}(\gz)
 = \frac{1}{2} \Big ( F_{\ga,h}(v_{\ga,h}+\B{i} \Im \gz) - (v_{\ga,h}+\B{i}\Im \gz)+1
+\Re \gz (F_{\ga,h}'(v_{\ga,h}+ \B{i} \Im \gz)-1)\Big ) \\
\partial_{\ol{\gz}} H^1_{\ga,h}(\gz)
= \frac{1}{2} \Big ( F_{\ga,h}(v_{\ga,h}+\B{i} \Im \gz) - (v_{\ga,h}+\B{i}\Im \gz)-1
- \Re \gz (F_{\ga,h}'(v_{\ga,h}+ \B{i} \Im \gz)-1)\Big )   
\end{gathered}
\end{equation}
It follows from the estimates in Lemma~\ref{L:lift-asymptotes-a} that $H^1_{\ga,h}$ is a homeomorphism 
and its complex dilatation satisfies the uniform bound in the lemma.  
\end{proof}

By the formulas in \eqref{E:first-partials-H} and the asymptotic estimates in Lemma~\ref{L:lift-asymptotes-a}, 
Parts b and c, we have 

\begin{itemize}
\item[--]
if $\Im H^1_{\ga,h}(\gz) \geq 0$, then 
\begin{equation} \label{E:asymptotic-dil-top}
\begin{gathered}
|\partial_{\ol{\gz}} H^1_{\ga,h}(\gz)| \leq D_6 |\gta_{\ga,h}(v_{\ga,h}+\B{i}\Im \gz)|,  \\
|\partial_\gz H^1_{\ga,h}(\gz) -1| \leq D_6 |\gta_{\ga,h}(v_{\ga,h}+\B{i}\Im \gz)|,
\end{gathered}
\end{equation}
\item[--] if $\Im H^1_{\ga,h}(\gz)\leq 0$, then 
\begin{equation}\label{E:asymptotic-dil-bottom}
\begin{gathered}
\Big | \partial_{\ol{\gz}} H^1_{\ga,h}(\gz) + \frac{1}{2}+\frac{1}{4\pi \ga \B{i}} 
\log h_\ga'(\gs_{\ga,h}) \Big | 
\leq D_6 |\gta_{\ga,h}(v_{\ga,h}+\B{i} \Im \gz) -\gs_{\ga,h}|,\\
\Big |\partial_\gz H^1_{\ga,h}(\gz) - \frac{1}{2}+\frac{1}{4\pi \ga \B{i}} \log h_\ga'(\gs_{\ga,h})\Big | 
\leq  D_6 |\gta_{\ga,h}(v_{\ga,h}+\B{i} \Im \gz)- \gs_{\ga,h}|. 
\end{gathered}
\end{equation}
\end{itemize}
Define the homeomorphism
\[G^1_{\ga,h}= L_{\ga,h}^{-1} \circ H^1_{\ga,h} 
:\{\gz\in \BB{C} \mid \Re \gz \in [0,1] \} \to \gF_{\ga,h}(\C{P}_{\ga,h}).\]
Using the complex chain rule, the uniform bound in Lemma~\ref{L:bounded-derivative-L}, and the asymptotic 
bounds in \eqref{E:asymptotic-dil-top}-\eqref{E:asymptotic-dil-bottom}, we obtain the following asymptotic 
bounds on $\partial_{\ol{\gz}}G^1_{\ga,h}$:
\begin{itemize}
\item[--]
if $\Im H^1_{\ga,h}(\gz) \geq 0$, then 
\begin{equation} \label{E:dil-G-t}
|\partial_{\ol{\gz}} G^1_{\ga,h}(\gz)| \leq D_6 D_9 |\gta_{\ga,h}(v_{\ga,h}+\B{i}\Im \gz)|,  
\end{equation}
\item[--] if $\Im H^1_{\ga,h}(\gz)\leq 0$, then 
\begin{equation}\label{E:dil-G-b}
\begin{gathered}
\Big | \partial_{\ol{\gz}} G^1_{\ga,h}(\gz) + \frac{1}{2}+\frac{1}{4\pi \ga \B{i}} \log h_\ga'(\gs_{\ga,h})\Big | 
\leq D_6 D_9 |\gta_{\ga,h}(v_{\ga,h}+\B{i} \Im \gz) -\gs_{\ga,h}|.
\end{gathered}
\end{equation}
\end{itemize}
Note that the chain rule does not automatically give us a similar upper bound on 
$|\partial_\gz G^1_{\ga,h}-1|$. 
However, it implies that $|(G^1_{\ga,h})'|$ is uniformly bounded from above and away from zero, and 
$\lim_{\Im \gz\to +\infty} (G^1_{\ga,h})'(\gz)=1$. 
We do not \textit{a priori} know the image of the map $G^1_{\ga,h}$, except that when 
$\Re \gz =0$, we must have 
\begin{equation}\label{E:equivariant-G^1}
G^1_{\ga,h}(\gz+1)= G^1_{\ga,h}(\gz)+1.
\end{equation}
However, the above bound on $|\partial_{\ol{\gz}} G^1_{\ga,h}-1|$ and functional relation makes this map 
to be almost an affine one. 
The precise statement is formulated in the next lemma. 

\begin{lem}\label{L:asymptotic-G-top}
There exists a constant $D_{10}$, independent of $\ga$ and $h$, such that 
for all $\gz$ with $\Im \gz>0$,  
\[|G^1_{\ga,h}(\gz)-\gz| \leq D_{10} (1-\log |\ga|).\] 
Moreover, $\lim_{\Im \gz\to +\infty} (G^1_{\ga,h}(\gz)-\gz)$ exists and is a finite number. 
\end{lem}

\begin{proof}
Fix real numbers $r_2 > r_1+1 \geq 1$, and define the set 
\[A=\{\gz\in \BB{C} \mid \Re \gz \in [0,1], \Im \gz \in [r_1,r_2]\}.\] 
By the uniform bound in \eqref{E:dil-G-t}, we have 
\begin{equation}\label{E:1} 
\Big | \iint_A \partial_{\ol{\gz}} G^1_{\ga,h}(\gz)  \, d\gz d \ol{\gz}  \Big | 
\leq D_6 D_9 \iint_A |\gta_{\ga,h}(v_{\ga,h}+ \B{i} \Im \gz)| \, d\gz d\ol{\gz}. 
\end{equation}
By the pre-compactness of the class $\IS$, $|v_{\ga,h}|$ is uniformly bounded from above, independent of 
$\ga$ and $h$. 
By Lemma~\ref{L:covering-asymptotes-a}, there is a constant $C$, independent of $\ga$ and $h$, such that 
the right hand side of the above inequality  is bounded by 
\begin{equation}\label{E:2}
\begin{aligned}
& C \Big (1+ |\ga| \int_{r_1+1}^{r_2} \frac{1}{e^{2\pi t |\ga| \cos (\arg \ga)}-1} \, dt\Big)  \\
& \leq C \Big (1+ |\ga| \frac{1}{2 \pi |\ga| \cos (\arg \ga)} 
\big (\log (1- e^{-2\pi t |\ga| \cos (\arg \ga)}\big ) \Big |_{t=r_1+1}^{t=r_2})\Big) \\
& \leq C \Big (1+ \frac{1}{\pi \sqrt{2}}\big (\log (1- e^{-2\pi t |\ga| \cos (\arg \ga)}\big ) \Big |_{t=r_1+1}^{t=+\infty})\Big)\\ 
&\leq C\big (1-   \frac{1}{\pi \sqrt{2}}\log (1- e^{-2\pi (r_1+1) |\ga| \cos (\arg \ga)})\big ). 
\end{aligned} 
\end{equation}
In the above inequalities we have used $\arg \ga \in [-\pi/4, \pi/4]$ and so $\cos (\arg \ga)\leq \sqrt{2}/2$.
Since $r_1\geq 0$, we have 
\begin{multline*}
1- \frac{1}{\pi \sqrt{2}}\log (1- e^{-2\pi (r_1+1)|\ga| \cos (\arg \ga)}) \\
\leq 1- \frac{1}{\pi \sqrt{2}} \log (1- e^{-2\pi |\ga| \cos (\arg \ga)}) 
\leq C' (1- \log |\ga|),
\end{multline*}
for some explicit constant $C'$ independent of $\ga$. 
Combining the above inequalities we conclude that the right hand side of Equation~\eqref{E:1} is bounded 
from above by a uniform constant times $1-\log |\ga|$. 

On the other hand, by the Green's integral formula, and using the relation in Equation~\eqref{E:equivariant-G^1}
on the vertical side of $A$ where $\Re \gz=0$, the integral in the left hand side of \eqref{E:1} 
is equal to 
\begin{equation}\label{E:3}
\oint_{\partial A} G^1_{\ga,h}(\gz) \, d\gz 
= (r_2-r_1)\B{i} - \int_0^1 G^1_{\ga,h}(r_2\B{i} +t) \, dt + \int_0^1 G^1_{\ga,h}(r_1 \B{i} +t) \, dt. 
\end{equation}
The derivative $|\partial_\gz G^1_{\ga,h}(\gz)|$ is uniformly bounded from above independent of 
$\ga$ and $h$, because of  the bounds in Equations~\eqref{E:asymptotic-dil-top}-\eqref{E:asymptotic-dil-bottom}
and Lemma~\ref{L:bounded-derivative-L}. 
This implies that 
\[ \Big | \int_0^1 G^1_{\ga,h}(r_2\B{i} +t) \, dt - G^1_{\ga,h}(r_2 \B{i}) \Big | , \; 
\Big|   \int_0^1 G^1_{\ga,h}(r_1 \B{i} +t) - G^1_{\ga,h}(r_1 \B{i}) \Big |\]
are uniformly bounded from above, independent of $\ga$ and $h$. 

Let us choose $r_1=0$ and a point $\gz\in A$ with $\Re \gz=0$. 
By the above arguments we conclude that $|G^1_{\ga,h}(\gz) - \gz| \leq C' (1- \log |\ga|) + |G^1_{\ga,h}(0)|$. 
However, $|G^1_{\ga,h}(0)|= |i_{\ga,h}|$ is uniformly bounded from above, independent of $\ga$ and $h$. 
This implies the desired upper bound in the fist part of the lemma at $\gz$. 
The uniform upper bound on $|\partial_\gz G^1_{\ga,h}|$ may be used to establish the first part of the lemma 
at other points $\gz \in \Dom G^1_{\ga,h}$. 

Let $\gz_1$ and $\gz_2$ be two points with $\Re \gz_1= \Re \gz_2=0$ and $r_2=\Im \gz_2 > r_1=\Im \gz_1$.
By the asymptotic estimate in Equation~\eqref{E:asymptotic-dil-top} and in 
Lemma~\ref{L:bounded-derivative-L}, $|\partial_{\gz} G^1_{\ga,h}(\gz)| \to 1$ as $\Im \gz\to +\infty$.
This implies that the integral in \eqref{E:3} tends to $(\gz_2-G^1_{\ga,h}(\gz_2)) - (\gz_1- G^1_{\ga,h}(\gz_1))$ 
as $\Im \gz_2$ and $\Im \gz_1$ tend to $+\infty$. 
On the other hand, as $r_2$ and $r_1$ tend to $+\infty$, the last expression in \eqref{E:2} tends to $0$. 
This means that the difference $(\gz_2-G^1_{\ga,h}(\gz_2)) - (\gz_1- G^1_{\ga,h}(\gz_1))$ satisfies the 
Cauchy's criterion, and hence the limit of $(\gz_2-G^1_{\ga,h}(\gz_2))$ exists as $\Im \gz \to +\infty$ 
along $\Re \gz=0$. 
Finally, since $|\partial_{\gz} G^1_{\ga,h}(\gz)| \to 1$ as $\Im \gz\to +\infty$ within $\Dom G^1_{\ga,h}$, the 
limit must exist as $\Im \gz \to +\infty$ within $\Dom G^1_{\ga,h}$. 
This finishes the proof of the last part of the lemma.
\end{proof}

\begin{propo}\label{P:asymptote-L-top}
There exists a constant $D_{11}$ such that for all $\ga \in A(r_3')$, $h\in \IS\u\{Q_0\}$, and all 
$\gx \in \Dom L_{\ga,h}$ with $\Im \gx\geq 0$ we have 
\[|L_{\ga,h}(\gx)-\gx| \leq D_{11} (1-\log |\ga|).\]
Moreover, the limit $\lim_{\Im \gx \to +\infty} L_{\ga,h}(\gx)-\gx$ exists. 
\end{propo}

\begin{proof}
Recall that $H^1_{\ga,h}(0)=v_{\ga,h}$, and $|v_{\ga,h}|$ is uniformly bounded from above independent of 
$\ga$ and $h$. 
This implies that $|H^1_{\ga,h}(\gz) -\gz|$ is uniformly bounded from above, independent of $\ga$, $h$, and 
$\gz\in \Dom H^1_{\ga,h}$. 
Indeed, $H^1_{\ga,h}(\gz) -\gz$ converges to $v_{\ga,h}$ as $\Im \gz\to +\infty$. 
On the other hand, by Lemma~\ref{L:asymptotic-G-top}, $(G^1_{\ga,h})^{-1}$ also satisfies the similar 
properties on the image of $G^1_{\ga,h}$. 
This implies the statements in the proposition for the composition 
$L_{\ga,h}=H^1_{\ga,h} \circ (G^1_{\ga,h})^{-1}$, on the image of $G^1_{\ga,h}$.
Finally, the functional relation in Equation~\eqref{E:equivariant-L} and the estimates on $F_{\ga,h}$ 
in Lemma~\ref{L:lift-asymptotes-a} may be used to prove the proposition at points $\gx \in \Dom L_{\ga,h}$. 
\end{proof}  

The quasi-conformal change of coordinate $H^1_{\ga,h}$ allows us to analyze the behavior of 
$L_{\ga,h}$ near the left hand side of $\gF_{\ga,h}(\C{P}_{\ga,h})$. 
For example, it shall allow us to study the spiraling behavior of $\gF_{\ga,h}^{-1}(x+\B{i}y)$, as $y$ tends to 
$+\infty$ or to $-\infty$, for small positive values of $x$. 
On the other hand, we do not a priori know the size of  $\gF_{\ga,h}(\C{P}_{\ga,h})$, say, its vertical 
width in terms of $1/\ga$. 
We also need to understand the behavior of $L_{\ga,h}$ near the right side of  $\gF_{\ga,h}(\C{P}_{\ga,h})$. 
However, a large number of iterates of $F_{\ga,h}$, (about $\Re (1/\ga)$ near $+\B{i}\infty$), 
on the image of $H^1_{\ga,h}$ are needed to cover $\hat{\C{P}}_{\ga,h}$. 
For this reason, it is not possible to use the functional equation~\eqref{E:equivariant-L} and the uniform 
estimates on $F_{\ga,h}$ in Lemma~\ref{L:lift-asymptotes-a} to derive estimates on $L_{\ga,h}$ near the 
right hand side of  $\gF_{\ga,h}(\C{P}_{\ga,h})$. 
For this purpose we need to define an alternative quasi-conformal change of coordinate for the right-hand side.
But, the issue here is that there is no reference point similar to $v_{\ga, h}$ near the right hand side of 
$\hat{\C{P}}_{\ga,h}$.
However, as $F_{\ga,h}$ commutes with the translation by $1/\ga$, we expect that the behavior of 
$L_{\ga,h}$ near the left hand side and the right hand side of  $\gF_{\ga,h}(\C{P}_{\ga,h})$ should be similar 
in nature.  
This is investigated through the quasi-conformal change of coordinate 
\[H^2_{\ga,h}:\{ \gz\in \BB{C} \mid \Re \gz \in [0,1]\} \to \gT_\ga (D_5) \ci  \Dom F_{\ga,h}\] 
defined below. 

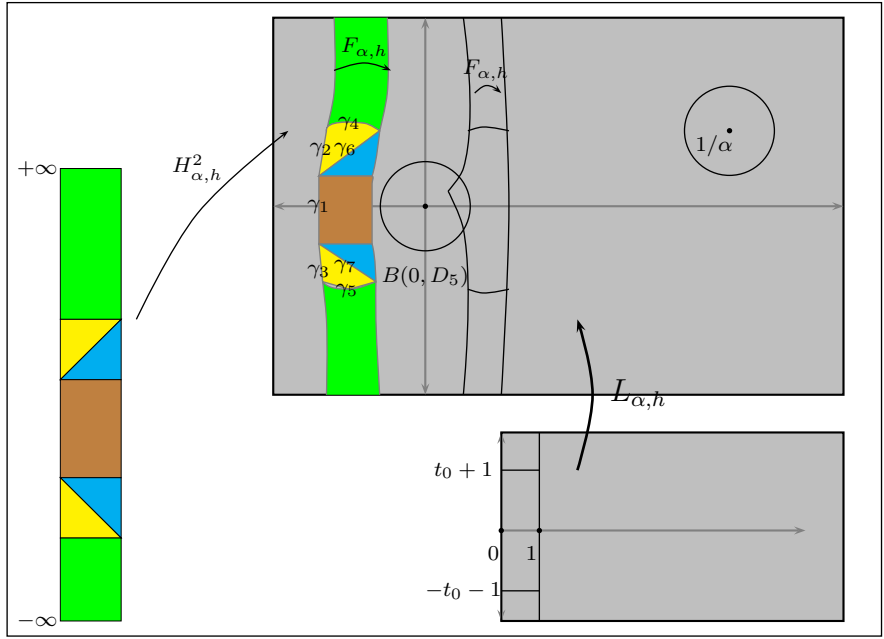
\begin{figure} \label{F:qc-coordinates-H2}
\begin{pspicture}(-.5,-0.2)(11,8.2)
\pspolygon[fillstyle=solid,fillcolor=lightgray](6,0)(10.5,0)(10.5,2.5)(6,2.5)
\psline[linecolor=gray]{->}(6,1.2)(10,1.2)
\psline[linecolor=gray]{<->}(6,0)(6,2.5)

\psline[linewidth=.5pt](6,0)(6,2.5)
\psline[linewidth=.5pt](6.5,0)(6.5,2.5)

\psdots[dotsize=2pt](6,1.2)(6.5,1.2)
\rput(5.9,.9){\tiny $0$} \rput(6.4,.9){\tiny $1$}
\psline[linewidth=.5pt](6,2)(6.5,2) \psline[linewidth=.5pt](6,.4)(6.5,.4)
\rput(5.5,2){\tiny $t_0+1$} \rput(5.5,.4){\tiny $-t_0-1$} 


\pspolygon[fillstyle=solid,fillcolor=lightgray](3,3)(10.5,3)(10.5,8)(3,8)
\psline[linecolor=gray]{<->}(3,5.5)(10.5,5.5)
\psline[linecolor=gray]{<->}(5,3)(5,8)

\rput(5,4.57){\tiny $B(0,D_5)$} \rput(8.8,6.3){\tiny $1/\ga$}

\psdots[dotsize=2pt](5,5.5)(9,6.5)
\pscircle[linewidth=.5pt](5,5.5){.6}\pscircle[linewidth=.5pt](9,6.5){.6}
\pscurve[linewidth=.5pt](5.5,8)(5.5,6)(5.3,5.7)
\pscurve[linewidth=.5pt](5.3,5.7)(5.5,5.2)(5.5,3)
\pscurve[linewidth=.5pt](6,8)(6.1,5.5)(6,3)

\pscurve[linewidth=.5pt]{->}(5.65,7)(5.8,7.1)(6,7) 
\rput(5.8,7.3){\tiny $F_{\ga,h}$}

\pscustom[linecolor=gray,linewidth=.5pt,fillstyle=solid,fillcolor=green]{
\pscurve(3.8,8)(3.8,7)(3.7,6.5)
\pscurve(3.7,6.5)(4.2,6.6)(4.4,6.5)
\pscurve[liftpen=1](4.4,6.5)(4.5,7)(4.5,8)}

\pscustom[linecolor=gray,linewidth=.5pt,fillstyle=solid,fillcolor=yellow]{
\psline(3.6,5.9)(3.7,6.5)
\pscurve(3.7,6.5)(4.2,6.6)(4.4,6.5)
\psline(4.4,6.5)(3.6,5.9)}

\pscustom[linecolor=gray,linewidth=.5pt,fillstyle=solid,fillcolor=cyan]{
\psline(4.4,6.5)(3.6,5.9)
\psline(3.6,5.9)(4.3,5.9)
\pscurve(4.3,5.9)(4.4,6.5)}

\pscustom[linecolor=gray,linewidth=.5pt,fillstyle=solid,fillcolor=brown]{
\psline(3.6,5.9)(4.3,5.9)
\pscurve(4.3,5.9)(4.3,5.8)(4.3,5)
\psline(4.3,5)(3.6,5)
\psline(3.6,5)(3.6,5.9)} 

\pscustom[linecolor=gray,linewidth=.5pt,fillstyle=solid,fillcolor=cyan]{
\psline(3.6,5)(4.3,5)
\pscurve(4.3,5)(4.35,4.5)
\psline(4.35,4.5)(3.6,5)}

\pscustom[linecolor=gray,linewidth=.5pt,fillstyle=solid,fillcolor=yellow]{
\psline(4.35,4.5)(3.6,5)
\psline(3.6,5)(3.65,4.5)
\psline(3.65,4.5)(4,4.4)(4.35,4.5)}

\pscustom[linecolor=gray,linewidth=.5pt,fillstyle=solid,fillcolor=green]{
\pscurve(3.7,3)(3.7,4)(3.65,4.5)
\pscurve(3.65,4.5)(4,4.4)(4.35,4.5)
\pscurve[liftpen=1](4.35,4.5)(4.35,4)(4.4,3)}

\pscurve[linewidth=.5pt]{->}(3.8,7.3)(4.2,7.4)(4.55,7.3) 
\rput(4.2,7.6){\tiny $F_{\ga,h}$}

\pscurve[linewidth=.5pt](5.56,6.5)(5.75,6.55)(6.1,6.5)
\pscurve[linewidth=.5pt](5.56,4.4)(5.75,4.35)(6.1,4.4)


\pspolygon[linewidth=.3pt,fillstyle=solid,fillcolor=green](.2,6)(1,6)(1,4)(.2,4)
\pspolygon[linewidth=.3pt,fillstyle=solid,fillcolor=yellow](.2,4)(1,4)(.2,3.2)
\pspolygon[linewidth=.3pt,fillstyle=solid,fillcolor=cyan](.2,3.2)(1,4)(1,3.2)
\pspolygon[linewidth=.3pt,fillstyle=solid,fillcolor=brown](.2,3.2)(1,3.2)(1,1.9)(.2,1.9)
\pspolygon[linewidth=.3pt,fillstyle=solid,fillcolor=cyan](.2,1.9)(1,1.9)(1,1.1)
\pspolygon[linewidth=.3pt,fillstyle=solid,fillcolor=yellow](.2,1.9)(1,1.1)(.2,1.1)
\pspolygon[linewidth=.3pt,fillstyle=solid,fillcolor=green](.2,1.1)(1,1.1)(1,0)(.2,0)

\rput(-0.1,0){\tiny $-\infty$}

\rput(-0.1,6){\tiny $+\infty$} 

\pscurve[linewidth=1pt]{->}(7,2)(7.2,3)(7,4)
\rput(7.8,3){$L_{\ga,h}$}

\rput(3.6,5.5){\tiny $\gga_1$}
\rput(3.63,6.25){\tiny $\gga_2$}
\rput(3.6,4.63){\tiny $\gga_3$}
\rput(3.95,6.25){\tiny $\gga_6$}
\rput(3.95,4.7){\tiny $\gga_7$}
\rput(4,6.6){\tiny $\gga_4$}
\rput(3.97,4.38){\tiny $\gga_5$}

\pscurve[linewidth=.5pt]{->}(1.2,4)(2,5.4)(3.2,6.5)
\rput(2,6){\tiny $H_{\ga,h}^2$}
\pspolygon[linewidth=.5pt](-.5,-0.2)(11,-0.2)(11,8.2)(-.5,8.2)
\end{pspicture}
\caption{Schematic presentation of the quasi-conformal change of coordinates $H^2_{\ga,h}$. 
It is defined as interpolations of some analytic and quasi-conformal mappings.}
\end{figure}

Recall that $H^1_{\ga,h}(0)=v_{\ga,h}$ and $|v_{\ga,h}|$ is uniformly bounded from above. 
Since $F_{\ga,h}$ is uniformly close to the translation by one, and $G^1_{\ga,h}$ is quasi-conformal with 
a uniform bound on its dilatation, independent of $\ga$ and $h$, 
there is a constant $t_0>0$ and a positive integer $i$ such that 
\[\forall \gx\in\{t_0\B{i},1+t_0\B{i}, -t_0\B{i}, 1-t_0\B{i}\}, \Re F_{\ga,h}^{-i}(L_{\ga,h}(\gx))
 \in [\Re (1/\ga)+D_5, -D_5],\]
and for all $\gx$ with $\Re \gx\in [0,1]$ and $|\Im \gx|\geq t_0$, $F_{\ga,h}^{-i}(L_{\ga,h}(\gx))$ is defined. 
Indeed, by making $t_0$ large enough, the four points $F_{\ga,h}^{-i}(L_{\ga,h}(t_0\B{i}))$, 
$F_{\ga,h}^{-i}(L_{\ga,h}(t_0\B{i}+1))$, $F_{\ga,h}^{-i}(L_{\ga,h}(-t_0\B{i}))$, and 
$F_{\ga,h}^{-i}(L_{\ga,h}(-t_0\B{i}+1))$ become arbitrarily close to the vertices of a parallelogram whose 
two sides converge to two horizontal segments of length one. 

Let us define the curves (see Figure~\ref{F:qc-coordinates-H2}) 
\begin{gather*}
\gga_1(t)=(\frac{t_0-t}{2t_0})F_{\ga,h}^{-i}(L_{\ga,h}(-t_0\B{i}))+ 
(\frac{t_0+t}{2t_0})F_{\ga,h}^{-i}(L_{\ga,h}(t_0\B{i})), \tfor t\in [-t_0, t_0];\\
\gga_2(t)= (1-t) F_{\ga,h}^{-i}(L_{\ga,h}(t_0\B{i})) + t F_{\ga,h}^{-i}(L_{\ga,h}((t_0+1)\B{i})), \tfor t\in [0,1]\\
\gga_3(t)= (1-t) F_{\ga,h}^{-i}(L_{\ga,h}(-t_0\B{i})) + t F_{\ga,h}^{-i}(L_{\ga,h}((-t_0-1)\B{i})), \tfor t\in [0,1]\\
\gga_4(s)= F_{\ga,h}^{-i}(L_{\ga,h}(s+ (t_0+1)\B{i})), \tfor s\in [0,1]; \\
\gga_5(s)= F_{\ga,h}^{-i}(L_{\ga,h}(s+ (-t_0-1)\B{i})), \tfor s\in [0,1]; \\
\gga_6(s)= (1-s) F_{\ga,h}^{-i}(L_{\ga,h}(t_0\B{i}))+s F_{\ga,h}^{-i}(L_{\ga,h}(1+ (t_0+1)\B{i})), \tfor s\in [0,1];\\
\gga_7(s)= (1-s) F_{\ga,h}^{-i}(L_{\ga,h}(-t_0\B{i}))+s F_{\ga,h}^{-i}(L_{\ga,h}(1+ (-t_0-1)\B{i})), \tfor s\in [0,1].
\end{gather*}
For $t_0$ large enough, independent of $\ga$ and $h$, the curves $\gga_4$, $\gga_5$, 
$F_{\ga,h}^{-i}(L_{\ga,h}([0,1]+t_0\B{i}))$ and $F_{\ga,h}^{-i}(L_{\ga,h}([0,1]-t_0\B{i}))$ are 
nearly horizontal. 
The curve $\gga_6$ has slope close to $1$, and the curve $\gga_7$ has slope close to $-1$. 
In particular, for large enough $t_0$, the curves $\gga_4$ and $\gga_6$ intersect only at their end points 
$F_{\ga,h}^{-i}(L_{\ga,h}(1+(t_0+1)\B{i}))$, wile the curves $\gga_5$ and $\gga_7$ intersect only at the point 
$F_{\ga,h}^{-i}(L_{\ga,h}(1-(t_0+1)\B{i}))$. 
The curves $\gga_6$ and $F_{\ga,h}^{-i}(L_{\ga,h}([0,1]+t_0\B{i}))$ only intersect at their staring points, 
and similarly the curves $\gga_7$ and $F_{\ga,h}^{-i}(L_{\ga,h}([0,1]- t_0\B{i}))$ intersect only at their 
starting points.

Define the map $H^2_{\ga,h}(s+i\B{t})$, for $s\in [0,1]$ and $t\in \BB{R}$, as follows:
\begin{equation*}
\begin{cases}
F_{\ga,h}^{-i}(L_{\ga,h}(s+ t\B{i})), & \tif |t| \geq t_0+1 \\
(1-s) \gga_1(t) + s F_{\ga,h}(\gga_1(t)),  & \tif |t| \leq t_0 \\
(\frac{t-t_0-s}{1-s}) \gga_4(s) + \frac{t_0-t+1}{1-s} \gga_6(s), & \tif s\in [0,1], t\in [t_0+s, t_0+1]\\
(\frac{s-1}{t-t_0-1})\gga_6(t-t_0)+\frac{t-t_0-s}{t-t_0-1} F_{\ga,h}(\gga_2(t)), & \tif t\in [t_0,t_0+1],s\in[t-t_0,1]\\                        
(\frac{-t-t_0-s}{1-s}) \gga_5(s) + \frac{t_0+t+1}{1-s} \gga_7(s), & \tif s\in [0,1], t\in [-t_0-s, -t_0-1]\\
(\frac{s-1}{t-t_0-1})\gga_7(t-t_0)+\frac{t-t_0-s}{t-t_0-1} F_{\ga,h}(\gga_3(t)), & \tif t\in [-t_0,-t_0-1],s\in [t-t_0,1].
\end{cases}
\end{equation*}
It follows from the above definition that
\[H^2_{\ga,h}(1+ t\B{i})= F_{\ga,h}(H^2_{\ga,h}(t\B{i})), \forall t\in \BB{R}.\]
Moreover, from the construction, one can see that the following lemma holds. 

\begin{lem}\label{L:qc-coordinate-2}
The map $H^2_{\ga,h}$ is quasi-conformal on the strip $\Re \gz \in [0,1]$ and the size of its dilatation 
$|\partial_{\ol{\gz}} H^2_{\ga,h} / \partial_\gz H^2_{\ga,h}|$ is uniformly bounded from above at almost every 
point on this strip. 
Moreover, there is a constant $D_{12}$ such that for all $\ga, \gb$ in $A(r_3')$, all $h$ in 
$\IS \cup \{Q_\ga\}$, and all $\gz$ with $\Re \gz\in [0,1]$ we have 
\[\big |H^2_{\ga,h}(\gz) - H^2_{\gb,h}(\gz)\big | \leq D_{12} |\ga-\gb|.\] 
\end{lem}

For $c\in \BB{C}$, we denote the translation by $c$ on the complex plane with   
\[T_c(w)=w+c.\]  
Since $F_{\ga,h}$ is periodic of period $1/\ga$, the map the quasi-conformal mapping 
$T_{1/\ga} \circ H^2_{\ga,h}$ also conjugates the translation by one to the action of $F_{\ga,h}$. 
We shall compare the map $L_{\ga,h}$ to this coordinate by studying the map 
\[G^2_{\ga,h}=  T_{-1/\ga}  \circ   L_{\ga,h}^{-1} \circ  T_{1/\ga}   \circ H^2_{\ga,h}: 
\{\gz\in \BB{C} \mid \Re \gz \in [0,1]\} \to \BB{C}.\] 
By the periodicity of $F_{\ga,h}$ and the functional equations for $L_{\ga,h}$ and $H^2_{\ga,h}$, we must 
have $G^2_{\ga,h}(\gz+1)= G^2_{\ga,h}(\gz)+1$, whenever $\Re \gz=0$. 
We shall use this map to analyze the conformal change of coordinate $L_{\ga,h}$ near the right hand side of 
its domain of definition. 
By the definition of this map, $G^2_{\ga,h}$ is quasi-conformal with 
$|\partial_{\ol\gz} G^2_{\ga,h}/ \partial_\gz G^2_{\ga,h}|$ uniformly bounded from above. 
Moreover, it is normalized by making
\begin{equation}\label{E:G2-normalization}
\lim_{\Im \gz \to +\infty} |\Im (G^2_{\ga,h}(\gz)- \gz)| = 0.
\end{equation} 
The above normalization, and the uniform bound on the dilatation of the map allows us to prove a 
uniform bound on the dependence of this map on $\ga$. 

\begin{lem}\label{L:dependence-G-ga}
For every $D'_{13}>0$ there is a constant $D_{13}$ such that for all $\ga, \gb$ in $A(r_3')$ and all $\gz$ 
with $\Re \gz \in [0,1]$ and $|\Im \gz|\leq D_{13}'$ we have 
\begin{itemize}
\item[a)] \[|G_{\ga,h}^1  \circ (G^1_{\gb,h})^{-1}(\gx) -\gx | \leq D_{13} |\ga-\gb|.\]
\item[b)] \[|G_{\ga,h}^2 \circ (G^2_{\gb,h})^{-1}(\gx) -\gx | \leq D_{13} |\ga-\gb|.\] 
\end{itemize}
\end{lem}

\begin{proof}
{\em Part a)} 
Let $\gm_{\ga,h}$ denote the complex dilation of the map $G^1_{\ga,h}$. 
By Lemma~\ref{L:qc-coordinate-2}, $|\gm_{\ga,h}|$ is uniformly bounded from above, independent of $\ga$ and 
$h$, by a constant $<1$.  
The complex dilation of the composition $G_{\ga,h}^1  \circ (G^1_{\gb,h})^{-1}$ at $\gz$ 
is given by the formula
\[\frac{\gm_{\ga,h} -\gm_{\gb,h}}{1- \gm_{\gb,h} \ol{\gm_{\ga,h}}} 
\Big (\frac{\partial _\gz G^1_{\gb,h} (\gz)}{|\partial _\gz G^1_{\gb,h} (\gz)|}\Big)^2.\]
On the other hand, as $G^1_{\ga,h}=L_{\ga,h}^{-1} \circ H^1_{\ga,h}$ and $L_{\ga,h}$ is holomorphic, 
by the chain rule, the complex dilation of $G^1_{\ga,h}$ is equal to the complex dilatation of $H^1_{\ga,h}$. 
Combining with the above equation, and the formulas in Equation~\eqref{E:first-partials-H}, we conclude that 
the size of the complex dilation of $G_{\ga,h}^1  \circ (G^1_{\gb,h})^{-1}$ is bounded from above by a 
uniform constant times $|\ga-\gb|$. 

Recall that $G^1_{\ga,h}(0)= i_{\ga,h}$, and $i_{\ga,h}$ is locally constant. 
Thus, for $\ga$ and $\gb$ sufficiently close, $G_{\ga,h}^1  \circ (G^1_{\gb,h})^{-1}(i_{\ga,h})=i_{\ga,h}$. 
By the classical results on the dependence of the solution of the Beltrami equation on the Beltrami coefficient, 
see \cite[Section 5.1]{AhBe60}, $|G_{\ga,h}^1  \circ (G^1_{\gb,h})^{-1}(\gx) -\gx|$, for $\gx$ in a compact set, 
is bounded from above by a uniform constant times $|\ga-\gb|$. 
This finishes the proof of the first part. 

\medskip

{\em Part b)} The proof is similar to the one given for Part a, except that we use the normalization of the maps 
at infinity instead; Equation~\eqref{E:G2-normalization}. 
\end{proof}


\subsection{Dependence of the Fatou coordinate on the linearity}\label{SS:dependence-L-alpha}

\begin{propo}\label{P:dependence-of-L_ga-a}
For all $D_{14}'>0$ there exists a constant $D_{14}$, independent of $\ga$ and $h$,  such that 
\begin{itemize}
\item[a)] for all $\gx$ in $\Dom L_{\ga,h} \cap B(0, D_{14}')$ 
\[|\frac{\partial}{\partial \ga} L_{\ga,h} (\gx)|\leq D_{14},\]
\item[b)] for all $w$ in $\Dom L_{\ga,h}^{-1} \cap B(0, D_{14}')$,  
\[|\frac{\partial}{\partial \ga} L^{-1}_{\ga,h} (w)|\leq D_{14} .\]
\item[c)] for all $\gx$ in $(\Dom (L_{\ga,h}) - 1/\ga) \cap B(0, D_{14}')$, 
\[|\frac{\partial}{\partial \ga} (T_{-1/\ga} \circ L_{\ga,h} \circ T_{1/\ga})(\gx)|\leq D_{14}.\]
\end{itemize}
\end{propo}

\begin{proof}
{\em Part a)} We have 
\begin{align*}
|L_{\ga,h}(\gx) - L_{\gb,h}(\gx)| 
&\leq |H^1_{\ga,h} \circ (G^1_{\ga,h})^{-1}(\gx) - H^1_{\ga,h} \circ (G^1_{\gb,h})^{-1}(\gx)| \\
& \qquad + |H^1_{\ga,h} \circ (G^1_{\gb,h})^{-1}(\gx) - H^1_{\gb,h} \circ (G^1_{\gb,h})^{-1}(\gx)| \\
& \leq \sup_{z} |D H^1_{\ga,h}|  \cdot \sup_{\gx} |D (G^1_{\ga,h})^{-1}(\gx)| \cdot 
|\gx -  G_{\ga,h}^1 (G^1_{\gb,h})^{-1}(\gx)| \\
& \qquad + |H^1_{\ga,h}(z)-H^1_{\gb,h}(z)| \\
& \leq C |\ga-\gb|.
\end{align*}
In the last line of the above equation  we have used the uniform bound in Lemma~\ref{L:dependence-G-ga}, 
and a uniform bound on the dependence of $H^1_{\ga,h}$ on $\ga$. 
The latter bound is obtained from the definition of the map $H^1_{\ga,h}$ and the uniform bound 
on the dependence of the map $F_{\ga,h}$ and the point $v_{\ga,h}$ on $\ga$ obtained in 
Lemmas~\ref{L:lift-dependence-a} and \ref{L:v-dependence}. 

\medskip

{\em Part b)} With $w= L_{\gb,h}(\gx)$, we have  
\begin{align*}
|L_{\ga,h}^{-1}(w) - L_{\gb,h}^{-1}(w)| &= |L_{\ga,h}^{-1}(L_{\gb,h}(\gx)) - L_{\ga,h}^{-1}(L_{\ga,h}(\gx))| \\
& \leq  \sup_{w} |(L_{\ga,h}^{-1})'(w)| \cdot |L_{\ga,h}(\gx) - L_{\gb,h}(\gx)|  \\
& \leq C' C |\ga-\gb|. 
\end{align*}
In the above inequalities, the uniform bound on $|(L_{\ga,h}^{-1})'(w)|$, when $w$ is restricted to $B(0,D_{14}')$, 
may be obtained from the pre-compactness of the class of maps $\IS$, and the continuous dependence 
of $L_{\ga,h}$ on $\ga$ and $h$. 
The constant $C$ is the one introduced in the proof of Part a). 

\medskip

{\em Part c)} 
The argument here is similar to the one in part a) withe difference that we use the maps $H^2_{\ga,h}$ and 
$G^2_{\ga,h}$. 
That is, with $w=H^2_{\ga,h}(\gz)$, we have
\begin{align*}
|(T_{-1/\ga} \circ L_{\ga,h} \circ T_{1/\ga})(\gx) & - (T_{-1/\gb} \circ L_{\gb,h} \circ T_{1/\gb})(\gx)| \\
& = |H^2_{\ga,h} \circ (G^2_{\ga,h})^{-1}(\gx) - H^2_{\gb,h} \circ (G^2_{\gb,h})^{-1}(\gx)| \\
&\leq |H^2_{\ga,h} \circ (G^2_{\ga,h})^{-1}(\gx) - H^2_{\ga,h} \circ (G^2_{\gb,h})^{-1}(\gx)|  \\
& \qquad + |H^2_{\ga,h} \circ (G^2_{\gb,h})^{-1}(\gx) - H^2_{\gb,h} \circ (G^2_{\gb,h})^{-1}(\gx)| \\
& \leq \sup_{z} |D H^2_{\ga,h}|  \cdot \sup_{\gx} |D (G^2_{\ga,h})^{-1}(\gx)| \cdot 
|\gx -  G_{\ga,h}^2 (G^2_{\gb,h})^{-1}(\gx)| \\
& \qquad + |H^2_{\ga,h}(z)-H^2_{\gb,h}(z)| \\
& \leq C |\ga-\gb|,
\end{align*}
for some constant $C$, independent of $\ga$ and $h$.
In the last line of the above inequalities we have used the uniform bounds in Lemas~\ref{L:bounded-derivative-L},
\ref{L:dependence-G-ga}.  
\end{proof}


\subsection{Geometry of the petals}\label{SS:petal-geometry}

\begin{proof}[Proof of Proposition~\ref{P:wide-petals}]
Let $r_3$ be the constant $r_3'$ introduced in Lemma~\ref{L:lift-asymptotes-a}. 
We shall continue to use the notation $\ga\ltimes h$ for the maps $f$ in the class $\PC{A^+(r_3')}$, 
introduced in Equation~\eqref{E:notation-ltimes}.
Recall the covering map $\gta_{\ga,h}$ defined in \eqref{E:tau-covering}, the lift $F_{\ga,h}$ of $\ga\ltimes h$ 
defined in Equation~\ref{E:lift-formula}, and the univalent map $L_{\ga,h}$ which conjugates $F_{\ga,h}$ 
to the translation by one. 
Then, $\gF_{\ga,h}^{-1}$ is the same as the composition $\gta_{\ga,h} \circ L_{\ga,h}$.

Let us define $x_{\ga,h}$ as  the supremum of the set of $x\geq 0$ such that $L_{\ga,h}$ has a univalent 
extension onto the set $(0, x)+\B{i} \BB{R}$, and $L_{\ga,h}$ maps this infinite strip into 
$\gS_2 \cup \hat{\C{P}}_{\ga,h}$, where $\gS_2$ is defined before Lemma~\ref{L:bounded-derivative-L} 
and $\hat{\C{P}}_{\ga,h}$ is the lift of $\C{P}_{\ga,h}$ separating $0$ from $1/\ga$.  
By Proposition~\ref{P:Fatou-coordinates}, $x_{\ga,h} \geq 2$. 
Also, $L_{\ga,h}(x_{\ga,h}+ \B{i} \BB{R})$ intersects the right hand side boundary of $\gS_2$ at some point 
whose imaginary part is uniformly close to $\Im (1/\ga)$. 

Consider the sets 
\[B_1=\{\gx \in \BB{C} \mid \Re \gx \in [0,1]\}, B_2= \{\gx\in \BB{C} \mid \Re \gx \in [x_{\ga,h}-1, x_{\ga,h}]\}.\] 
We aim to show that the curve $L_{\ga,h}(x_{\ga,h}+\B{i}\BB{R})$ is within a uniformly bounded distance 
from a translation of the curve $L_{\ga,h}(\B{i}\BB{R})$. 
Next we show that the translation constant is $\Re (1/\ga)$. 
As $F_{\ga,h}$ tends to the translation by one near $+\B{i}\infty$, the functional equation \eqref{E:equivariant-L}
implies that $x_{\ga,h}$ is uniformly close to $\Re (1/\ga)$. 

There is a constant $\gh>0$ such that every $\gx\in B_1$ with $|\gx|\geq \gh$,
$L_{\ga,h}(\gx)$ belongs to $\Sigma_2-1/\ga$.
Indeed, by the pre-compactness of the class $\IS$, $\gh$ may be chosen independent of $\ga$ and $h$. 

The uniform estimate in Lemma~\ref{L:lift-asymptotes-a} also hold on $\Sigma_2-1/\ga$, since $F_{\ga,h}$ is 
periodic of period $1/\ga$. 
This implies that for every $L_{\ga,h} (\gx)$ with $|\Im \gx|\geq \gh$ and $\Re \gx\in [0,1]$,  
there is $j_\gx \in \BB{Z}$ with $F_{\ga,h}\co{j_\gx}(L_{\ga,h}(\gx)) \in L_{\ga,h}(B_2)-1/\ga$.  
For $\gx\in B_1$ with $|\gx| \geq \gh$, define the map  
\[  H(\gx)= L_{\ga,h}^{-1} \circ  T_{1/\ga} \circ F_{\ga,h} \co{j_\gx} \circ L_{\ga,h}(\gx).\] 
The map $H$ may have discontinuities on its domain of definition, but since it commutes with the translation 
by one, it induces a continuous map from the top and bottom ends of the cylinder $B_1/ \BB{Z}$ to
the cylinder  $B_2/\BB{Z}$. 
By the pre-compactness of the class of maps $F_{\ga,h}$, applied on a compact neighborhood of  $0$, 
the map $H$ may be extended to a quasi-conform mapping from $B_1/ \BB{Z}$ to $B_2/ \BB{Z}$, whose 
complex dilatation is uniformly bounded away from the unit circle. 
Comparing the asymptotic expansions of the maps, near the top end $\Im H$ is asymptotic to the 
translation by $\Im (1/\ga)$. 
Note that since $H$ is conformal near the two ends of the cylinder, it maps every vertical line in $B_1/\BB{Z}$ 
to a curve in $B_2/\BB{Z}$, going from one end to the other, and spirals around the cylinder by a uniformly 
bounded amount. 
This implies that the lift of $H$ to a map from $\BB{C}$ to $\BB{C}$ is uniformly close to a translation. 
 
By the above paragraph, the map $T_{1/\ga} \circ F_{\ga,h}\co{j_\gx}$ is uniformly close to a translation, as 
a map from $L_{\ga,h}(B_1)$ to $L_{\ga,h}(B_2)$. 
Thus, $F_{\ga,h}\co{j_\gx}$ from $L_{\ga,h}(B_1)$ to $T_{-1/\ga} \circ L_{\ga,h}(B_2)$ must be close to a 
translation. 
How ever, since $|j_{\gx}|$ is uniformly bounded form above for when $\Im \gx=\gh$, and each iterate of 
$F_{\ga,h}$ is uniformly close to the translation by one, $F_{\ga,h}\co{j_\gx}$ must be uniformly close to 
the identity map. 
This implies that, $L_{\ga,h}(B_2)$ is uniformly close to $T_{1/\ga} \circ L_{\ga,h}(B_1)$. 

By Proposition~\ref{P:asymptote-L-top}, the sets $L_{\ga,h}(B_1)$ to $L_{\ga,h}(B_2)$ are asymptotically 
vertical, and also $F_{\ga,h}$ tends to the translation by one near $+\B{i} \infty$. 
Thus, the number iterates by $F_{\ga,h}$ required to go from $L_{\ga,h}(B_1)$ to $L_{\ga,h}(B_2)$ must be 
uniformly close to $\Re (1/\ga)$. 
By the functional equation~\eqref{E:equivariant-L}, $x_{\ga,h}$ is uniformly close to $\Re (1/\ga)$. 

Recall that $\gta_{\ga,h}$ is periodic of period $1/\ga$.
Finally, since $L_{\ga,h}(x_{\ga,h}+\B{i}\BB{R})$ is uniformly close the $L_{\ga,h}(\B{i}\BB{R})+ \Re (1/\ga)$, 
by subtracting a uniformly bounded number from $x_{\ga,h}$, if necessary, we may assume that 
$\gta_{\ga,h}$ is univalent on $L_{\ga,h}((0, x_{\ga,h})+ \B{i} \BB{R})$. 
Thus, the composition $\gta_{\ga,h} \circ L_{\ga,h}$ is univalent on the set $(0, x_{\ga,h})+ \B{i} \BB{R})$. 
This finishes the proof of the proposition. 
\end{proof}

\begin{proof}[Proof of Proposition~\ref{P:bounded-spirals} -- Part a)]
We continue to use the notation $\ga\ltimes h$ for the maps $f$ in $\PC{A^+(r_3')}$. 
Recall that $r_3$ is the constant $r_3'$ obtained in Lemma~\ref{L:lift-asymptotes-a}. 
We shall use the decomposition of $\gF_{\ga,h}^{-1}$ as $\gta_{\ga,h} \circ L_{\ga,h}$. 

Let $\gx=\gx_1+\B{i} \gx_2$, and $w=w_1+\B{i}w_2= L_{\ga,h}(\gx)$, with $\gx_1$, $\gx_2$, $w_1$, and $w_2$ 
in $\BB{R}$. 
By Proposition~\ref{P:asymptote-L-top}, for a fixed $\gx_1 \in (0, \Re \frac{1}{\ga}-\B{k})$, 
as $\gx_2$ tends to $+\infty$, $w_1$ tends to a finite constant say $w_1'$, and $\gx_2 - w_2$ tends to a 
finite constant say $w_2'$. 
Indeed, we have 
\[|\gx_1-w_1'| \leq D_{11} (1- \log |\ga|) , |w_2'| \leq D_{11} (1-\log |\ga|), \]
where $D_{11}$ is independent of $\ga$ and $h$. 
We calculate the limit as in 
\begin{align*}
\lim_{\gx_2\to +\infty}& \big (\arg (\gta_{\ga,h} \circ L_{\ga,h}(\gx)) + 2\pi \gx_2 \Im \ga \big) \\
&= \arg \gs_{\ga,h} 
+ \lim_{\gx_2 \to +\infty} \big (\arg \frac{1}{1- e^{-2\pi \B{i} \ga  L_{\ga,h}(\gx)}} + 2 \pi w_2 \Im \ga \big ) 
+ \lim_{\gx_2 \to +\infty} 2\pi (\gx_2 - w_2)\Im \ga  \\
&= \arg \gs_{\ga,h} +
 \lim _{\gx_2 \to +\infty} \big ( -2\pi w_2 \Im \ga + 2\pi w_1 \Re \ga  + 2\pi w_2 \Im \ga \big) 
 + 2\pi w_2' \Im \ga \big) \\
& = \arg \gs_{\ga,h} + 2 \pi w_1' \Re \ga  + 2\pi w_2' \Im \ga \\
& = \arg \gs_{\ga,h} + 2\pi \gx_1 \Re \ga + 2\pi (w_1' -\gx_1) \Re \ga + 2\pi w_2' \Im \ga  . 
\end{align*}
\end{proof}
Above, we have used that as $\gx_2 \to +\infty$, the size of $e^{-2\pi \B{i} \ga  L_{\ga,h}(\gx)}$ tends to 
$+\infty$.  
Hence, $e^{-2\pi \B{i} \ga  L_{\ga,h}(\gx)}$ and $1- e^{-2\pi \B{i} \ga  L_{\ga,h}(\gx)}$ must have the same argument 
in the limit. 

\begin{proof}[Proof of Proposition~\ref{P:bounded-iterates-remaining}-- Part a)]
Recall the set $S_{\ga,h}^t$ defined in Section~\ref{SS:renormalization}. 
By the definition, $\gF_{\ga,h}(S_{\ga,h}^t)$ is contained in the set 
$\{\gx\in \BB{C} /mid \Re \gx \in (0, \Re \frac{1}{\ga}-\B{k})\}$. 
First we note that the projection of this set onto the real line must have uniformly bounded diameter, independent 
of $\ga$ and $h$. 
That is because, by Theorem~\ref{T:Ino-Shi2} and the Koebe distortion theorem, the set of maps 
$\nprt{1}(\ga\ltimes h)$, over all $\ga \in A^+(r_3)$ and $h\in \IS \cup \{Q_0\}$, forms a compact class of map. 
In particular, the pre-image of a straight ray landing at $0$ under any of these maps, spirals at most a uniformly 
bounded number of times about $0$. 
Lifting this property by $\ex^t$, we conclude that $\gF_{\ga,h}(S_{\ga,h}^t)$ must have a uniformly bounded 
horizontal width. 
As $k_{\ga,h}^t$ is chosen as the smallest positive integer satisfying Proposition \ref{P:renormalization-top}, 
$\gF_{\ga,h}(S_{\ga,h}^t)$ must be contained in the set 
\[\{\gx\in \BB{C} \mid \Re \gx \in (\Re \frac{1}{\ga}-\B{k}- \gd, \Re \frac{1}{\ga}-\B{k})\},\] 
for some $\gd$ independent of $\ga$ and $h$. 

On the other hand, $k_{\ga,h}^t$ is the number of iterates by $F_{\ga,h}$ required to go from 
$L_{\ga,h} \circ \gF_{\ga,h}(S_{\ga,h}^t)$ to $L_{\ga,h}(\{\gx \in \BB{C}/mid \Re \gx \in [1/2,3/2]\})+1/\ga$. 
By Proposition \ref{P:asymptote-L-top}, $L_{\ga,h} \circ \gP_{\ga,h}(\C{P}_{\ga,h})$ is bounded by two curves 
that are asymptotically vertical near the top end. 
Sine $L_{\ga,h}'$ tends to $+1$ near the top, see Lemma \ref{L:bounded-derivative-L}, the width of the top end 
of $L_{\ga,h} \circ \gP_{\ga,h}(\C{P}_{\ga,h})$ tends to $\Re \frac{1}{\ga}-\B{k}$. 
By the same lemma, $L_{\ga,h}(\gF_{\ga,h}(S_{\ga,h}^t))$ is contained within uniformly bounded distance from the 
left side of $L_{\ga,h} \circ \gP_{\ga,h}(\C{P}_{\ga,h})$. 
The lift $F_{\ga,h}$ is uniformly close to the translation by one. 
Thus, the number of iterates by $F_{\ga,h}$ required to go from $L_{\ga,h} \circ \gP_{\ga,h}(S_{\ga,h})$ 
to $L_{\ga,h}(\{\gx \in \BB{C}/mid \Re \gx \in [1/2,3/2]\})+ 1/\ga$ is uniformly bounded from above.
That is, $k_{\ga,h}^t$ is uniformly bounded from above, independent of $\ga$ and $h$. 
\end{proof}

We shall prove the other half of the proposition, the uniform bound on $k_f^b$ at the end of 
Section~\ref{SS:2-1-derivative-b}.  


\subsection{Ecale map and its dependence on $\ga$}\label{SS:ecale-maps}
Recall the notation $h_\ga=\ga\ltimes h$, as well as let $S_{\ga,h}^t$ and $S_{\ga,h}^b$ denote the sectors 
$S_{h_\ga}^t$ and $S_{h_\ga}^b$, respectively, defined in Section~\ref{SS:renormalization}.  
Similarly, let $k_{\ga,h}^t$ denote the positive integer introduced in Propositions~\ref{P:renormalization-top}, 
for the map $f=h_\ga$. 
The map  
\[E^t_{\ga,h}= \gF_{\ga,h} \circ h_\ga \co{k_{\ga,h}^t} \circ \gF_{\ga,h}^{-1}: \gF_{\ga,h}(S_{\ga,h}^t) 
\to \gF_{\ga,h}(\C{P}_{\ga,h}),\]
induces, via the projection $\ex^t(\gx)=(-4/27) e^{2\pi \B{i}\gx}$, the renormalization $\nprt{1}(\ga\ltimes h)$. 
Recall the domains $V  \Subset U$ introduced in Section~\ref{S:Near-Parabolic}. 
By Theorem~\ref{T:Ino-Shi2}, $\nprt{1}(\ga\ltimes h)$ has a restriction to a domain that belongs to the class 
$\PC{\{-1/\ga\}}$.   
With the notations in Equation~\eqref{E:renormalization-notation}, this implies that 
\[ e^{2\pi \B{i}/\ga} \cdot \hat{\gp}_{\ga,h}(V)  \subseteq \ex^t(\gF_{\ga,h}(S_{\ga,h}^t)).\]
By Theorem~\ref{T:Ino-Shi2}, $\hat{\gp}_{\ga,h}$ has univalent extension onto $U$. 
Let $V'$ be an arbitrary Jordan neighborhood of $0$, cf.\ Proposition~\ref{P:changes-non-linearity}, such that 
\begin{equation}\label{E:intermediate-domain}
V \Subset V' \Subset U.
\end{equation} 
The set $e^{2\pi \B{i}/\ga} \cdot \hat{\gp}_{\ga,h}(V')$ may, or may not, contain $\ex^t(\gF_{\ga,h}(S_{\ga,h}^t))$. 

By the above paragraph, there is a connected set $X_{\ga,h} \subset \gF_{\ga,h}(\C{P}_{\ga,h})$, 
that is equal to the set $\gF_{\ga,h}(S_{\ga,h}^t)$ above some vertical line, and projects under $\ex^t$ onto $V'\setminus \{0\}$. 
Moreover, the map $E_{\ga,h}^t$ has holomorphic extension onto $X_{\ga,h}$ with 
$E_{\ga,h}(X_{\ga,h}) \subset \gF_{\ga,h}(\C{P}_{\ga,h})$.  
   
In the pre-Fatou coordinate, $E_{\ga,h}: X_{\ga,h}\to \BB{C}$ corresponds to the map 
\begin{gather*}
I^t_{\ga,h}= L_{\ga,h} \circ E^t_{\ga,h} \circ L_{\ga,h}^{-1}: L_{\ga,h}(X_{\ga,h}) 
\to \Dom F_{\ga,h} .
\end{gather*}
A key point here is that $I_{\ga,h}$ is given by a uniformly bounded number of iterates of 
$F_{\ga,h}$ plus a translation.
This is stated in the next lemma.  

\begin{lem}\label{L:alternative-definition-Ecale-a}
There exists a constant $D_{15}>0$ such that for all $\ga\in A^+(r_3')$ and all $h\in \IS\cup \{Q_0\}$, 
we have the following: 
\begin{itemize}
\item[a)] for all $w \in L_{\ga,h}(X_{\ga,h})$, 
\begin{gather*} 
I^t_{\ga,h}(w)=F_{\ga,h} \co{k_{\ga,h}^t}(w) - \frac{1}{\ga}.
\end{gather*} 
\item[b)] for all $w \in L_{\ga,h}(X_{\ga,h})$, 
\[\Big |\frac{\partial F_{\ga,h}\co{k_{\ga,h}^t}}{\partial \ga} (w) \Big | \leq D_{15}.\]
\end{itemize}
\end{lem}

\begin{proof}
Recall the covering map $\gta_{\ga,h}= \gF_{\ga,h}^{-1} \circ L_{\ga,h}^{-1}$. 
For all $w\in L_{\ga,h}(X_{\ga,h})$, we have 
\begin{align*}
\gta_{\ga,h} \circ I_{\ga,h}^t(w)&
=h_\ga \co{k_{\ga,h}^t} \circ \gF_{\ga,h}^{-1} \circ L_{\ga,h}^{-1}(w) \\
&= h_\ga \co{k_{\ga,h}^t}  \circ \gta_{\ga,h} (w) \\
&= \gta_{\ga,h} \circ F_{\ga,h} \co{k_{\ga,h}^t}(w). 
\end{align*}
Since $\gta_{\ga,h}$ is $1/\ga$-periodic, the above equality implies that the difference between 
$I_{\ga,h}^t(w)$ and $F_{\ga,h} \co{k_{\ga,h}^t}(w)$ is equal to a constant in $\BB{Z}/\ga$, 
and the value of the constant is independent of $w$. 
However, since $F_{\ga,h}$ is asymptotically equal to the translation by one near $+\B{i}\infty$,
$F_{\ga,h} \co{k_{\ga,h}^t}(w)$ belongs to $\hat{\C{P}}_{\ga,h}+1/\ga$. 
Thus, the difference is equal to $1/\ga$. 
This finishes the proof of the first part of the lemma. 

The positive integer $k_{\ga,h}^t$ is uniformly bounded from above independent of $\ga$ and $h$. 
This part of the Proposition~\ref{P:bounded-iterates-remaining} is proved earlier. 
On the other hand, by Lemma~\ref{L:lift-dependence-a}, $\partial_{\ga} F_{\ga,h}$ is uniformly bounded from 
above on $\Dom F_{\ga,h}$, and by lemma~\ref{L:lift-asymptotes-a}, $|\partial_w F_{\ga,h}|$ is also uniformly 
bounded from above on $L_{\ga,h}(X_{\ga,h})$.  
This implies the second part of the lemma. 
\end{proof}

\begin{proof}[Proof of Proposition \ref{P:changes-non-linearity}- Part a]
Let $V'$ be a Jordan neighborhood of $0$ satisfying Equation~\eqref{E:intermediate-domain} and 
assume $X_{\ga,h}$ is the lift of $V'$ defined in the paragraph after Equation~\eqref{E:intermediate-domain}. 

For $w$ in $X_{\ga,h}-1/\ga$ we have 
\begin{align*}
E^t_{\ga,h}\circ T_{1/\ga}
&=L_{\ga,h}^{-1} \circ I_{\ga,h}^t \circ L_{\ga,h} \circ T_{1/\ga}  \\
&=L_{\ga,h}^{-1} \circ T_{-1/\ga} \circ F_{\ga,h}^{k_{\ga,h}^t} \circ L_{\ga,h} \circ T_{1/\ga}  \\
&=L_{\ga,h}^{-1} \circ F_{\ga,h}^{k_{\ga,h}^t} \circ T_{-1/\ga}\circ L_{\ga,h} \circ T_{1/\ga}. 
\end{align*} 

Let us fix a constant $C>0$ large enough such that $\ex^t\{w \in X_{\ga,h} \mid |\Im w| \leq C \}$ 
contains the annulus $V' \setminus V$. 
The existence of a uniform $C$, independent of $\ga$ and $h$, is guaranteed by the Koebe distortion 
Theorem applied to the map $\hat{\gp}_{\ga,h}: U \to \BB{C}$. 
Now, assume that $w$ belongs to $X_{\ga,h}-1/\ga$, and $|Im (w-1/\ga)| \leq C$.  
Then, by Proposition~\ref{P:dependence-of-L_ga-a}, there is a constant $D_{14}$, depending only on $C$,   
such that $|\partial (T_{-1/\ga}\circ L_{\ga,h} \circ T_{1/\ga})(w)/\partial \ga|$ is uniformly bounded from above. 
By Lemma~\ref{L:alternative-definition-Ecale-a}, $|\partial F_{\ga,h}^{k_{\ga,h}^t}|/\partial \ga$ is 
uniformly bounded from above. 
Also, since $k_{\ga,h}^t$ is uniformly bounded from above, see Proposition~\ref{P:bounded-iterates-remaining}, 
the iterates $F_{\ga,h}^{k_{\ga,h}^t}$ displace a point by a uniformly bounded amount. 
Thus, we may apply Proposition~\ref{P:dependence-of-L_ga-a}, with a constant $D_{14}'$ depending only on 
$C$ and the uniform bound on $k_{\ga,h}^t$, to conclude that $|\partial L_{\ga,h}^{-1}/\partial \ga|$ is 
uniformly bounded from above at $F_{\ga,h}^{k_{\ga,h}^t} \circ T_{-1/\ga}\circ L_{\ga,h} \circ T_{1/\ga}(w)$. 
Combining these argument we conclude that for every $w$ in $X_{\ga,h}$ with $|\Im w| \leq C$, 
\[\Big |\frac{\partial}{\partial \ga} (E_{\ga,h}^t \circ T_{\ga,h}(w))\Big |\]
is uniformly bounded from above. 
The map $E_{\ga,h} \circ T_{1/\ga}$ projects via $\ex^t$ to the map $\hat{\gp}_{\ga,h}$. 
Therefore, $|\partial  \hat{\gp}_{\ga,h} /\partial \ga|$ must be uniformly bounded from above on $V'\setminus V$. 
By the maximum principle, this it must be uniformly bounded from above on $V'$. 
\end{proof} 


\subsection{Analysis of  the bottom NP-renormalization}\label{SS:2-1-derivative-b}
For $h\in \IS\cup \{Q_0\}$ and $\ga\in A^+(r_3)$, we continue to use the notation $(\ga \ltimes h)$ 
for the map $h_\ga$ defined as $h_\ga(z)= h(e^{2\pi i \ga} z)$. 
By proposition~\ref{P:Fatou-coordinates}, there is a Jordan domain, $\C{P}_{\ga,h}$ and a conformal 
change of coordinate $\gF_{\ga,h}: \C{P}_\ga\to \BB{C}$ conjugating the dynamics of $h_\ga$ on $\C{P}_\ga$ 
to the translation by one. 
Let $\gs_{\ga,h}$ denote the non-zero fixed point of $h_\ga$ obtained in \ref{P:sigma-fixed-point}. 
Recall that the complex rotation of $h_\ga$ at $\gs_{\ga,h}$ is denoted by $\gb$, that is, 
$h_\ga'(\gs_{\ga,h})=e^{2\pi \B{i}\gb}$. 

Consider the covering 
\[\check{\gta}_{\ga,h}(w)=  \frac{\gs_{\ga,h}}{1-e^{-2\pi \B{i} \gb w}}.\] 
This is periodic of period $1/\gb$, where $+\B{i}\infty$ corresponds to $0$ and $-\B{i}\infty$ corresponds to 
$\gs_{\ga,h}$. 

The petal $\C{P}_{\ga,h}$ lifts under $\check{\gta}_{\ga,h}$ to a periodic set, one of its connected components 
separates $0$ from $-1/\gb$. 
Note that when $\ga\in A^+(r_3)$, by Formula~\ref{E:holomorphic-index-formula}, $\Re (-1/\gb)\geq 0$.  
We denote this component by $\check{\C{P}}_{\ga,h}$. 
The map $h_\ga$ on $\C{P}_{\ga,h}$ lifts under $\check{\gta}_{\ga,h}$ to a univalent map $\check{F}_{\ga,h}$
defined on $\check{\gta}_{\ga,h}^{-1}(\C{P}_{\ga,h})$. 
This lift satisfies, 
\[h_\ga \circ \check{\gta}_{\ga,h}(w)=\check{\gta}_{\ga,h} \circ  \check{F}_{\ga,h}(w), \quad 
\check{F}_{\ga,h}(w+1/\gb)= \check{F}_{\ga,h}(w)+1/\gb, \quad w\in \check{\gta}_{\ga,h}^{-1}(\C{P}_{\ga,h}),\] 
and is given by the formula, 
\[\check{F}_{\ga,h}(w)=w+ \frac{1}{2\pi \B{i}\gb} \log \big (1- \frac{\gs_{\ga,h}u_{\ga,h}}{1+ z u_{\ga,h}(z)}\big),
\twith z= \check{\gta}_{\ga,h}(w).\]
As in the previous case, we cork with the branch of $\log$ with $\Im \log (\cdot) \subseteq (-\pi, +\pi)$. 
With this chose, $\check{F}_{\ga,h}$ is asymptotic to a translation by $+1$ near the lower end, 
\[\lim_{\Im (\gb w) \to +\infty} | \check{F}_{\ga,h}(w) -(w+1) |= 0.\]

The unique critical point of $h_\ga$ lifts under $\check{\gta}_{\ga,h}$ to a $1/\gb$-periodic set of points, 
one of which lies on $\check{\C{P}}_{\ga,h}$ and is denoted by $\check{\cp}_{\ga,h}$. 

One may repeat all the constructions and arguments in Sections~\ref{SS:basic-properties-lift-a} to 
\ref{SS:ecale-maps}, replacing $\ga$ by $-\gb$. 
That is, the analysis is now carried out near the lower end of the domain $\check{\C{P}}_{\ga,h}$. 
This provides us with a proof for part b of Proposition~\ref{P:changes-non-linearity}, and a proof for part 
b of Proposition~\ref{P:bounded-iterates-remaining}. 
Note that by the holomorphic index formula, and the pre-compactness of the class $\IS$, $|1/\ga + 1/\gb|$
is uniformly bounded from above, see Lemma~\ref{L:preliminary-estimate-on-Index}. 
We give a proof of part b of Proposition~\ref{P:bounded-spirals}, where there is a slight difference between the
calculations. 
\begin{proof}[Proof of Proposition~\ref{P:bounded-spirals}-- Part b)]
Let $\check{L}_{\ga,h}$ be the univalent map (analogue of $L_{\ga,h}$) that conjugates $\check{F}_{\ga,h}$ to 
the translation by $+1$, which is normalized by mapping $0$ to $\check{\cp}_{\ga,h}$. 
We use the decomposition of the map $\gF_{\ga,h}$ as $\check{L}_{\ga,h} \circ \check{\gta}_{\ga,h}$. 
where $\check{\gta}_{\ga,h}$ is the covering map defined above. 

Let $\gx=\gx_1+ \B{i} \gx_2$, and $w=w_1+\B{i} w_2= \check{L}_{\ga,h}(\gx)$, where $w_1$, $w_2$, $\gx_1$, 
and $\gx_2$ are real numbers. 
As $\gx_2$ tends to $-\infty$, $w_2$ tends to $-\infty$. 
Let $w_1'$ denote the limit of $w_1$ as $\gx_2$ tends to $-\infty$, and let $w_2'$ denote the limit of $\gx_2-w_2$, as $\gx_2$ tends to $-\infty$. 
By the analogue of Lemma~\ref{P:asymptote-L-top} for $\check{L}_{\ga,h}$ near the bottom end, we have 
\[|w_2'| \leq D_{11} (1- \log |\ga|) , \quad |\gx_1 -w_1'| \leq D_{11} (1- \log |\ga|).\]
Then, the limit in part b of the proposition may be calculates as, 
\begin{align*}
\lim_{\gx_2\to -\infty}& \big (\arg (\check{\gta}_{\ga,h} \circ \check{L}_{\ga,h}(\gx) -\gs_{\ga,h}) 
+ 2\pi \gx_2 \Im \gb \big) \\
&= \arg \gs_{\ga,h} + \lim_{\gx_2 \to -\infty} 
\big (\arg \frac{e^{2\pi \B{i} \gb \check{L}_{\ga,h}(\gx)}}{1- e^{-2\pi \B{i} \gb  \check{L}_{\ga,h}(\gx)}} 
+ 2 \pi w_2 \Im \gb \big) + \lim_{\gx_2 \to -\infty} 2\pi (\gx_2 - w_2)\Im \gb  \\
&= \arg \gs_{\ga,h} +
 \lim _{\gx_2 \to -\infty} \big ( -2\pi w_2 \Im \gb + 2\pi w_1 \Re \gb  + 2\pi w_2 \Im \gb \big) 
 + 2\pi w_2' \Im \ga \big) \\
& = \arg \gs_{\ga,h} + 2 \pi w_1' \Re \gb  + 2\pi w_2' \Im \gb \\
& = \arg \gs_{\ga,h} + 2\pi \gx_1 \Re \gb + 2 \pi (w_1' -\gx_1) \Re \gb + 2\pi w_2' \Im \gb. 
\end{align*}
In the above calculations we have used that $1- e^{-2\pi \B{i}\gb \check{L}_{\ga,h}(\gx)}$ tends to $1$, as
$\gx_2$ tens to $-\infty$.

\end{proof}


\subsection{Pairs of complex rotations}\label{SS:pairs}
For $h\in \IS \cup\{Q_0\}$ and $\ga\in A(r_1)$, the map $\ga \ltimes h$ has a non-zero fixed point in 
$W$ denoted by $\gs(\ga \ltimes h)$; see Proposition~\ref{P:sigma-fixed-point}. 
This fixed point has holomorphic dependence on $h$ and $\ga$. 
Moreover, when $\ga\in A(r_3)$, $\ga\ltimes h$ is renormalizable. 
It follows from the definition of renormalization that $\arg (\ga \ltimes h)'(\gs(\ga \ltimes h))\neq 0$. 

Hence, there is a choice of $\gb(\ga \ltimes h)$, with $\Re \gb(\ga \ltimes h)\in (-1, 0)$ and 
$\gb(\ga \ltimes h)$ holomorphic in $h$ and $\ga$, such that 
$(\ga \ltimes h)'(\gs(\ga \ltimes h))= e^{2\pi \B{i}\gb(\ga \ltimes h)}$. 
Moreover, as $\ga$ tends to zero in $A(r_3)$, $\gb(\ga \ltimes h)$ tends to $0$ in a sector. 
See \refL{L:preliminary-estimate-on-beta} for further details. 

The function  
\[I(\ga \ltimes h)=\frac{1}{2\pi \B{i}} \int_{\partial W} \frac{1}{z-(\ga \ltimes h)(z)} \mathrm{d}z, \quad 
\ga\in A(r_1), h\in \IS\cup\{Q_0\},\]
is holomorphic in $h$ and $\ga$.

\begin{lem}\label{L:preliminary-estimate-on-Index}
There exist positive constants $B_1, B_2, B_3$ such that for all $h_1, h_2\in \IS$ and $\ga$ in  $A(r_1)$ we have 
\begin{itemize}
\item[a)]$ | I (\ga \ltimes h_1)| \leq B_1$
\item[b)]$ |\frac{\partial}{\partial \ga} I(\ga \ltimes h_1)| \leq B_2$
\item[c)]$ |I(\ga \ltimes h_1)-I(\ga \ltimes h_2)|\leq B_3 \Td(h_1, h_2)$
\end{itemize}
\end{lem} 
\begin{proof}
\emph{Part a)}
By Proposition~\ref{P:sigma-fixed-point}, $h_\ga$ has no fixed point on $\partial W$. 
Thus, by the pre-compactness of the class of maps $\IS$, there is $\gd>0$ such that for all 
$z\in \partial W$, $|z-h_\ga(z)| \geq \gd$. 
Hence, $|I(\ga \ltimes h)(z)| \leq \ell(W)/(2\pi\gd)$, where $\ell(W)$ denotes the 
circumference of $W$. 
 
\medskip

\emph{Part b)}
First note that by the Koebe distortion Theorem, $|h'|$ is uniformly bounded from above on 
$e^{2\pi \B{i} \ga} \cdot  \ol{W}$. 
Thus, 
\begin{align*}
|\frac{\partial}{\partial \ga} I(\ga \ltimes h)| 
&\leq \frac{1}{2\pi} \int_{\partial W} \frac{|\frac{\partial}{\partial \ga} h_\ga(z)|}{|z- h_\ga(z)|^2}\, dz \\
&\leq \frac{ \ell(W)}{2\pi \gd^2} \cdot \sup_{z \in W} \Big (
|h'(e^{2\pi \B{i} \ga} z) | \cdot  2\pi \cdot |e^{2\pi \B{i}\ga}| \cdot |z| \Big ). 
\end{align*}

\medskip

\emph{Part c)}
For this part of the lemma we use the majorant principle. 
Fix $h_1$ and $h_2$ in $\IS$ and let $R=\Td(h_1, h_2)$. 
For $i=1,2$, let $h_i= P \circ \gp_i^{-1}$, where $\gp_i: V \to \BB{C}$ is a univalent map with $\gp_i(0)=0$, 
$\gp_i'(0)=1$, and $\gp_i$ has a quasi-conformal extension onto $\BB{C}$. 
Then, by the definition of $\Td$, and the compactness of the class of normalized quasi-conformal mappings 
with dilatation bounded from above by a constant, there is a quasi-conformal mapping $\gp: \BB{C} \to \BB{C}$, 
which is identical to $\gp_1 \circ \gp_2^{-1}$ on $\gp_2(V)$, and $\log \Dil(\gp)=R$. 

Let us define the dilatation quotient $\gm(z)= \partial_{\ol{z}}\gp/ \partial_z \gp$, and let 
$r=\| \gm\|_{\infty}< 1$. 
Then, we have $(1+r)/(1-r)=e^R$. 
For each $\gl\in B(0, 1/r)$, let $\gp^\gl: \BB{C} \to \BB{C}$ denote the unique quasi-  mapping with 
dilatation quotient $\gl\cdot \gm$ normalized with $\gp^\gl(0)=0$ and $(\gp^\gl)'(0)=1$. 
That is, $\gp^\gl$ is the unique solution of the Beltrami equation 
$\partial_{\ol{z}}\gp^\gl=(\gl \cdot \gm) \partial_z \gp^\gl$ with the normalization at $0$. 
By the classical results on Beltrami equation, see for example \cite{AhBe60}, the map $\gp^\gl$ 
has holomorphic dependence on $\gl$. 
For $\gl=1$, we have $\gp^1=\gp_1 \circ \gp_2^{-1}$. 

Define the holomorphic map $h^\gl=P\circ \gp_1^{-1}\circ \gp^\gl$. 
By definition, $h^0=h_1$ and $h^1=h_2$, and $h^\gl$ has holomorphic dependence on $\gl$. 
Now consider the holomorphic map 
\[G(\gl)= I(\ga \ltimes h_1) - I(\ga, h^\gl), \gl\in B(0, 1/r).\]
We have $G(0)=0$, and by part a of the lemma, $|G| \leq 2 B_1$, on $B(0,1/r)$. 
Then, by the Schwarz lemma, $|G(1)| \leq 2 B_1 \cdot r$. 

When $\Td(h_1,h_2)\geq 2 B_1$, the left hand side of the inequality in Part c is bounded by $2B_1$. 
So, the inequality holds for $B_3=1$.
On the other hand, when $\Td(h_1,h_2)=R \leq 2 B_1$, by the relation $(1+r)/(1-r)=e^R$, 
$R$ and $r$ are comparable. 
This finishes the proof of part c). 
\end{proof}

\begin{lem}\label{L:preliminary-estimate-on-beta}
There exist positive constants $B_4, B_5, B_6$ such that for all $h_1, h_2\in \IS$ and $\ga$ in $A(r_3)$ 
we have 
\begin{itemize}
\item[1)]$ B_4^{-1} |\ga| \leq | \gb(\ga \ltimes h_1)| \leq B_4 |\ga|$
\item[2)]$ B_5^{-1}\leq |\frac{\partial \gb}{\partial \ga} (\ga \ltimes h_1)| \leq B_5 $
\item[3)]$|\gb(\ga \times h_1)-\gb(\ga \ltimes h_2)|\leq B_6 |\ga|^2 \Td(h_1, h_2)$
\end{itemize}
\end{lem} 
\begin{proof}
\emph{Part 1)}
For $\ga\in A(r_3)$, $|e^{2\pi \B{i} \ga}|$ is uniformly bounded from above and away from $0$.
By the Koebe distortion theorem, for any univalent map $\gf: V \to \BB{C}$, with $\gf(0)=0$ and $\gf'(0)=1$, 
$|W|$ is uniformly bounded from above and away from $0$ on $W$. 
Combining the two statements we conclude that $|e^{2\pi \B{i}\gb(\ga \ltimes h)}|$ must be uniformly bounded 
from above and away from $0$. 
In particular, $\Im \gb(\ga \ltimes h)$ is uniformly bounded from above and below. 

Let us define the holomorphic map $G$ according to $G(z) z=1-e^{2\pi \B{i} z}$, for $z\in \BB{C}$. 
By the above paragraph, $|G(\gb(\ga \ltimes h))|$ and $|G(\ga)|$ are uniformly bounded from above and below. 
Then, the holomorphic index formula may be written as 
\[\ga G(\ga)+ \gb(\ga \ltimes h) G(\gb(\ga \ltimes h))= \ga \gb(\ga \ltimes h) G(\ga) G(\gb(\ga \ltimes h)) I(\ga \ltimes h).
\]
By the uniform bound in \refL{L:preliminary-estimate-on-Index}-a, we conclude the two estimates 
in the first part. 

\medskip

\emph{Part 2)}
We differentiate the index formula \ref{E:holomorphic-index-formula} with respect to $\ga$ to obtain 
\[\frac{e^{2\pi \B{i}\ga}}{(1-e^{2\pi \B{i}\ga})^2} + 
\frac{e^{2\pi \B{i}\gb(\ga \ltimes h)}}{(1-e^{2\pi \B{i}\gb(\ga \ltimes h)})^2} \frac{\partial}{\partial \ga}\gb(\ga \ltimes h)
=\frac{1}{2\pi \B{i}} \frac{\partial}{\partial \ga} I(\ga \ltimes h) 
,\]
which reduces to 
\[\frac{1}{ \sin^2 (\pi \ga)}+\frac{1}{ \sin^2 (\pi \gb(\ga \ltimes h))} \frac{\partial}{\partial \ga}\gb(\ga \ltimes h)
=\frac{1}{\pi \B{i}} \frac{\partial}{\partial \ga} I(\ga \ltimes h).\]
Thus, 
\[\frac{\partial}{\partial \ga}\gb(\ga \ltimes h)=  \frac{-\sin^2(\pi \gb(\ga \ltimes h))}{\sin^2(\pi\ga)} + \frac{\sin^2(\pi \gb(\ga \ltimes h))}{\pi \B{i}} 
 \frac{\partial}{\partial \ga} I(\ga \ltimes h)     
\]

The upper bound in part 2 follows from the upper bound in \refL{L:preliminary-estimate-on-Index}-2 and 
the uniform bounds in the previous part. 
To obtain  the lower bound, we need to restrict $\ga$ to a smaller region. 
First note that by the bounds in part 1, the first ration on the right-hand side of the above formula is 
compactly contained in $\BB{C}\setminus \{0\}$.  
The, by restricting $\ga\in A(r_4)$, for some $r_4\in (0, +\infty)$, we guarantee that the second ration is small enough. 
This implies the for $\ga\in A(r_4)$, $|\frac{\partial}{\partial \ga}\gb(\ga \ltimes h)|$ is uniformly bounded away 
from $0$. 

\medskip

\emph{Part 3)}
We subtract the holomorphic index formula for the maps $(\ga \ltimes h_1)$ and $(\ga \ltimes h_2)$ to get 
\[\frac{1}{\gb(\ga \ltimes h_1)G(\gb(\ga \ltimes h_1))} - \frac{1}{\gb(\ga \ltimes h_2)G(\gb(\ga \ltimes h_2))} 
= I(\ga \ltimes h_1) -I(\ga \ltimes h_2),\]
which provides us with 
\begin{multline*}
\gb(\ga \ltimes h_2)G(\gb(\ga \ltimes h_2)) - \gb(\ga \ltimes h_1)G(\gb(\ga \ltimes h_1)) =  \\
\gb(\ga \ltimes h_1) \gb(\ga \ltimes h_2) G(\gb(\ga \ltimes h_1)) G(\gb(\ga \ltimes h_2)) \big (I(\ga \ltimes h_1) 
-I(\ga \ltimes h_2) \big ).
\end{multline*}
Since $\gb(\ga \ltimes h)$ is uniformly bounded, 
$|G(\gb(\ga \ltimes h_1)) - G(\gb(\ga \ltimes h_2))| = O (|\gb(\ga \ltimes h_1) - \gb(\ga \ltimes h_2)|)$. 
Also, by the estimates in part 1, $|\gb(h_1)|$ and $|\gb(\ga \ltimes h_2)|$ are bounded by $B_4\ga$. 
Combining these with the above formula, one obtains the desired inequality using the bound in part $3$ 
of \refL{L:preliminary-estimate-on-Index}. 
\end{proof}

\begin{proof}[Proof of proposition~\ref{P:2-2-derivative}]
This mainly follows from \refT{T:Ino-Shi2}.
For every fixed $\ga\in A(r_3)$, the map $h \mapsto \pi \circ \npr{1} (\ga, h)$ induces a holomorphic map 
from the Teichmuler space of $\BB{C}\setminus \ol{V}$ to the Teichmuller space $\BB{C}\setminus \ol{V}$. 
By the Royden-Gardiner theorem, any holomorphic map of Teichmuller spaces does not expand distances. 
On the other hand, by the holomorphic extension property in \refT{T:Ino-Shi2} the image of this map is a 
compact subset of the Teichmuller space of $\BB{C}\setminus \ol{V}$. 
It follows that this map is uniformly contracting. 
\end{proof}

\begin{rem}
By Theorem~\ref{T:Ino-Shi2}, the maps $\hat{h}(\ga \ltimes h)$ and $\check{h}(\ga\ltimes h)$ extend 
onto holomorphic maps on the strictly larger domain $U$, which contains the closure of $V$.  
This may be used to prove the existence of a constant $c_{2,2}$ which is strictly less than $1$. 
However, we do not need this uniform contraction in this paper.
\end{rem}

\begin{proof}[Proof of proposition \ref{P:1-1-derivative}]
Recall the relation $\check{\ga}(\ga \ltimes h)= -1/\gb(\ga \ltimes h)$. 
From the estimates in \refL{L:preliminary-estimate-on-beta}, we have 
\[|\frac{\partial \check{\ga}}{\partial \ga} (\ga \ltimes h)|
=|\frac{1}{\gb(\ga \ltimes h)}|^2 \cdot | \frac{\partial \gb}{\partial \ga}(\ga \ltimes h)|
\geq \frac{1}{B_4^2 |\ga|^2} B_5^{-1}.\]
\end{proof}

\begin{proof}[Proof of Proposition~\ref{P:1-2-derivative}]
Here we use the estimates in \refL{L:preliminary-estimate-on-Index}. 
\begin{align*}
|\check{\ga}(\ga \ltimes h_1)-\check{\ga}(\ga \ltimes h_2)| 
&= |\frac{-1}{\gb(\ga \ltimes h_1)}+\frac{1}{\gb(\ga \ltimes h_2)}| \\
&\leq \frac{B_4^2}{|\ga|^2}   |\gb(\ga \ltimes h_1)-\gb(\ga \ltimes h_2)| \\
&\leq  \frac{B_4^2}{|\ga|^2}  B_6 |\ga^2| \Td(h_1, h_2) 
= B_4^2 B_6 \Td(h_1, h_2). 
\end{align*}
\end{proof}
\section{Polynomial-like renormalizations and their combinatorics}
\label{S:poly-like-renormalization}
In this section we outline the basic properties of the dynamics of quadratic polynomials that 
we refer to in this paper. 
One may consult \cite{Mi06} and \cite{Be91} for basic notions in complex dynamics.
The material on the polynomial-like renormalization presented here is mainly following the 
foundational work of Douady and Hubbard presented in \cite{DH84,DH85}. 
One may also refer to \cite{Mi2} \cite{Sch04} for detailed discussions on the combinatorial 
aspects of the topic. 

In this section, we assume that all quadratic polynomials are normalized so that they are monic and 
their critical points are at $0$. 
That is, they are of the form $P_c(z)=z^2+c$. 

\subsection{Combinatorial rotation of the dividing fixed point} 
The \textit{filled Julia} set $K(P_c)$ and the \textit{Julia} set $J(P_c)$ of the quadratic polynomial $P_c$ are defined as 
\[K(P_c)=\set{z\in \BB{C} \mid \textstyle{\sup_{n\in \BB{N}}} |P_c\co{n}(z)|<+\infty}, J(P_c)=\partial K(P_c).\]
These are compact subsets of $\BB{C}$. 
Each set $K(P_c)$ is connected if and only if the critical point $0$ belongs to $K(P_c)$  
In this paper we only work with quadratics with connected Julia sets. 

By the maximum principle, $K(P_c)$ is full, i.e.\ its complement has no bounded connected component. 
The \textit{B\"ottcher} coordinate of $P_c$ is the conformal isomorphism 
\[\gf_c: \BB{C}\setminus K(P_c) \to \BB{C} \setminus \overline{\BB{D}}, \text{ where } 
\BB{D}=\{z\in \BB{C} \mid |w|<1\},\]
that is tangent to the identity near infinity. 
It conjugates $P_c$ on $\BB{C} \setminus K(P_c)$ to $w \mapsto w^2$ on 
$\BB{C}\setminus \ol{\BB{D}}$. 
By means of this isomorphism, the \textit{external ray} of angle $\gt\in [0, 2\pi)$ and \textit{equipotential} of radius $r\in (1, +\infty)$ 
are defined as  
\[R_c^\gt=\set{\gf_c^{-1}(r e^{i \gt})\mid r\in (1, +\infty)}, E_c^{r}=\set{\gf_c^{-1}(r e^{i\gt})\mid \gt\in [0, 2\pi]}.\]
The map $P_c$ send $E_c^r$ to $E_c^{r^2}$ and send $R_c^\gt$ to $R_c^{2\gt}$. 
An external ray $R_c^\gt$ is called \textit{periodic} under $P_c$, if there is $n\in \BB{N}$ with 
$P_c\co{n}(R_c^\gt)=R_c^\gt$. 
Equivalently, this occurs if $2^n \gt= \gt \mod 2\pi \BB{Z}$.  
An external ray $R_c^\gt$ is said to \textit{land} at a point, if the limit of $\gf_c^{-1}(r e^{i \gt})$, as $r \to 1$, exists in $\BB{C}$. 
The landing point of a periodic ray $R_c^\gt$ is necessarily a periodic point of $P_c$.  

Let $f:U\ci \BB{C} \to \BB{C}$ be a holomorphic map with a periodic point $z\in U$. 
The \textit{multiplier} of $z$ is defined as $(f\co{n})'(z)$, where $n$ is the smallest 
positive integer with $f\co{n}(z)=z$.   
The periodic point is called \textit{parabolic}, if its multiplier is a root of unity.  

We recall the Douady's theorem on the landing property of repelling and parabolic periodic points. 
One may refer to \cite{DH84} for a proof of this result, and also \cite{ErLe89,Mi06,Pe93,Po86}
for alternative proofs and generalizations of this result. 
 
\begin{propo}[Douady, 1987]\label{P:landing-of-periodic-rays}
Let $P_c$ be a quadratic polynomial with connected $K(P_c)$. 
Every repelling or parabolic periodic point of $P_c$ is the landing point of at least one, but at 
most a finite number of, external rays. 
\end{propo}

Let $a_0$ denote the landing point of the unique fixed ray of $P_c$; $R_c^0$. 
When $P_c$ has only one fixed point, $a_0$ must be parabolic with multiplier $1$.
This occurs for $c=1/4$.   
For $c\neq 1/4$, $a_0$ is repelling. 
The other fixed point of $P_c$, denoted by $a_c$, is \textit{attracting} in a region called the
\textit{main hyperbolic component} of the Mandelbrot set. 
This is the large cardioid visible in the center of the Mandelbrot set. 

For $c$ outside the closure of the main hyperbolic component of $M$, $a_c$ is \textit{repelling}. 
By the above proposition there are at least one, but a finite number of, external rays landing on $a_c$.
By a simple topological argument (since the map has only one critical point) the set of rays landing 
at $a_c$ is formed of the orbit of a single periodic ray.
Let $\gt_j\in [0, 2\pi)$, for $1\leq j\leq q$ and $q\geq 2$, denote the angles of the external rays 
landing at $a_c$, labeled in increasing order. 
There is a non-zero integer $p\in (-q/2, q/2]$, with $(|p|, q)=1$, such that  $P_c(R_c^{\gt_j})=R_c^{\gt_{j'}}$ where $j'=j+p \pmod q$. 
The rational number $p/q$ is called the \textit{combinatorial rotation number} of $P_c$ at $a_c$.     
The fixed point $a_c$, when it is repelling or parabolic, is referred to as \textit{the dividing 
fixed point} of $P_c$. 
It follows that for any rational number $p/q\in (-1/2, 1/2]$, there are parameters $c$ in the 
Mandelbrot set where $a_c$ has combinatorial rotation number $p/q$ at $a_c$. 
We shall come back to this in a moment. 


\subsection{Polynomial-like renormalization}
Assume that the fixed point $a_c$ is repelling. 
The closure of the $q$ rays landing at $a_c$ cut the complex plane into $q$ (open) connected components which we denote by 
$Y_j$, for $0\leq j \leq q-1$. 
By a simple topological consideration, $P_c$ on these pieces has a simple covering property. 
One of these components, which we denote by $Y_0$, contains both critical point $0$ and the pre-fixed point $-a_c$, while its 
image under $P_c$ covers all $Y_j$, for $0\leq j\leq q-1$.
We may relabel the other components so that $P_c(Y_j)=Y_{j+1}$, for $1\leq j\leq q-2$, and $P_c(Y_{q-1})=Y_0$. 
Thus, the critical point is mapped into $Y_1$ in one iterate of $P_c$, and is mapped back into $Y_0$ under $q$ iterates of $P_c$.

Fix $r>1$. 
The equipotential $E_c^r$ divides each piece $Y_j$, for $0\leq j\leq q-1$, into two connected components. 
We denote by $Y_j^1$, for $0 \leq j \leq q-1$, the closure of the bounded connected component of $Y_j\setminus E_c^r$.
These are called \textit{puzzle pieces} of level $1$.
The closure of the connected components of $P_c^{-i}(\inte Y_j^1)$, for $i\in \BB{N}$ and $0 \leq j \leq q-1$, 
are called puzzle pieces of level $i$.
They form nests of pieces breaking the Julia set into components.
As $P_c(Y_0)$ covers $\cup_{j=0}^q-1 \inte Y_j^1 $, the connected components of $P_c^{-1}(\inte Y_j^1)$ contained in 
$Y_0$ divide $Y_0$ into $q$ pieces. 
We denote the one containing $0$ by $Z_0^2$ and the remaining ones by $Z_j^2$, for $1\leq j\leq q-1$. 

For $c$ in the Mandelbrot set with $a_c$ repelling and $q$ rays landing at $a_c$, $P_c$ is 
called \textit{polynomial-like renormalizable} of \textit{satellite} type, if $P_c\co{(jq)}(0)\in Z_0^2$,
for all $j\in \BB{N}$.
Using $|P_c'(a_c)|>1$, one builds a simply connected domain $\tilde{Z}_0^2$ containing the 
closure of $Z_0^2$ such that $P_c\co{q}(\tilde{Z}_0^2)$ contains the closure of $\tilde{Z}_0^2$ 
and $P_c\co{q}:\tilde{Z}_0^q\to P_c\co{q}(\tilde{Z}_0^2)$ is a proper branched covering of 
degree two. 
We denote this map by $\tilde{\C{R}}_{\scriptscriptstyle PL}(P_c)$, and note that the orbit of 
$0$ under $\tilde{\C{R}}_{\scriptscriptstyle PL}(P_c)$ remains in $\tilde{Z}_0^2$. 

For $c$ as in the above paragraph, if $P_c$ is not polynomial-like renormalizable of satellite type, 
there is the smallest $n\in \BB{N}$ such that $P_c\co{(nq)}(0)\in Z_j^2$, for some 
$1\leq j\leq q-1$. 
Then, let $V^1$ denote the closure of the connected component of $P_c^{-nq}(\inte Z_j^2)$ 
containing $0$ that is obtained by pulling back along the orbit $0, P_c(0), \dots, P_c\co{nq}(0)$. 
If the orbit of $0$ under $P_c$ never enters $\inte V^1$, then $P_c$ is not polynomial-like 
renormalizable. 
Otherwise, let $j_1\in \BB{N}$ be the smallest integer with $P_c\co{j_1}(0)\in \inte V^1$, and 
denote by $V^2$ the closure of the component of $P_c^{-j_1}(\inte V^1)$ that is obtained by 
pulling back along $0, P_c(0), \dots, P_c\co{j_1}(0)$. 
Here, $P_c\co{j_1}: \inte V^2\to \inte V^1$ is a proper branched covering of degree two.  
Inductively, we define domains $V^1\supset V^2\supset V^3\supset \dots$, and positive integers 
$j_1, j_2, j_3, \dots$ such that each $j_m$ is the smallest positive integer with 
$P_c\co{j_m}(0)\in \inte V^{m}$ and $V^{m+1}$ is the closure of the pull-back of $\inte V^{m}$ 
along the orbit $0, P_c(0), \dots, P_c\co{j_m}(0)$.   
For all $m$, $V^m$ is compactly contained in the interior of $V^{m-1}$ and 
$P_c\co{j_m}: V^m\to V^{m-1}$ is a proper branched covering of degree two.
The quadratic $P_c$ is \textit{polynomial-like renormalizable} of \textit{primitive type}, 
if there is $m\in \BB{N}$ such that the orbit of $0$ under $P_c\co{j_m}: V^m\to V^{m-1}$ 
remains in $V^m$.
Note that if this occurs, the sequence $j_m$ is eventually constant.  
For the smallest positive integer $m$ satisfying this property, we denoted  
$P_c\co{j_m}: V^m\to V^{m-1}$ by $\tilde{\C{R}}_{\scriptscriptstyle PL}(P_c)$. 
If there is no $m$ with this property, then $P_c$ is not polynomial-like renormalizable.

A quadratic map is called \textit{polynomial-like renormalizable}, if it is polynomial-like 
renormalizable of either satellite or primitive type.  

A \textit{polynomial-like mapping} of degree $d$ is a proper branched covering holomorphic map 
$f:U\to V$ of degree $d$, where $U$ and $V$ are simply connected domains with $U$ 
compactly contained in $V$. 
For example, the restriction of any polynomial $P$ to $P^{-1}(B(0,R))$, for sufficiently large $R$, 
is a polynomial-like map. 
The successive renormalizations obtained above are non-trivial examples of polynomial-like maps 
of degree two.
One may define the \textit{filled Julia} set and the \textit{Julia} set of a polynomial-like map in 
the same fashion;
\[K(f)=\set{z\in U \mid f\co{j}(z)\in U, \forall  j \in \BB{N}}, J(f)=\partial K(f).\]
Similarly, $K(f)$ and $J(f)$ are connected if and only if the orbits of all branched points of $f$ 
remain in $U$. 

Two polynomial-like mappings $f:U\to V$ and $g: U'\to V'$ are \textit{quasi-conformally} 
conjugate if there is a quasi-conformal map $h: V\to V'$ such that $g\circ h=h \circ f $ on $U$. 
They are called \textit{hybrid} conjugate if they are quasi-conformally conjugate and 
the quasi-conformal conjugacy $h$ between them may be chosen so that $\ol{\partial} h=0$ on $K(f)$.   
A remarkable result of Douady and Hubbard is that the dynamics of a polynomial-like map is 
the same as the dynamics of some polynomial. 

\begin{thm}[Straightening \cite{DH85}]
Let $f: U\to V$ be a polynomial-like map of degree $d$ with connected Julia set. 
Then, $f$ is hybrid conjugate to an appropriate restriction of a polynomial of degree $d$.
Moreover, the polynomial is unique up to an affine conjugacy.
\end{thm}
In the normalized quadratic family $z\mapsto z^2+c$, $c\in \BB{C}$, each affine conjugacy class contains only one element.
Thus, every polynomial-like map of degree two is hybrid conjugate to a unique element of this family. 
Also, although the hybrid conjugacy $h$ in the above theorem is not unique, $h$ is 
uniquely determined on $J(f)$ upto affine conjugacy.

Recall the polynomial-like renormalization of satellite type 
$\tilde{\C{R}}_{\scriptscriptstyle PL}(P_c)=P_c\co{q}: \tilde{Z}_0^2\to P_c\co{q}(\tilde{Z}_0^2)$ obtained above.
Let $M(p/q)$ denote the set of all $c\in M$ such that the dividing fixed point $a_c$ of $P_c$ has combinatorial rotation 
$p/q$ and $P_c\co{jq}(0)\in Z_0^2$, for all $j\in \BB{N}$.
By the straightening theorem, for all $c\in M(p/q)$, except at $c$ where $P_c'(a_c)= e^{2\pi i p/q}$, 
$\tilde{\C{R}}_{\scriptscriptstyle PL}(P_c)$ is hybrid conjugate to some quadratic polynomial denoted by $\plr{1}(P_c)$. 
Indeed, through straightening theorem, Douady and Hubbard obtain a 
homeomorphism\footnote{The homeomorphism sends the parameter $c$ with 
$P_c'(a_c)= e^{2\pi i p/q}$ to the point $1/4\in M$.}  
\[\gF_{p/q}: M(p/q)\to M.\]
The set $M(p/q)$ is called the $p/q$-\textit{satellite copy} of the Mandelbrot set.

Similarly, for polynomial-like renormalization of primitive type 
$\tilde{\C{R}}_{\scriptscriptstyle PL}(P_c)=P_c\co{j_m}: V^{m}\to V^{m-1}$, 
one may consider the connected component containing $c$ of all parameters $c$ where the 
isotopy type of $V^m$ remain constant and the Julia set of 
$\tilde{\C{R}}_{\scriptscriptstyle PL}(P_c): V^m \to V^{m-1}$ is connected. 
Through straightening theorem, this gives rise to a homeomorphic copy of $M$ within $M$, called a \textit{primitive copy} of the 
Mandelbrot set.

The satellite and primitive copies of $M$ obtained above are maximal in the sense that they are not contained in 
any other homeomorphic copy of $M$, except $M$ itself. 
In this paper we will only work with the satellite copies of the Mandelbrot set.

The main hyperbolic component of $M$ is the set of $c$ where $|P_c'(a_c)|<1$.
For all rationals $p/q\in [-1/2,1/2]$ with $(p,q)=1$, there is $c$ on the boundary of this component where $P_c'(a_c)= e^{2\pi i p/q}$.
There are $q$ rays landing at $a_c$ with combinatorial rotation $p/q$, and the parameter $c$ gives rise to the 
satellite copy $M(p/q)$ attached to the main hyperbolic component of $M$ at $c$. 

\begin{figure}\label{F:PL-renormalization}
\begin{pspicture}(5,4)
\psccurve[fillstyle=solid,fillcolor=lightgray,linecolor=lightgray,linewidth=.5pt](1.9,2)(2.6,1.7)(4,2)(2.6,2.3)(1.9,2)
\psccurve[fillstyle=solid,fillcolor=lightgray,linecolor=lightgray,linewidth=.5pt](2,1.9)(2.3,2.75)(2,3.4)(1.7,2.75)(2,1.9)
\psccurve[fillstyle=solid,fillcolor=lightgray,linecolor=lightgray,linewidth=.5pt](2.1,2)(1.4,1.7)(.4,2)(1.4,2.3)(2.1,2)
\psccurve[fillstyle=solid,fillcolor=lightgray,linecolor=lightgray,linewidth=.5pt](2,2.1)(2.3,1.25)(2,.2)(1.7,1.25)(2,2.1)

\psccurve[fillstyle=solid,fillcolor=lightgray,linecolor=lightgray,linewidth=.5pt](1.85,2)(2.6,1.5)(4.2,2)(2.6,2.5)(1.85,2)

\psccurve[linewidth=.5pt](1.9,2)(2.6,1.7)(4,2)(2.6,2.3)(1.9,2)
\psccurve[linewidth=.5pt](2,1.9)(2.3,2.75)(2,3.4)(1.7,2.75)(2,1.9)
\psccurve[linewidth=.5pt](2.1,2)(1.4,1.7)(.4,2)(1.4,2.3)(2.1,2)
\psccurve[linewidth=.5pt](2,2.1)(2.3,1.25)(2,.2)(1.7,1.25)(2,2.1)
\psccurve[linewidth=.5pt](1.85,2)(2.6,1.5)(4.2,2)(2.6,2.5)(1.85,2)

\pscurve[linecolor=gray,linewidth=.5pt]{->}(3,2.2)(2.8,2.8)(2.2,3) \rput(2.9,2.9){\tiny $P_c$}
\pscurve[linecolor=gray,linewidth=.5pt]{->}(1.8,3)(1.25,2.8)(1.2,2.2) \rput(1,2.8){\tiny $P_c$}
\pscurve[linecolor=gray,linewidth=.5pt]{->}(1.2,1.8)(1.25,1.2)(1.8,1)\rput(1,1.2){\tiny $P_c$}
\pscurve[linecolor=gray,linewidth=.5pt]{->}(2.2,1)(2.8,1.2)(3,1.6)\rput(2.9,1.1){\tiny $P_c$}

\pscurve[linecolor=gray,linewidth=.5pt]{->}(3.9,2)(4.2,2.4)(3.9,2.4) \rput(4.5,2.4){\tiny $P_c\co{4}$}
\psdot[dotsize=2pt](2,2)
\end{pspicture}
\caption{A PL-renormalization of satellite type, $P_c\co{4}: U \to V$. 
In this example, the combinatorial type of PL-renormalization is $1/4$.}
\end{figure}
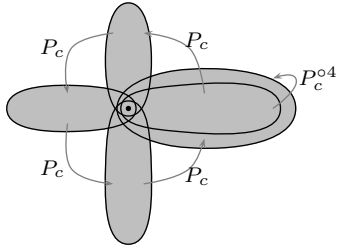

\subsection{Combinatorics of the PL-renormalization}\label{SS:nest-Mandelbrot-sets}
Assume that $P_c$ is polynomial-like renormalizable of satellite type 
$M(p_1/q_1)$, for some $p_1/q_1\in (-1/2,1/2]\cap \BB{Q}$. 
If the quadratic polynomial $\plr{1} (P_c)$ is also polynomial-like renormalizable of satellite type, 
the corresponding parameter, $\plr{1} (P_c)(0)$, belongs to
$M(p_2/q_2)$, for some $p_2/q_2\in (-1/2,1/2]\cap \BB{Q}$. 
We have the return map $\tilde{\C{R}}_{\scriptscriptstyle PL}(\plr{1}(P_c))$ that is hybrid conjugate to $\plr{2}(P_c)$. 
The parameter $c$ belongs to the homeomorphic copy  
\[M(p_1/q_1, p_2/q_2)= \gF_{p_1/q_1}^{-1}(M(p_2/q_2))\]
of $M$ contained in $M(p_1/q_1)$.  
Similarly, for an infinitely polynomial-like renormalizable of satellite type quadratic map 
we obtain a sequence of maximal satellite copies $M(p_j/q_j)$, 
for $j=1,2,\dots$, where $M(p_j/q_j)$ determining the type of renormalization at level $j-1$.
Then, we define the nest of Mandelbrot copies 
\[M(p_1/q_1) \supset M(p_1/q_1, p_2/q_2) \supset M(p_1/q_1, p_2/q_2, p_3/q_3) \supset \dots,\] 
where, for $n\geq 3$,  
\[M(p_1/q_1, p_2/q_2, \dots , p_n/q_n)= \gF_{p_1/q_1}^{-1} \circ \gF_{p_2/q_2}^{-1} \circ \dots 
\circ \gF_{p_{n-1}/q_{n-1}}^{-1}(M(p_n/q_n)).\]  
For the simplicity of the notation we set 
\begin{gather*}
M_1(c)= M(p_1/q_1), M_2(c)= M(p_1/q_1,p_2/q_2), \dots,\\
 M_n(c)=M(p_1/q_1, p_2/q_2, \dots, p_n/q_n).
\end{gather*}
The \textit{root points} of these copies are defined as 
\begin{gather*}
c_0=1/4, c_1=\gF_{p_1/q_1}^{-1}(1/4)\in M_1(c), c_2= \gF_{p_1/q_1}^{-1} \circ \gF_{p_2/q_2}^{-1}(1/4), \\
c_n= \gF_{p_1/q_1}^{-1} \circ \gF_{p_2/q_2}^{-1} \circ \dots \circ \gF_{p_n/q_n}^{-1}(1/4)\in M_n(c), 
n\geq 3.
\end{gather*}
In other words, $c_n \in M_n$ is the unique parameter with $\plr{n}(P_{c_n})(z)=z^2+1/4$. 

For an infinitely polynomial-like renormalizable $P_c$, we have the sequence of quadratic polynomials 
$\plr{n}(P_c)$, for $n \geq 0$, and the return maps $\tilde{\C{R}}_{\scriptscriptstyle PL}(\plr{n}(P_c))$. 
By the straightening theorem, there are quasi-conformal maps $S_n$, for $n\geq 0$, that hybrid conjugate  
$\tilde{\C{R}}_{\scriptscriptstyle PL}(\plr{n}(P_c))$ to $\plr{(n+1)}(P_c))$. 
Let $a(\plr{n}(P_c))$ denote the dividing fixed point of the quadratic map $\plr{n}(P_c)$, and set 
\begin{gather*}
a_1=a(P_c)=a_c, a_2=S_0^{-1}(a(\plr{1}(P_c))), a_3=S_0^{-1}\circ S_1^{-1}(a(\plr{2}(P_c))), \\
a_{n+1}=S_0^{-1}\circ S_1^{-1} \circ \dots \circ S_{n-1}^{-1}(a(\plr{n}(P_c))), \tfor n\geq 3.
\end{gather*}
Each $a_n$, for $n\geq 1$, is a dividing periodic point of $P_c$ of minimal period $\Pi_{j=1}^{n} q_j$. 
The multipliers of these periodic points play significant role in this study.


\subsection{The rigidity and MLC conjectures}
\label{SS:rigidity-mlc}

The \textit{combinatorics} of a quadratic polynomial $P_c$ is defined as an equivalence relation on the set of angles of external rays. 
That is, two angles $\gt_1$ and $\gt_2$ in $[0,2\pi]$ are considered equivalent, if the external 
rays $R_c^{\gt_1}$ and $R_c^{\gt_2}$ land at the same point on $J(P_c)$. 
For an infinitely polynomial-like renormalizable quadratic polynomial $P_c$, it turns out that the 
combinatorics  of the map is uniquely determined by the nest of relatively maximal Mandelbrot 
copies containing $c$.

Two quadratic polynomials are called \textit{combinatorially equivalent}, if they induce the same equivalence relation on the circle. 
In the case of infinitely polynomial-like renormalizable maps, combinatorially equivalent maps fall 
into the same nest of relatively maximal Mandelbrot copies.
We note that if two maps $P_c$ and $P_{c'}$ are combinatorially equivalent and one of them is 
infinitely polynomial-like renormalization, the other one must also be infinitely polynomial-like 
renormalizable. 
 
The \textit{rigidity conjecture} states that combinatorially equivalent quadratic polynomials 
with all their periodic points repelling are conformally conjugate.
This conjecture is equivalent to the local connectivity of the Mandelbrot set \cite{DH84} in its general 
form. 
However, in the case of infinitely polynomial-like renormalizable maps, there is a simple criterion 
that implies both conjectures at the corresponding parameter. 
If a nest of Mandelbrot copies shrinks to a single point $c$, then  $P_c$ is combinatorially rigid, 
$c$ lies on the boundary of the Mandelbrot set, and $P_c$ may be approximated by hyperbolic 
quadratic polynomials with connected Julia sets.
This is because each Mandelbrot copy in $M$ contains a parameter $c'$ where $P_{c'}$ 
has a periodic critical point, and a root point $c''$ on $\partial M$ where $P_{c''}$ 
has a parabolic periodic point with multiplier $+1$.  

\begin{propo}\label{P:shrinking-nest-n-MLC}
If a nest of Mandelbrot copies shrinks to a single point, then the Mandelbrot set is locally connected at the intersection. 
\end{propo}

By the result of Douady and Hubbard on the equivalence of the two conjecture, the shrinking of a 
nest of Mandelbrot copies implies the local connectivity of the Mandelbrot set at the intersection.  
However, this has never been stated in the above form, although  it may be proved following the 
standard techniques developed in 1980's.  
Here we briefly outline a proof, which requires some basic definitions presented below.

The definition of the B\"ottcher coordinate may be extended to quadratic polynomials with 
disconnected Julia sets. 
For $c\in \BB{C}\setminus M$, there is a connected domain $U_c$ bounded by piece-wise
analytic curves, a real number $r_c>1$, and a conformal bijection  
$\gf_c: U_c \to \{ z\in \BB{C} \mid  |z|> r_c \}$ 
that is tangent to the identity at infinity. 
Moreover, the critical value $c$ belongs to $U_c$ and $|\gf_c(c)|= r_c^2$. 
Douady and Hubbard proved that the mapping defined as 
\[c\mapsto \gf_c(c)\]
provides a conformal bijection from $\BB{C} \setminus M$ to $ \BB{C} \setminus \overline{\BB{D}}$. 
Through this map, one may define the external rays of the Mandelbrot set as the preimages of 
the straight rays $(1, +\infty) e^{i \gt}$, for $\gt\in [0,2\pi]$. 

A point $c\in M$ is called a \textit{Misiurewicz parameter}, if there is an integer $n\geq 2$ such 
that $P_c\co{n}(0)$ is a repelling periodic point of $P_c$.
Let us also say that $c\in M$ is a \textit{parabolic parameter} if $P_c$ has a parabolic periodic point. 
A key result regarding the landing property of the external rays of the Mandelbrot set is the following. 

\begin{propo}[Douady-Hubbard \cite{DH84}]
Every parabolic parameter $c\in M$ different from $1/4$ is the landing point of two distinct 
external rays of the Mandelbrot set.  
Every Misiurewicz parameter $c\in M$ is the landing point of at least one, but at most a finite 
number of, external rays of the Mandelbrot set.
\end{propo}

Let $M'$ be a homeomorphic copy of $M$ strictly contained in $M$. 
There is a unique parameter $c'\in M'$ that corresponds to the point $1/4$ under the homeomorphism 
mapping $M'$ to $M$. 
The map $P_{c'}$ has a parabolic periodic point with multiplier $+1$. 
This is the root point of the copy $M'$ defined above, and since $M'\neq M$, $c'\neq 1/4$.
By the above proposition, there are two external rays of $M$ landing at $c'$. 
The union of these rays and their landing point $c'$ divide the complex plane into two connected 
components. 
We let $W(M')$ denote the closure of the connected component containing $M' \setminus \{c'\}$.  
 In the literature, the interior of $W(M')$ (sometimes with the root point $c'$ added to it) is called the 
\textit{parabolic wake} containing the copy $M'$. 
Note that as the set $M$ is connected, $W(M')\cap M$ must be connected. 

If there are more than one external ray landing at a Misiurewicz parameter $m \in M$, the union 
of these rays and their landing point $m$ divides the complex plane into a finite number of regions. 
One of these regions contains $0$ in its interior. 
The remaining connected components are called \textit{Misiurewicz wakes} at $m$. 
A Misiurewicz wake of the Mandelbrot set is, by definition, a Misiurewicz wake at some 
Misiurewicz parameter $m\in M$, and is assumed to be an open subsets of the plane. 
It follows that the intersection of any Misiurewicz wake with the Mandelbrot set in a connected subset 
of $\BB{C}$, and also the Mandelbrot set $M$ minus any Misiurewicz wake is a connected subset of 
$\BB{C}$. 

\begin{propo}[\cite{Mi2}]\label{P:little-copy-trimmed}
Let $M'$ be a homeomorphic copy of the Mandelbrot set strictly contained in $M$. 
Then the set $M'$ is obtained from removing a countable number of Misiurewicz wakes from the set 
$W(M') \cap M$.  
\end{propo}

Now we are ready to prove the result we need. 

\begin{proof}[Proof of Proposition~\ref{P:shrinking-nest-n-MLC}]
Assume that $M \supset M_1 \supset M_2 \supset \dots $ is a nest of Mandelbrot copies shrinking 
to a single point $c\in M$. 
Then, $P_c$ is infinitely polynomial-like renormalizable, and in particular, $c$ is not the root point 
of any of the copies $M_i$, $i\geq 1$. 

Fix $\gep>0$ and choose $n\in \BB{N}$ such that the diameter of $M_n$ is less than $\gep/2$. 
By Proposition~\ref{P:little-copy-trimmed}, $M_n$ is equal to $W(M_n) \cap M$ minus a 
countable number of Misiurewicz wakes.  

Let $U$ be a ball of diameter $\gep$ containing $M_n$. 
The set $(\BB{C}\setminus U) \cap M$ is a compact subset of $M$. 
Thus, by the above paragraph, there are a finite number of Misiurewicz wakes $L_1, L_2, \dots L_n$ such that $(\BB{C}\setminus U) \cap M$ in contained in the union of these wakes.
The set $U$ minus the closure of $\cup_{i=1}^n L_i$, is an open subset of $\BB{C}$, 
and has a connected intersection with $M$.
This set has diameter at  most $\gep$.

As $\gep$ was chosen arbitrary, this implies that there is a basis of neighborhoods $U_i$ containing 
$c$, $i\geq 1$, such that each $M\cap U_i$ is a connected set. 
\end{proof}


\subsection{Bounds on multipliers}\label{SS:bounds-on-multipliers}
We need a relation between the combinatorial rotation number of a repelling period cycle 
and the (analytic) multiplier of that cycle. 
A formula of this type is given by the so called Pommerenke-Levin-Yoccoz inequality, 
see \cite{Hub93}, \cite{Po86}, \cite{Lev91}, \cite{Pe93}. 
Indeed, the general form of the inequality applies to repelling periodic points of arbitrary degree 
polynomials, but here we only state the version for the quadratic polynomials.

\begin{thm}[PLY inequality]\label{T:PLY-inequality}
Let $P_c$ be a polynomial with a connected Julia set. 
Suppose that $\gz$ is a repelling periodic point of $P$ such that  
\begin{itemize}
\item[a)] the minimal period of $\gz$ is $k$; 
\item[b)] there are $q$, $0< q < +\infty$, external rays landing at $\gz$;
\item[c)] the external rays landing at $\gz$ are cyclically permuted with combinatorial  
rotation number $p/q$. 
\end{itemize}
Then, for a suitable branch of $\log$, the multiplier of $\gz$, denoted by $\gr$, 
satisfies
\[\Big|\log\gr - \big(2\pi\textnormal{\B{i}} \frac{p}{q}+\frac{k}{q}\log d \big)\Big| 
\leq\frac{k}{q} \log d.\]
\end{thm}

Let $\langle p_i/q_i\rangle_{i=1}^\infty$ be a sequence of rational numbers in $[-1/2, 1/2]$ and 
$n$ be a positive integer. 
For $c\in M(\langle p_i/q_i\rangle_{i=1}^n)$, $P_c$ has dividing periodic points $a_1$, 
\dots, $a_n$, where each $a_j$ has period $k_j$ given by the formula
\begin{equation}\label{E:periods}
k_1=1, \quad k_j=\prod_{i=1}^{j-1} q_i, \tfor 2 \leq j \leq n.
\end{equation}
Moreover, the combinatorial rotation number of $a_n$ is $p_n/q_n$ and there are $q_n$ 
external rays landing at $a_n$. 
By the above theorem, there is a branch 
of $\log$ such that the multiplier $\gr_n$ of the $n$-th dividing periodic point $a_n$ of $P_c$ 
satisfies 
\begin{equation}\label{E:PLY-copies}
\Big|\frac{1}{2\pi \B{i}} \log\gr_n - \big(\frac{p_n}{q_n}- \B{i}\frac{k_n}{q_n} \frac{\log 2}{2\pi} \big)\Big| 
\leq\frac{k_n}{q_n} \frac{\log 2}{2\pi}.
\end{equation}
See Figure~\ref{F:PLY-log}. 
We have $\frac{1}{2\pi}\log 2= 0.110...$. 

\begin{figure}[ht]\label{F:PLY-log}
\begin{center}
\begin{pspicture}(1.4,0.2)(13.6,3.8)
\epsfxsize=12cm
\rput(7.5,2){\epsfbox{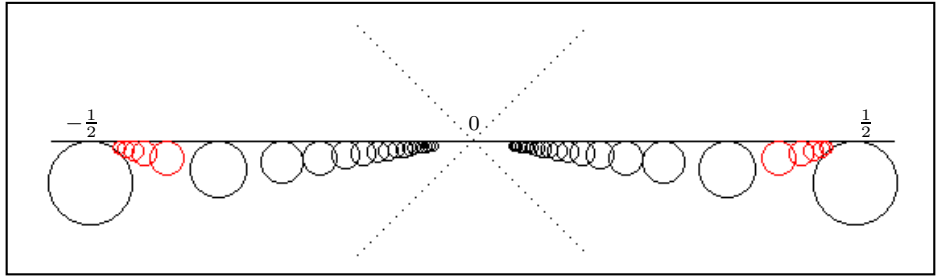}}  
\rput(7.55,2.2){\tiny $0$}
\rput(2.4,2.2){\tiny $-\frac{1}{2}$}
\rput(12.7,2.2){\tiny $\frac{1}{2}$}
\pspolygon(1.4,0.2)(13.6,0.2)(13.6,3.8)(1.4,3.8)
\end{pspicture}
\caption{The black circles denote the locus of $\frac{1}{2\pi \B{i}} \log \gr_1$ for combinatorial rotations 
$\pm1/i$, $2\leq i\leq 20$. 
The red circles denote the locus of this quantity for combinatorial rotations 
$\langle (2,+1),(i,+1) \rangle)$, for $2\leq i \leq 7$, as well as $\langle (2,-1),(i,+1) \rangle)$, 
for $2\leq i \leq 7$.}
\end{center}
\end{figure}

\section{Infinitely near-parabolic renormalizable maps}
\label{S:INP-renormalizable}
By Theorem~\ref{T:Ino-Shi2} and Definition~\ref{D:extended-renormalization}, 
for $f\in \PC{A(r_3)}$ and $f=Q_\ga$ with $\ga \in A(r_3)$, the top and bottom 
near-parabolic renormalizations $\nprt{1}(f)$ and $\nprb{1}(f)$ are defined. 
If either of these maps belongs to $\PC{A(r_3)}$, which depends on whether $\ga(\nprt{1}(f))$ and 
$\ga(\nprb{1}(f))$ belong to $A(r_3)$, we may define the top and bottom near-parabolic 
renormalization of that map in order to obtain the second near-parabolic renormalization of $f$.
This successive near-parabolic renormalization process may be carried out infinitely often for some 
maps $f$. 
One may associate a one sided infinite sequence of t's and b's to determine the type of 
the successive near-parabolic renormalizations, where ``t'' stands for ``top'' and ``b'' stands for 
``bottom''.
In other words, for any $\gk$ in the set 
\begin{equation}\label{E:types}
\C{T}=\{t,b\}^{\BB{N}}= \{(\gk_1, \gk_2, \gk_3, \dots) \mid \forall i\geq 1, \gk_i\in \{t,b\} \},
\end{equation}
we say that a map  $f\in \PC{A(r_3)}$, or $f=Q_\ga$ with $\ga \in A(r_3)$, is 
\textit{infinitely near-parabolic renormalizable of type} $\gk$ if the following infinite sequence of maps is defined
\begin{gather}\label{E:sequence-of-maps}
f_1=f, \;
f_{n+1}= 
\begin{cases}
\nprt{1}(f_n), & \tif \gk_n=t, \\
\nprb{1}(f_n), &\tif \gk_n=b.  
\end{cases}
\end{gather} 
Then, there are $\ga_n, \gb_n\in \BB{C}$, for $n\geq 1$, such that 
\begin{align*}
f_n'(0)=e^{2\pi\B{i}\ga_n}, \Re \ga_n \in (-1/2, 1/2] \\
f_n'(\gs_n)=e^{2\pi\B{i}\gb_n},  \Re \gb_n\in (-1/2, 1/2].
\end{align*}
We shall use the notations 
\begin{align*}
\ga_n=\ga(f_n)\in A(r_3), \gb_n= \gb(f_n), \gs_n=\gs(f_n),
\end{align*}
throughout the rest of this paper. 

The rotation numbers $\ga_n$ and $\gb_n$, for $n \geq 1$, are related by the formulas 
\begin{equation}\label{E:asymptotic-rotations}
\frac{1}{1-e^{2\pi \B{i} \ga_n}}+ \frac{1}{1-e^{2\pi \B{i} \gb_n}} =\frac{1}{2\pi i} 
\int_{\partial W} \frac{1}{z-f_n(z)} \, \mathrm{d} z, \;
\end{equation}
and 
\begin{equation}
 \ga_{n+1}=
\begin{cases} 
-1/\ga_n - [\Re (-1/\ga_n)] & \text{ if } \gk_n= t,\\
-1/\gb_n - [\Re (-1/\gb_n)] & \text{ if } \gk_n= b,
\end{cases}
\end{equation}
where $[ \cdot]$ denotes the closest integer function. 
We prove in the next section, Lemma~\ref{L:preliminary-estimate-on-Index}, that the absolute 
value of the right-hand side of Equation~\eqref{E:asymptotic-rotations} is uniformly bounded 
from above.
This implies that when some $\ga_n$ is small (and non-zero), then $\gb_n$ is small and the sign of 
$\Re \gb_n$ is equal to the sign of $\Re \ga_n$ times $-1$.

We say that a map is \textit{infinitely near-parabolic renormalizable} if there is $\gk\in \C{T}$ 
such that the map is infinitely near-parabolic renormalizable of type $\gk$. 
By a continuity argument one may see that for any $\gk\in \C{T}$, the set of infinitely 
near-parabolic renormalizable maps of type $\gk$ is non-empty. 
See Figure~\ref{F:cantor-sets-of-rotations}.

\begin{figure}[ht]
\begin{center}
\fbox{\includegraphics[scale=1]{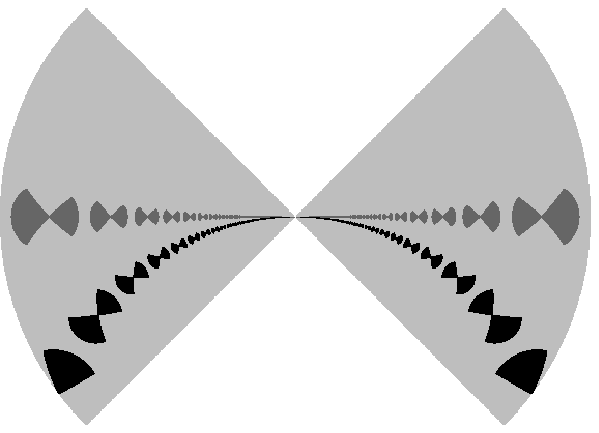}}
\caption{This is a schematic presentation of the values of the complex rotation $\ga$ 
where one and two iterates of the top and bottom near-parabolic renormalizations are defined.
The light gray region sows the set $A(r_3)$. 
The darker gray region is where $\operatorname{\C{R}_{\scriptscriptstyle NP-t}\co{2}}$ and 
$\operatorname{\C{R}_{\scriptscriptstyle NP-b}} \circ 
\operatorname{\C{R}_{\scriptscriptstyle NP-t}}$ are defined.  
On the black region the iterates  
$\operatorname{\C{R}_{\scriptscriptstyle NP-b}\co{2}}$ and 
$\operatorname{\C{R}_{\scriptscriptstyle NP-t}} \circ \operatorname{\C{R}_{\scriptscriptstyle NP-b}}$ 
are defined.} 
\label{F:cantor-sets-of-rotations}
\end{center}
\end{figure}

When a map is infinitely near-parabolic renormalizable, one may obtain fine scale understanding of 
the dynamics of that map. 
For example, for $\gk$ of constant type $\gk_i=t$ for all $i\geq 1$, this has lead to 
important properties of the dynamics of quadratic polynomials $Q_\ga$, with irrational $\ga$ of high 
type, in \cite{BC12}, \cite{Ch10-I}, \cite{Ch10-II}, \cite{AC12}, and \cite{CC13}.  
Thus, it is a significant problem to understand when a map is infinitely near-parabolic renormalizable. 
By virtue of Theorem~\ref{T:Ino-Shi2}, the covering structure of the top and bottom renormalizations
of a one time renormalizable map are determined. 
Thus, the answer to this question relies on controlling the multipliers of the successive renormalizations of the map at the origin.


\subsection{Cantor structure in the bifurcation loci}\label{S:rigidity}

Let $\gk=(\gk_1, \gk_2, \gk_3, \dots)\in \C{T}$, and $r\in (0, r_3]$, where $r_3$ is the constant 
obtained in Proposition~\ref{P:renormalization-top}. 
Define $\gL_r(\gk_1)=A(r)$. 
By Theorem~\ref{T:Ino-Shi2}, for every $f\in \IS$ and $\ga\in A(r)$, $\nprt{1}(\ga \ltimes f)$ 
and $\nprb{1}(\ga \ltimes f)$ are defined.
For integers $n\geq 1$, we consider
\begin{equation}\label{E:inpr-IS}
\gL_r(\langle  \gk_i \rangle_{i=1}^n)= 
\Big \{\ga \ltimes f \;  \Big | \;
\begin{array}{rr}
\operatorname{\C{R}_{\scriptscriptstyle NP-\gk_n}} \circ \dots \circ 
\operatorname{\C{R}_{\scriptscriptstyle NP-\gk_1}}(\ga \ltimes f) \text{ is defined} \\
\tand \forall i \text{ with } 1\leq i \leq n, \ga_{i} \in A(r).
\end{array}  
\Big \}.
\end{equation}
We recall that in the above definition, $\ga_i$ is the rotation number of the map 
$\operatorname{\C{R}_{\scriptscriptstyle NP-\gk_{i-1}}} \circ \dots \circ 
\operatorname{\C{R}_{\scriptscriptstyle NP-\gk_1}}(\ga \ltimes f)$ at $0$. 
In other words, $\gL_r(\gk_1, \dots, \gk_n)$ is the set of maps $\ga \ltimes f$ that are $n$ times
near parabolic renormalizable of type $\gk_1, \dots, \gk_n$ with the rotation number of all the  
successive renormalizations in the set $A(r)$. 
Given $\gk\in \C{T}$, one naturally defines 
\[\gL_r(\gk)= \bigcap_{n=1}^{\infty} \gL_r( \langle \gk_i\rangle_{i=1}^n).\]

The main result of this paper is stated in the following theorem.
We recall that $k_1$ and $r_4$ are the constant obtained in Proposition~\ref{P:k-horizontal}. 

\begin{thm}\label{T:Cantor-structure}
For all $k_1$-horizontal family of maps $\gU:A(r_4) \to \PC{A(r_4)}$ and all 
$\gk\in \C{T}$, every connected component of the set $\gL_{r_4}(\gk) \cap \gU(A(r_4))$ 
is a single point. 

In particular, for all $f\in \IS \cup\{Q_0\}$, and all $\gk\in \C{T}$, every connected component of 
the set $\gL_{r_4}(\gk) \cap (A(r_4) \ltimes f)$ is a single point. 
\end{thm}

\begin{proof}
Let $A_1$ be a connected component of $\gL_{r_4}(\gk) \cap \gU(A(r_4))$. 
By definition, for every $n\geq 1$, every point in $A_1$ is near parabolic renormalizable of type
$\langle \gk_i \rangle_{i=1}^k$. 
Inductively define the sets $A_{n+1}=\operatorname{\C{R}_{\scriptscriptstyle NP-\gk_{n}}}(A_n)$, 
for $n\geq 1$.
By Proposition~\ref{P:k-horizontal}, each $A_i$, $i \geq 1$, is a $k_1$-horizontal curves. 

Let us define the numbers $\gv_i$ as the diameter of the projection of the set $A_i$ onto 
the $\ga$ coordinate.
Then, since the diameter of each component of $A(r_4)$ is equal to $\sqrt{2} r_4$, we have 
$\gv_i \leq \sqrt{2} r_4$. 
On the other hand, by \refP{P:2-2-derivative}, $\nprt{1}$ and $\nprb{1}$ are strictly expanding 
in horizontal direction. 
We must have $\gv_i=0$, for all $i\geq 1$.  
It follows from the definition of $k_1$- horizontal curves that the diameter of each $A_i$ 
must be zero. 
Indeed, the uniform contraction implies that there are constants $C$ and $\gm \in (0,1)$ such that
for all $k_1$-horizontal family of maps $\gU:A(r_4) \to \PC{A(r_4)}$ and all $\gk\in \C{T}$ 
\[\diam \big( \gL_{r_4}(\langle \gk_i \rangle_{i=1}^n) \cap \gU(A(r_4)) \big) \leq C \gm^n.\]

The latter part of the theorem follows from the first part as the family $\gU(\ga)=(\ga, f)$ may be 
thought of a $0$-horizontal family. 
\end{proof}

The quadratic polynomial $P_c$, with $c\in \BB{C}$, is conformally conjugate to some $Q_{\ga(c)}$,
with $\Re \ga(c) \in (-1/2, 1/2]$. 
Indeed, there are two choices for $\ga(c)$ with this property. 
The choice does not make any difference for the sake of the next statement, 
although we will make one of these choices in the next section for our convenience. 
An immediate corollary of Theorem~\ref{T:Cantor-structure} applied to the quadratic family is 
formulated in the next corollary. 

It follows from the proof of the above theorem that the set  $\gL_{r_4}(\gk) \cap \gU(A(r_4))$ is isomorphic to a 
$\IS$-bundle over a Cantor set. This is, 
formulated in the next corollary. 

\begin{cor}
For all $\gk\in \C{T}$ the restriction of the map $\ga \ltimes h \mapsto h$ to each connected component of the set 
$\gL_{r_4}(\gk) \cap \gU(A(r_4))$ is one-to-one and onto whose image is equal to $\IS$. 
\end{cor} 

The operators $\nprt{1}$ and $\nprb{1}$ map the $\IS$-fibers of the set $\cup_{\gk \in \C{T}} \gL_{r_4}(\gk) \cap \gU(A(r_4))$
to the fibers. Combining with the uniform contraction in Theorem~\ref{T:Ino-Shi2}, we obtain the uniform contraction of the 
operators $\nprt{1}$ and $\nprb{1}$ on the co-dimension one fibers. 

\begin{cor}
Let $\langle p_i/q_i\rangle_{i=1}^\infty$ be a sequence of rational numbers in $(-1/2, 1/2]$ 
and $\gk\in \C{T}$ be a type such that for all $c\in M(\langle p_i/q_i\rangle_{i=1}^\infty)$,
$Q_{\ga(c)}$ is infinitely near parabolic renormalizable of type $\gk$ and for every 
$i\geq 1$ the rotation number $\ga_i$ of the $i$-th renormalization of $Q_{\ga(c)}$  
belongs to $A(r_4)$. 
Then, the nest of Mandelbrot copies $M(\langle p_i/q_i\rangle_{i=1}^n)$ shrinks (geometrically) 
to a single point as $n$ tends to infinity.  
\end{cor}

Although in the above corollary the sequence of rational numbers and the type $\gk$ are not related 
apriori, in Section~\ref{SS:multi-complex-rotation} we associate a canonic type $\gk$ to any given 
sequence of rational numbers in $(-1/2,1/2]$. 
\section{Application to the complex quadratic polynomials}\label{S:quadratic-family}
\subsection{Modified continued fractions}\label{SS:MCF}
We work with a modified notion of continued fractions that is more suitable in the study of the 
near-parabolic renormalization. 

For $x\in \BB{R}$, let $[x]$ denote the closest integer to $x$, where we use the convention 
\begin{align*}
x\in ([x]-1/2, [x]+1/2], \tfor x>0, \\
x\in [[x]-1/2, [x]+1/2), \tfor x<0.
\end{align*}
We have adapted the above convention to obtain a $x \mapsto -x$ symmetry in the continued fraction 
expansion introduced below. 

Let $x \in [-1/2, 1/2] \setminus \{0\}$ be a rational number and set $x_1=x$. 
There is a positive integer $n$ such that the numbers 
\begin{equation}\label{E:G-orbit}
[x_{i+1}]= \frac{-1}{x_i} - \left [ \frac{-1}{x_i}\right], 1\leq i\leq n,
\end{equation}
are defined and $x_{n+1}=0$. 
For $1\leq i \leq n$, we define the signs $\gep_i'=+1$ if $x_i>0$ and $\gep'_i=-1$ if $x_i<0$, and 
then define the integers $b_i\geq 2$ according to 
\[b_i= 
\begin{cases}
\big[ \frac{-1}{x_i}\big] & \tif \gep'_i=-1, \\  
\big[\frac{1}{x_i}\big]  &  \tif \gep'_i=+1.  
\end{cases}
\]
It follows that $[-1/x_i]=\gep'_i b_i$, and hence $x_{i}^{-1} = \gep'_i b_i - x_{i+1}$. 
Now let us define the signs $\gep_1= \gep_1'$, and $\gep_i=(-1)\gep'_{i-1} \gep'_i$, for $2\leq i \leq n$. 
Then, one can see that $x$ is given by the finite continued fraction 
\begin{align*}
x=x_1&=1/(\gep'_1 b_1- x_2)= \gep_1/(b_1- \gep_1' x_2), \\ 
&=\cfrac{\gep_1}{b_1- \gep'_1 \cfrac{1}{\gep'_2 b_2 - x_3}}=
\cfrac{\gep_1}{b_1 +\cfrac{-1 \gep_1' \gep_2'}{b_2 - \gep_2' x_3}}
= \cfrac{\gep_1}{b_1 + \cfrac{\gep_2}{b_2 - \gep_2' x_3}}.
\end{align*}
Inductively, repeating the above process until $x_{n+1}=0$, we obtain
\[x=\cfrac{\gep_1}{b_1+ \cfrac{\gep_2}{\ddots +\cfrac{\gep_n}{b_n}}}.\]

\begin{rem}
The above continued fraction expansion is slightly different from the usual notion of modified 
(closest integer) continued fraction expansion in the literature, where we use $x- [x]$ instead 
of $d(x, \BB{Z})$ and allow $x_i$ to be negative as well as positive. 
However, the only difference between the two expansions is in the signs $\gep_i$.
The reason for adapting to the above algorithm is that we shall later extend the map 
$x\mapsto x-[x]$ to a holomorphic map on a neighborhood of the interval $[-1/2, 1/2]$. 
This allows us to study the maps $\ga \mapsto \nprt{1}(\ga \ltimes f)$ and 
$\ga \mapsto \nprb{1}(\ga \ltimes f)$ as holomorphic maps of $\ga$ rather than anti-holomorphic 
maps of $\ga$. 
\end{rem}

Given $n\geq 1$ and a sequence of pairs $\langle b_i : \gep_i\rangle_{i=1}^n$,  
where each $b_i\geq 2$ is an integer and $\gep_i\in \{+1, -1\}$, we use the notation 
$[\langle b_i : \gep_i\rangle_{i=1}^n]$ to denote the rational number generated by 
this sequence of pairs. 
That is, 
\[[b_1: \gep_1]= \frac{\gep_1}{b_1}, \quad 
 [(b_1: \gep_1), (b_2 : \gep_2)]= \cfrac{\gep_1}{b_1+\cfrac{\gep_2}{b_2}}, \]
and for $n\geq 3$,
\[[\langle b_i : \gep_i \rangle_{i=1}^n] = 
\cfrac{\gep_1}{b_1+ \cfrac{\gep_2}{\ddots +\cfrac{\gep_n}{b_n}}}.\]
However, note that a rational number of the above form is not necessarily in the interval $[-1/2, 1/2]$. 
But, it is fairly close. 
The only condition we need to impose to obtain a rational number in the interval $[-1/2, 1/2]$
is that when $b_1=2$ we must have $\gep_1 \gep_2=+1$.


\subsection{Sequences of rational numbers}\label{SS:SRN}
Let $n\geq 1$ be an integer, and let $m_i\geq 1$ and $b_{i,j}\geq 2$ be integers, and 
$\gep_{i,j} \in \{+1, -1\}$, for $1\leq i \leq n$ and $1\leq j \leq m_i$.
We define the notation 
\[\langle m_i : b_{i,j} : \gep_{i,j}\rangle_{i=1}^n\] 
to represent the finite sequence of rational numbers
\[\langle [b_{i,j} :\gep_{i,j}]_{j=1}^{m_i} \rangle_{i=1}^n\]
For each $i$ and $l$ with $1\leq i \leq n$ and $1\leq l\leq m_i$ we let 
\[\frac{p_{i, l}}{q_{i,l}}=[\langle b_{i,j} : \gep_{i,j} \rangle_{j=1}^l].\]
Thus, 
\[\langle m_i : b_{i,j} : \gep_{i,j}\rangle_{i=1}^n= \langle \frac{p_{i, m_i}}{q_{i,m_i}}\rangle_{i=1}^n. \]
Although we may remove the second subscript $m_i$ from the numerator and denominator of the above 
sequence to get the simpler notation $p_i/q_i=p_{i,m_i}/q_{i,m_i}$, 
these have been put there to avoid possible future confusions that the sequence $p_i/q_i$ forms the 
fractions of a single number. 
That is, a priori there is no relation between these fractions.
When we are interested in infinite sequences of rational numbers, we shortened the notation 
$\langle m_i : b_{i,j} : \gep_{i,j}\rangle_{i=1}^\infty$ to $\langle m_i : b_{i,j} : \gep_{i,j} \rangle$.

By setting the initial data $p_{i, -1}=q_{i,0}=1$ and $p_{i,0}=q_{i,-1}=0$, we have the usual 
recursive formulas for the continued fractions $p_{i,l}/q_{i,l}$, $0\leq l \leq m_i$
\begin{equation*}
q_{i,l+1}=b_{i,l+1}q_{i,l} + \gep_{i,l} q_{i,l-1}, \quad 
p_{i,l+1}=b_{i,l+1}p_{i,l} + \gep_{i,l} p_{i,l-1},
\end{equation*}
By an inductive process, the above formulas imply that for all $i\geq 1, 0\leq l \leq m_i-1$, we have 
\begin{equation}\label{E:increasing-denominator}
q_{i,l+1}> q_{i,l}.
\end{equation}

Moreover, for every $i\geq 1$, 
\begin{equation}\label{E:q-vs-beta}
\frac{1}{q_{i,m_i}}= \Big |  \prod_{k=1}^{m_i} [\langle b_{i,j} : \gep_{i,j} \rangle_{j=k}^{m_i}]  \Big | .
\end{equation}


\subsection{Pairs of multipliers vs pairs of complex rotations}\label{SS:multi-complex-rotation}
Let $\gk=(\gk_1, \gk_2, \gk_3, \dots)$ be in $\C{T}$, where the set of types $\C{T}$ is defined in 
Equation~\eqref{E:types}.
In analogy with the set $\gL_r(\gk_1, \dots , \gk_n)$ defined in Equation~\eqref{E:inpr-IS}, we
consider the sets  
\[\gL_r^1(\langle \gk_i \rangle_{i=1}^n)= \Big \{\ga\in A(r_3) \Big | 
\begin{array}{rr}
\operatorname{\C{R}_{\scriptscriptstyle NP-\gk_n}} \circ \dots \circ 
\operatorname{\C{R}_{\scriptscriptstyle NP-\gk_1}}(Q_\ga) \text{ is defined} \\
\tand \forall i \text{ with } 1\leq i \leq n, \ga_{i} \in A(r).
\end{array}  
\Big \}\]  
For example, by Theorem~\ref{T:Ino-Shi2}, $\nprt{1}(Q_\ga)$ and $\nprb{1}(Q_\ga)$ are defined for 
$\ga \in \gL_{r_3}^Q(\gk_1)$. 

Each $Q_\ga$ with $\ga \in \BB{C}$ is conformally conjugate to some quadratic polynomial 
$P_c$ with a unique $c\in \BB{C}$. 
The connectedness locus of the family $Q_\ga$, that is, the set of $\ga$ such that the Julia set of 
$Q_\ga$ is connected, is $\BB{Z}$-periodic in $\ga$. 
However, this connectedness locus modulo $\BB{Z}$ forms a double cover of the Mandelbrot set, 
branched over $c=1/4$ (which is only covered once). 
Here $\ga=0$ is mapped to $c=1/4$. 
For each $c\in \BB{C}\setminus \{1/4\}$ there are two distinct parameters $w_1$ and $w_2$ 
in $\BB{C}$ such that for all $\ga$ in $w_1+\BB{Z}$ and $w_2+\BB{Z}$, $Q_\ga$ is 
conformaly conjugate to $P_c$.
See Figure~\ref{F:Mand-rot-coord}.

For $\ga$ in the upper half plane, $0$ is an attracting fixed point of $Q_\ga$, while for $\ga$ in $\BB{R}$
the multiplier of $Q_\ga$ at $0$ belongs to the unit circle.
In analogy to the Mandelbrot set, 
the connectedness locus of $Q_\ga$ minus $\BB{R}$ consists of the upper half plane and 
the connected components attached to the real line at rational values of $\ga$.
There is a unique connected component attached to the real line at $0$. 
We denote the closure of this component by $M_\ga$. 
Then, there is a one-to-one correspondence between the Mandelbrot set and $M_\ga$ such that 
the corersponding quadratic polynomials are conformally conjugate. 
Let $c\mapsto \ga(c)$, from $M$ to $M_\ga$, denote this bijection. 
We define the sets 
\[M_\ga( \big\langle \frac{p_i}{q_i} \big\rangle_{i=1}^n)
=\big \{\ga (c) \mid c \in M(\big\langle \frac{p_i}{q_i} \big\rangle_{i=1}^n)\big\}.\]
As in Section~\ref{S:poly-like-renormalization}, the notions of dividing periodic points 
and their combinatorial rotation numbers are defined on the above components.  

\begin{figure}\label{F:Mand-rot-coord}
\fbox{\includegraphics[scale=1]{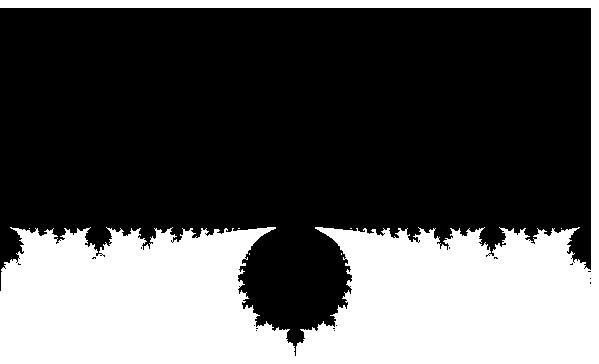}} \hspace*{2em}
\fbox{\includegraphics[scale=1]{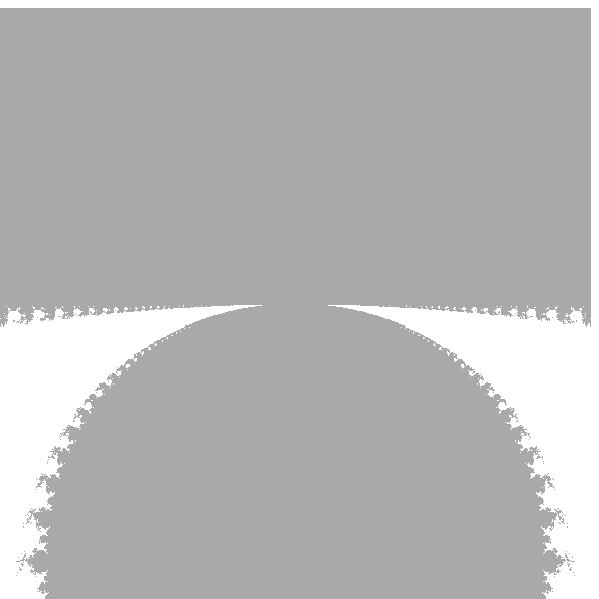}}
\caption{The connectedness locus of the family $Q_\ga$ and a zoom into a neighborhood of $0$ on 
the right hand side. Need to produce better figures.}
\end{figure}

Let $Q_\ga$ be an infinitely polynomial-like renormalizable.
We denote the sequence of the dividing periodic points of $Q_\ga$ by $\gn_i$, $i\geq 1$. 
In other words, $\gn_i$ in the periodic point of $Q_\ga$ that is mapped to $a_i$ by the conformal map 
conjugating $Q_{\ga(c)}$ to $P_c$, for each $i\geq 1$.
We recall that $\gn_1$ is a fixed point.
Note that $\gn_i$ is not an arbitrary element in the cycle of $\gn_i$. 
Rather, it is a particular point in this cycle.
Let us denote the multiplier of $\gn_i$ by $\gr_i$, for $i\geq 1$.
That is, 
\[\gr_1= Q_\ga'(\gn_1), \; \gr_2= (Q_\ga\co{q_1})'(\gn_2), \;
\gr_n= (Q_\ga \co{(q_1\dots q_{n-1})})'(\gn_n),  n\geq 3.\]

For $\ga\in M_\ga\setminus\{0\}$, $0$ is a repelling fixed point of $Q_\ga$ while $\gs(Q_\ga)$ 
may be either attracting or repelling, depending on $\Im (\gb(Q_\ga))$. 
For $\ga \in M_\ga\setminus\{0\}$, $\gn_1=\gs_1=\gs(Q_\ga)$ is the non-zero fixed point of $Q_\ga$, 
and we have 
\begin{equation}\label{E:first-multiplier}
\gb_1 = \frac{1}{2\pi \B{i}} \log \gr_1.
\end{equation}

Let $p_i/q_i$ be a sequence of non-zero rational numbers in $(-1/2, 1/2]$. 
By the previous section, there is $\langle m_i : b_{i,j} : \gep_{i,j} \rangle_{i=1}^\infty$, 
with integers $m_i\geq 1$, $b_{i,j}\geq 2$, and signs $\gep_{i,j}$ for $i\geq 1$ and $1\leq j\leq m_i$, such 
that $p_i/q_i= p_{i,m_i}/q_{i,m_i}=[\langle b_{i,j} : \gep_{i,j}\rangle_{i=1}^{m_i}]$, for $i\geq1$.  
Then, we consider the integers 
\begin{equation}\label{E:sub-levels}
l_1=0, \;   l_k=\sum_{i=1}^{k-1} m_i, k\geq 1.
\end{equation}
Then, we define the map 
\begin{equation}\label{E:kappa}
\gk: \big(\BB{Q}\cap \big((-1/2,0) \cup (0,1/2]\big)\big)^\BB{N} \to \C{T},
\end{equation}
as follows. 
Given $\langle m_i : b_{i,j} : \gep_{i,j}\rangle$ in 
$ \big(\BB{Q}\cap \big((-1/2,0) \cup (0,1/2]\big)\big)^\BB{N}$ and $n\geq 1$, 
the $n$-th entry of $\gk(\langle m_i : b_{i,j} : \gep_{i,j}\rangle)$, denoted by 
$\gk_n( \langle m_i : b_{i,j} : \gep_{i,j}\rangle)$, is defined as 
\[\gk_n(\langle m_i :  b_{i,j} : \gep_{i,j}\rangle)=
\begin{cases}
b & \tif n \in \{l_k+1 \mid k\geq 1\} \\
t & \text{ otherwise.}
\end{cases}\]
In other words, the first $m_1$ entries of $\gk$ are given by one ``b'' followed by $m_1-1$  times ``t'',
the next $m_2$ entries of $\gk$ are given by one ``b'' followed by $m_2-1$ times ``t'', and so on.  
The map $\gk$ only depends on the sequence $m_i$ (the lengths of the rational numbers) rather 
than the entries in each rational number. 
The individual entries come into play later. 
 
In the following proposition $r_3$ denotes the constant introduced in 
Proposition~\ref{P:renormalization-top}. 
\begin{propo}\label{P:multiplier-rotation}
Let $\langle p_i/q_i\rangle_{i=1}^\infty= \langle m_i : b_{i,j} : \gep_{i,j}\rangle$ 
be a sequence of non-zero rational numbers in $(-1/2, 1/2]$ 
and consider the type $\gk=\gk(\langle m_i : b_{i,j} : \gep_{i,j}\rangle)\in \C{T}$. 
For every integer $k\geq 0$ and every $\ga$ in the intersection of 
$M_\ga(\langle p_i/q_i\rangle_{i=1}^k)$ and $\gL_{r_3}^1(\gk_1, \dots, \gk_{l_k})$, if 
$\ga_{l_k+1} \in A(r_3)$ then 
\[\gb_{l_k+1} =\frac{1}{2\pi \textnormal{\B{i}}} \log \gr_k.\]
\end{propo}

\begin{proof}
Since $\ga \in M_\ga (p_1/q_1, \dots, p_k/q_k)$, for every integer $j$ with $1\leq j \leq k$ 
the dividing periodic point $\gn_j$ of $Q_\ga$ is defined. 
In particular, the right hand side of the equality in the proposition is defined.
On the other hand, as $\ga\in \gL(\gk_1, \dots, \gk_{l_k})$, for every integer $j$ with $2\leq j \leq l_k$
$f_{j+1}= \operatorname{\C{R}_{\scriptscriptstyle NP-\gk_j}} \circ \dots \circ 
\operatorname{\C{R}_{\scriptscriptstyle NP-\gk_1}}(Q_\ga)$ is defined. 
Moreover, by the definition of $\gL(\gk_1, \dots, \gk_{l_k})$ and the assumption in the proposition 
for every $1\leq j \leq l_k+1$, $\ga_j$ is defined and belongs to $A(r_3)$. 
By Proposition~\ref{P:sigma-fixed-point}, this implies that each $f_j$ has a unique non-zero 
fixed point $\gs_{j}$ in the (fixed) neighborhood $W$ of $0$. 
In particular, $\gb_{l_k+1}$ in the left-hand side of the equation in the proposition is defined. 
It remains to relate these quantities. 

Recall that for $\ga \in M_\ga (p_1/q_1, \dots, p_k/q_k)$, $\gn_1$ is a fixed point, and in general
for $j$ with $2\leq j \leq k$, $\gn_j$ is a periodic point of period $\prod_{i=1}^{j-1} q_i$. 
Moreover, as $\ga$ varies in $M_\ga (p_1/q_1, \dots, p_j/q_j)$, $\gn_j$ has 
holomorphic dependence on $\ga$. 
On the other hand, by the definition of the near-parabolic renormalizations, each $\gs_{l_j+1}$, 
for $1\leq j \leq k$, lifts to a periodic cycle of $Q_\ga$, which we denote by $O_{l_j+1}$.
We claim that for each $j$ with $1\leq j \leq k$, $O_{l_j+1}$ is equal to the cycle of $\gn_j$. 
We prove this below by induction on $j$. 

For $j=1$, $l_1=0$ and $O_1$ is equal to $\gs_1=\gs(Q_\ga)=\gn_1$. 
Assume that $O_{l_{j-1}+1}$ is equal to the cycle of $\gn_{j-1}$ for $j-1< l+1$ 
and we want to prove that $O_{l_j+1}$ is equal to the cycle of $\gn_j$. 
By the definition of the types, $\gk_{l_{j-1}+1}=b$ and for all integers $i$ with $l_{j-1}+1 < i \leq l_j$, 
we have $\gk_i=t$. 
This implies that the zero fixed point of $f_{l_j+1}$ lifts to $\gs_{l_{j-1}+1}$ on the dynamic pane of 
$f_{l_{j-1}+1}$ (in all the intermediate levels $i$ it lifts to $0$). 
Thus, by the induction hypothesis, the zero fixed point of $f_{l_j+1}$ lifts to the cycle of $\gn_{j-1}$. 
On the other hand, as $\ga_{l_j+1}$ tends to $0$ in $A(r_3)$, $\gs_{l_j+1}$ tends to $0$ and 
$\gs_{l_j+1}$ is the only fixed point of $f_{l_j+1}$ within $W$-neighborhood of $0$. 
This implies that the lift of this fixed point, which is the cycle $O_{l_j+1}$, tends to the cycle of 
$\gn_{j-1}$. 
However, as $\ga\in M(p_1/q_1, \dots, p_j/q_j)$, $\gn_j$ is the only periodic point of 
$Q_\ga$ that bifurcates from $\gn_{j-1}$. 
That is, for sufficiently small $\ga_{l_j+1}$ (equivalently, for $\ga$ sufficiently close to the root point of 
$M(p_1/q_1, \dots, p_j/q_j)$), $O_{l_j+1}$ is equal to the cycle of $\gn_j$.
By the holomorphic dependence of the cycles of $\gn_j$ and $O_{l_j+1}$ on $\ga$, we conclude that 
these cycles are equal on the connected components of $\gL(\gk_1, \dots, \gk_{l_j})$. 

By the above argument, $\gs_{l_k+1}$ lifts to the orbit of $\gn_k$ through the changes of 
the coordinates in the near-parabolic renormalizations. 
In particular, these cycles must have the same multipliers, as in the equation in the proposition. 
\end{proof}

\begin{cor}
Let $\langle m_i : b_{i,j} : \gep_{i,j}\rangle$ be a sequence of non-zero rational numbers in 
$(-1/2, 1/2]$. 
For every $\ga$ in the intersection of $M_\ga(\langle m_i : b_{i,j} : \gep_{i,j}\rangle)$ and 
$\gL_{r_3}^1(\gk(\langle m_i : b_{i,j} : \gep_{i,j}\rangle))$, and every $k\geq 1$ we have 
\[\gb_{l_k+1} = \frac{1}{2\pi \textnormal{\B{i}}} \log \gr_k.\]
\end{cor}


\subsection{The quadratic growth condition and the complex rotations}

Here, our goal is to find a combinatorial condition, in terms of the combinatorial rotation numbers 
of the dividing periodic points, that guarantees an infinitely polynomial-like renormalizable map under that 
condition is infinitely near-parabolic renormalizable.

The algorithm defining the modified continued fraction expansion in Section~\ref{SS:MCF} 
has a natural extension onto the complex plane which plays a crucial role in this section.     
It is defined as follows.  
Recall from Section~\ref{SS:MCF} the closest integer function $[\cdot]$ defined on $\BB{R}$. 
We consider the map $\inv(z)=-1/z$, for $z\in \BB{C}\setminus \{0\}$, and the map 
$\saw(z)=z-[\Re (z)]$, for $z\in \BB{C}$. 
We shall work with the composition of these two maps denote by 
\[G(z)= \saw \circ \inv(z)   , \forall z\in \BB{C}\setminus \{0\}.\]
Then, $G$ maps the interval $[-1/2, 0) \cup (0,1/2]$ to $[-1/2, 1/2]$, and 
for a non-zero rational number in $x_1$ in $[-1/2, 1/2]$, 
\[x_i=G\co{i-1}(x_1), 1\leq i \leq n.\] 
where the numbers $x_i$ are defined in Equation~\eqref{E:G-orbit}. 
In particular, 
\[G([\langle b_j : \gep_j\rangle_{j=i}^n])=- \gep_1 \cdot [\langle b_j : \gep_j\rangle_{j=i+1}^n].\] 
That is, applying $G$ to a continued fraction, removes the first pair from the expansion, and 
then only modifies the first sign in the remaining expansion. 
In particular, 
\[G\co{n-1}([\langle b_j : \gep_j\rangle_{j=1}^n])= \pm 1/b_n, \; 
G\co{n}([\langle b_j : \gep_j\rangle_{j=1}^n])=0.\] 

For each integer $b_1\geq 2$ and $\gep_1\in \{+1, -1\}$, the image of the round 
ball of radius $1/2$ centered at $-\gep_1 b_1$, $B(-\gep_1 b_1, 1/2)$, under 
the map $\inv$ is a round ball containing $\gep_1/b_1$. 
Note that this ball is not centered at $\gep_1/b_1$. 
Let $\C{F}_1$ denote the collection of all these balls for integers $b_1\geq 2$ and $\gep_1 \in\{+1, -1\}$. 
If we care to determine a specific ball in this collection, we use the notation 
$\C{F}_1(\langle b_1 :  \gep_1\rangle )$ to denote the one containing $\gep_1/b_1$. 
It follows that $G:\C{F}_1(\langle b_1 :  \gep_1\rangle ) \to B(0, 1/2)$ is a holomorphic bijection.  

Similarly, for integers $n\geq 2$, we may define the collection $\C{F}_n$ of round balls 
that are mapped onto $B(0, 1/2)$ by the iterate $G\co{n}$. 
The element of $\C{F}_n$ containing $[\langle b_1 :  \gep_1\rangle_{i=1}^n]$ is 
denoted by $\C{F}_n(\langle b_i :  \gep_i \rangle_{i=1}^n)$ and we note that each such element 
is a disk that is symmetric with respect to the real line. 
Moreover, 
\[G\co{n}:\C{F}_n([\langle b_i :  \gep_i \rangle_{i=1}^n]) \to B(0, 1/2), n\geq 1,\] 
is a holomorphic bijection given by a M\"obius transformation that maps the real slice of the 
domain to the real slice of the image.  
 
\begin{lem}\label{L:F-balls-in-cone}
For every $n\geq 1$, every $[\langle b_j :  \gep_j \rangle_{j=1}^n] \in \BB{Q}$, every 
$z\in \C{F}_n([\langle b_j :  \gep_j \rangle_{j=1}^n])$, and every $k$ with $0\leq k \leq n-1$, 
we have 
\begin{gather*}
 \frac{4}{5} \cdot \frac{1}{b_{k+1}}  \leq |G\co{k}(z)| \leq \frac{4}{3}\cdot  \frac{1}{b_{k+1}}, \quad 
\arg G\co{k}(z)\in 
\begin{cases}
[\frac{-\pi}{4}, \frac{\pi}{4}] &   \tif \Re (G\co{k}(z))>0 \\
[\frac{3\pi}{4}, \frac{5\pi}{4}] & \tif \Re (G\co{k}(z))<0
\end{cases}. 
\end{gather*}
\end{lem}

\begin{proof} 
First note that 
\[G\co{k}\big (\C{F}_n([\langle b_j :  \gep_j \rangle_{j=1}^n]\big)
=\C{F}_{n-k}(G\co{k}([\langle b_j :  \gep_j \rangle_{j=1}^n])
\ci \C{F}_{1}(G\co{k}([\langle b_j :  \gep_j \rangle_{j=1}^n]).\]
and the first pair in the expansion of $G\co{k}([\langle b_j :  \gep_j \rangle_{j=1}^n]$ has the form 
$(b_{k+1}, \pm1)$. 
Hence, $G\co{k}(z)$ belongs to either $\C{F}_1(1/b_{k+1})$ or $\C{F}_1(-1/b_{k+1})$. 
Each of these sets is a round ball symmetric with respect to the real line passing through the pair of 
points $1/(b_{k+1}+1/2)$ and $1/(b_{k+1}-1/2)$ or the pair of points $-1/(b_{k+1}+1/2)$ 
and $-1/(b_{k+1}-1/2)$, respectively. 
In particular, for $G\co{k}(z) \in \C{F}_1(\pm 1/ b_{k+1})$, 
\[ \frac{4}{5} \cdot \frac{1}{b_{k+1}} \leq \frac{1}{b_{k+1}+1/2}
\leq |z| 
\leq \frac{1}{(b_{k+1}-1/2)} \leq \frac{4}{3} \cdot \frac{1}{b_{k+1}}.\] 
On the other hand, each of $\C{F}_1(1/ b_{k+1})$ and $\C{F}_1(-1/ b_{k+1})$ 
is a round disk of diameter 
\[\frac{1}{b_{k+1}-1/2} - \frac{1}{b_{k+1}+1/2} \leq \frac{1}{\sqrt{2} b_{k+1}}.\]
Hence, $\C{F}_1(\pm 1/ b_{k+1})$ is contained in the round ball of 
radius $1/(\sqrt{2} b_{k+1})$ about $\pm 1/ b_{k+1}$.  
This implies the bounds on $\arg (G\co{k}(z))$.  
\end{proof}

\begin{lem}\label{L:simple-q-growth}
There is a constant $C_0$ satisfying the following.
Let $x=[\langle b_j: \gep_j \rangle_{j=1}^n]  \in \BB{Q}$, for some $n\geq 1$, and assume that 
$b_{j+1} \geq b_j^2$, for $1\leq j \leq n-1$. 
Define the numbers 
\[x_0=1, x_i=G\co{(i-1)}(x), \quad 1 \leq i \leq n.\]
Then, 
\[ \frac{|x_n|}{|x_1|} \leq  C_0 \cdot  \prod_{i=0}^{n-1} |x_i|.\]
\end{lem}

\begin{proof}
For each $1\leq j\leq n$ we have 
\[\big(1+ \frac{1}{2b_j+1}\big) \frac{1}{b_j} = \frac{1}{b_j+1/2}
\leq |x_j| 
\leq \frac{1}{b_j-1/2} = \frac{1}{b_j} \cdot \big(1+ \frac{1}{2b_j-1}  \big).\]
Then, since $b_{j+1} \geq b_j^2$, by the above equation we have 
\begin{multline*}
|x_{j+1}| 
\leq \frac{1}{b_{j+1}} \big(1+ \frac{1}{2b_{j+1}-1}\big) \\
\leq \frac{1}{b_j^2} \big(1+ \frac{1}{2b_{j+1}-1}\big) 
\leq |x_j|^2 \big(1+ \frac{1}{2b_{j+1}-1}\big) \big(1- \frac{1}{2b_j+2}\big)^2 \\
\leq |x_j|^2  (1+ \frac{C}{b_j}), 
\end{multline*}
For some uniform constant $C$.
Then, 
\begin{equation*}
\frac{|x_n|}{\prod_{i=0}^{n-1} |x_i|}
 = \prod_{i=0}^{n-1} \big( \frac{|x_{i+1}|}{|x_{i}|^2}\big) 
 \leq |x_1| \big(\prod_{i=1}^{n-1} (1+\frac{C}{b_i})\big) 
 \leq |x_1| \exp (\sum_{i=1}^\infty \frac{C}{b_i}) 
 \leq |x_1| e^C. 
\end{equation*}
\end{proof}

\begin{lem}\label{L:distortion-algorithm}
There exists $C_1>0$ such that for every $n\geq 1$ and every disk $B_n$ in 
$\C{F}_n$, the distortion of the map $G\co{n}: B_n \to B(0,1/2)$ is bounded by $C_1$, that is, 
\[\forall z,w\in B_n, \frac{1}{C_1} \leq \Big| \frac{(G\co{n})'(z)}{(G\co{n})'(w)}\Big | \leq C_1.\] 
\end{lem}
\begin{proof}
Assume that $B$ and $B'$ are round disks that are symmetric with respect to the real line 
(invariant under complex conjugation), and $g: B \to B'$ be a M\"obius map that sends 
$B\cap \BB{R}$ to $B'\cap \BB{R}$. 
Then, the distortion of $g$ on $B$ is realized at the two end points of the interval $B\cap \BB{R}$. 
This implies that the distortion of the map $G\co{n}$ on $B_n$ is equal to its distortion on 
$B_n\cap \BB{R}$. 
Indeed, the latter statement is a classical result that follows from direct calculations. 
\end{proof}

Let $N\geq 2$ be an integer and define the class of sequences of rational 
numbers\footnote{$\C{QG}$ stands for \textit{quadratic growth}.} 
\begin{equation}\label{E:qg-def}
\qg=\Big\{ \big \langle \frac{p_{i,m_i}}{q_{i,m_i}})\big \rangle_{i=1}^\infty 
=\langle m_i : b_{i,j} : \gep_{i,j} \rangle \Big|  
\begin{array}{l}
b_{1,1}\geq N ;  \forall i\geq 1, b_{i+1,1} \geq q_{i,m_i}^2 \\
\forall i\geq 1, 1\leq j \leq m_i-1, b_{i,j+1} \geq b_{i,j}^2.
\end{array}
\Big\}.
\end{equation}
For example, if we let $m_i=1$ for all $i\geq 1$, then a sequence 
$\langle p_{i,1}/q_{i,1}\rangle_{i=1}^\infty$ belongs to $\qg$ if and only if $q_1\geq N$ and 
for all $i\geq 1$, $p_i \in \{+1, -1\}$ and $q_{i+1}\geq q_i^2$.   
However, choosing larger values of $m_i$ at different stages allows us to cover more 
rational number at stage $i$. 
This imposes a stronger condition on the size of the next denominator through 
$q_{i+1}\geq b_{i, m_i}^2$. 

\begin{propo}\label{P:rationals-with-growth}
For every $r\in (0,1/2)$ there is a constant $N>0$ satisfying the following property. 
Let $\langle p_i/q_i \rangle_{i=1}^\infty$ belong to $\qg$ and define the integers 
\[k_1=1, k_n=\prod_{i=1}^{n-1} q_{i, m_i}, \forall i\geq 1.\]
Then, for every $i\geq 1$ and every $z\in \BB{C}$ satisfying 
\[\Big|z - \big(\frac{p_i}{q_i}- \textnormal{\B{i}}\frac{k_i}{q_i} \frac{\log 2}{2\pi} \big)\Big| 
\leq \frac{\log 2}{2\pi}  \frac{k_i}{q_i},\]
we have the following two properties. 
\begin{itemize}
\item[a)] For every $n_i$ with $0\leq n_i \leq m_i-2$,
\[|G\co{n_i}(z)| \leq r,  \tand 
\arg (G\co{n_i}(z)) \in 
\begin{cases}
[-\frac{\pi}{4},\frac{\pi}{4}] & \tif \Re (G\co{n_i}(z)) >0  \\
[\frac{3\pi}{4},\frac{5\pi}{4}] & \tif  \Re (G\co{n_i}(z)) <0
\end{cases};\]
\item[b)] 
\[|G\co{(m_i-1)}(z)| \leq r.\]
\end{itemize}
\end{propo} 

\begin{proof}
Let us choose $\langle m_i : b_{i,j} : \gep_{i,j} \rangle$ so that 
\[\frac{p_i}{q_i}=\frac{p_{i,m_i}}{q_{i,m_i}}= [\langle b_{i,j} : \gep_{i,j} \rangle_{j=1}^{m_i}].\]  
Recall the constants $C_0$ and $C_1$ introduced in Lemmas~\ref{L:simple-q-growth} and 
\ref{L:distortion-algorithm}. 
Then, choose $N\geq 2$ such that the following inequalities hold. 
\begin{gather}
\frac{1}{N} \frac{4}{3}\leq \frac{r}{2}, \label{E:N-condition-1} \\
\frac{1}{N} \frac{4}{3} C_1 
\leq \frac{1}{C_1} \cdot \frac{\sqrt{C_1}-1}{\sqrt{C_1}+1}\cdot  \frac{r}{2}, \label{E:N-condition-2}\\ 
\frac{1}{N}  C_1 \frac{4}{3} \frac{\log 2}{2\pi} \frac{5}{4} C_1 C_0 
\leq \frac{1}{C_1} \cdot \frac{\sqrt{C_1}-1}{\sqrt{C_1}+1}\cdot  \frac{r}{2} \label{E:N-condition-3}
\end{gather}

We break the proof into several steps. 

\medskip

{\em Step 1.}
For $\langle m_i : b_{i,j} : \gep_{i,j} \rangle \in \qg$, by Lemma~\ref{L:F-balls-in-cone}, 
for every $i\geq 1$ and every $n_i$ with $0\leq n_i \leq m_i-1$, by Equation~\eqref{E:N-condition-1},
we have 
\begin{equation}\label{E:preliminary-bounds}
\big| G\co{n_i} \big ([ \langle b_{i,j} : \gep_{i,j} \rangle_{j=1}^{m_i}]\big ) \big| 
\leq \frac{4}{3}\cdot \frac{1}{b_{i,n_i+1}}
\leq  \frac{4}{3}\cdot \frac{1}{b_{i,1}}
\leq \frac{4}{3}\cdot  \frac{1}{b_{1,1}}
\leq \frac{4}{3}\cdot  \frac{1}{N}
\leq \frac{r}{2}.
\end{equation} 

\medskip

{\em Step 2.} 
By the definition of the class $\qg$ and Equation~\eqref{E:increasing-denominator}, 
for every $i \geq 1$, 
\begin{equation}\label{E:q-increment}
q_{i+1, m_{i+1}} \geq q_{i+1, 1}=b_{i+1,1}\geq q_{i, m_i}^2.
\end{equation}
For $i=1$, by Equation~\eqref{E:N-condition-1}, we have 
\begin{equation}\label{E:period-over-rays-1}
\frac{\log 2}{2\pi} \cdot \frac{k_1}{q_{1,m_1}} 
=\frac{\log 2}{2\pi} \cdot \frac{1}{q_{1,m_1}} 
\leq \frac{\log 2}{2\pi} \cdot \frac{1}{b_{1, 1}}  
\leq \frac{\log 2}{2\pi} \cdot \frac{1}{N} 
\leq \frac{r}{2}.
\end{equation}
Recall that $q_{i,0}=1$, for $i\geq 1$. 
This is for the simplicity of the formulas in the following. 
For $i\geq 2$,   
\begin{equation}\label{E:period-over-rays}
\begin{aligned}
\frac{\log 2}{2\pi} \cdot  \frac{k_i}{q_{i,m_i}} 
&= \frac{\log 2}{2\pi}  \cdot \Big ( \prod_{l=1}^{i-1} q_{l, m_l} \Big ) \cdot \frac{1}{q_{i,m_i}} &&  \\
&\leq \frac{\log 2}{2\pi} \prod_{l=0}^{i-1} \Big(\frac{q_{l, m_l}^2}{q_{l+1, m_{l+1}}}\Big)  &&
\qquad (\text{Eq.}~\eqref{E:q-increment}) \\
&\leq \frac{\log 2}{2\pi} \cdot \frac{1}{q_{1,1}} \leq \frac{\log 2}{2\pi} \cdot  \frac{1}{b_{1,1}} && \\
&\leq \frac{\log 2}{2\pi} \cdot \frac{1}{N} \leq \frac{r}{2}. 
\end{aligned}
\end{equation}

\medskip

{\em Step 3.}
Assume that $m_i=1$ for some $i\geq 1$.
For every $z$ satisfying the hypothesis of the proposition, by Equations~\eqref{E:preliminary-bounds}, 
\eqref{E:period-over-rays-1}, and \eqref{E:period-over-rays}, 
$z$ belongs to a disk of diameter bounded from above by $r/2$ attached to the real line at 
$p_{i,m_i}/q_{i,m_i}$ with $|p_{i,m_i}/q_{i,m_i}| \leq r/2$. 
Hence, $|z|\leq r/2+ r/2=r$. 
This implies the inequality in part b), and there is nothing to prove for part a).

\medskip

Let us fix an $i\geq 1$. 
From now on we assume that $m_i>1$. 

\medskip

{\em Step 4.}
Recall that 
\[G\co{(m_i-1)}: \C{F}_{m_i-1}([\langle b_{i,j}: \gep_{i,j} \rangle_{j=1}^{m_i-1}]) \to B(0,1/2)\] 
is a bijection.   
Let us define the set 
\begin{equation}\label{E:ball-in-ball}
B_i^r \subseteq \C{F}_{m_i-1}([\langle b_{i,j}: \gep_{i,j} \rangle_{j=1}^{m_i-1}])
\end{equation}
as the pre-images of $B(0,r)$ under the above restriction of $G\co{(m_i-1)}$. 
Then, $B_i^r$ is a round ball containing the point $[\langle b_{i,j}: \gep_{i,j} \rangle_{j=1}^{m_i-1}]$. 

Define the numbers $x_{i,l}$, for $i\geq 1$, as 
\[x_{i,0}=1,\]
as well as the numbers $x_{i,l}$, for $i\geq 1$ and $1\leq l \leq m_i-1$, as 
\[x_{i,l}= \prod_{k=1}^l  \Big| [\langle b_{i,j} : \gep_{i,j} \rangle_{j=k}^{m_i-1}] \Big|=  
\prod_{k=0}^{l-1} \big| G\co{k}([\langle b_{i,j}: \gep_{i,j} \rangle_{j=1}^{m_i-1}])\big|.\]
For every $l$ with $0 \leq l \leq m_i-1$ we have 
\begin{equation}\label{E:G-derivative}
\big |(G\co{l})'([\langle b_{i,j}: \gep_{i,j} \rangle_{j=1}^{m_i-1}])\big|= (x_{i,l})^{-2}.
\end{equation}
Hence, by the uniform bound on the distortion of $G\co{m_i-1}$ in Lemma~\ref{L:distortion-algorithm}, 
we have 
\begin{equation}\label{E:size-of-balls}
r \frac{1}{C_1} x_{i,m_i-1}^2 \leq \diam (B_i^r) \leq r C_1 x_{i,m_i-1}^2 .
\end{equation}

Although the ball $B_i^r$ is not centered at $[\langle b_{i,j}: \gep_{i,j} \rangle_{j=1}^{m_i-1}]$, 
the uniform bound on the distortion of $G\co{m_i-1}$ implies that the center of $B_i^r$ is not too far 
from $[\langle b_{i,j}: \gep_{i,j} \rangle_{j=1}^{m_i-1}]$. 
One can verify that if $h$ is a M\"obius map of the unit disk with 
$1/C_1\leq |h'(x)|/|h'(y)|\leq C_1$ for all $x$ and $y$ in the unit disk, 
then $|h(0)| \leq (\sqrt{C_1}-1)/(\sqrt{C_1}+1)$. 
By virtue of Lemma~\ref{L:distortion-algorithm},  
\begin{equation}\label{E:bdd-deviation-from-center}
B \Big([\langle b_{i,j}: \gep_{i,j} \rangle_{j=1}^{m_i-1}],
\frac{1}{C_1}\frac{\sqrt{C_1}-1}{\sqrt{C_1}+1} x_{i,m_i-1}^2 r\Big) 
\subseteq B_i^r. 
\end{equation}
That is, $B_i^r$ contains a round ball of size comparable to $r x_{i,m_i-1}^2$ about 
$[\langle b_{i,j}: \gep_{i,j} \rangle_{j=1}^{m_i-1}]$. 
Let us define the constant 
\[C_2= \frac{1}{C_1}\frac{\sqrt{C_1}-1}{\sqrt{C_1}+1}\]

\medskip

{\em Step 5.} 
Define the numbers $y_{i,l}$ as 
\begin{gather} \label{E:Def-y_il}
y_{i,0}=1, \; i\geq 1, \notag \\
y_{i,l}= \prod_{k=1}^l \Big| [\langle b_{i,j} : \gep_{i,j} \rangle_{j=k}^{m_i}]\Big|=  
\prod_{k=0}^{l-1} \big| G\co{k}([\langle b_{i,j}: \gep_{i,j} \rangle_{j=1}^{m_i}])\big|, 
\quad i\geq 1 \tand 1\leq l \leq m_i.
\end{gather}
By Equation~\eqref{E:q-vs-beta}, we have 
\[\frac{1}{q_{i,m_i}}= y_{i,m_i}.\]
By Lemma~\ref{L:distortion-algorithm}, we have 
\begin{equation*}
\frac{1}{C_1} x_{i, m_i-1} \leq y_{i, m_{i-1}} \leq C_1 x_{i, m_i-1},
\end{equation*}
which implies 
\begin{equation}\label{E:q-compare-beta}
\frac{1}{C_1} x_{i, m_i-1} \frac{1}{b_{i,m_i}} 
\leq \frac{1}{q_{i,m_i}}= y_{i,m_i} \leq C_1 x_{i, m_i-1} \frac{1}{b_{i,m_i}}
\end{equation}

\medskip

{\em Step 6.} 
By the uniform bound on the distortion of $G\co{m_i-1}$ and Equations~\eqref{E:N-condition-2} and 
\eqref{E:G-derivative},
\begin{multline}\label{E:small-change}
\Big | \frac{p_{i,m_i}}{q_{i,m_i}} - [\langle b_{i,j} : \gep_{i,j}\rangle_{j=1}^{m_i-1}]\Big| 
\leq C_1 x_{i,m_i-1}^2 \cdot  \frac{4}{3} \cdot \frac{1}{b_{i,m_i}}   \\
\leq C_1 x_{i,m_i-1}^2 \cdot \frac{4}{3} \cdot \frac{1}{b_{i,1}} 
\leq  C_1 x_{i,m_i-1}^2 \cdot \frac{4}{3} \cdot \frac{1}{b_{1,1}}  \\
\leq C_1 x_{i,m_i-1}^2 \cdot \frac{4}{3} \cdot \frac{1}{N}  
\leq \frac{r}{2}\cdot C_2 \cdot  x_{i, m_i-1}^2.
\end{multline}

\medskip

{\em Step 7.} 
We have, 

\begin{align*}\label{E:perdiod-over-rays-estimate} 
\frac{k_i}{q_{i,m_i}} \cdot \frac{1}{x_{i, m_i-1}^2} 
&\leq C_1 \frac{k_i}{1} \cdot  \frac{x_{i,m_i-1}}{b_{i,m_i}x_{i, m_i-1}^2} 
&& (\text{Eq. }\eqref{E:q-compare-beta}) \\
&\leq C_1 \frac{k_i}{b_{i,1}}
\cdot \frac{b_{i,1}}{b_{i,m_i}x_{i, m_i-1}} 
&& \\
&\leq C_1 \frac{4}{3} \frac{k_i}{q_{i-1}^2}
\cdot \frac{5}{4} \frac{1}{y_{i,1} b_{i,m_i} x_{i, m_i-1}} 
&& (\text{Eq. } \eqref{E:q-increment}, \text{Lem. } \ref{L:F-balls-in-cone}) \\
&\leq C_1 \frac{4}{3} \frac{1}{N} 
\cdot \frac{5}{4} C_1 \frac{1}{y_{i,1}b_{i,m_i}^2 y_{i, m_i}} 
&& (\text{Eq. }\eqref{E:period-over-rays}, \text{Eq. }\eqref{E:q-compare-beta}) \\
&\leq C_1\frac{4}{3} \frac{1}{N} 
\cdot \frac{5}{4} C_1 \frac{1}{y_{i,1} y_{i, m_i-1}} \cdot \frac{1}{b_{i,m_i}}
&& (\text{Eq. }\eqref{E:Def-y_il}) \\
&\leq C_1 \frac{4}{3} \frac{1}{N} 
\cdot \frac{5}{4} C_1 C_0
&& (\text{Lem. }\ref{L:simple-q-growth})
\end{align*}
In particular, by Equation~\eqref{E:N-condition-3}, the above inequalities imply that 
\begin{equation}\label{E:bound-on-diameter-multiplier}
\frac{\log 2}{2\pi} \cdot \frac{k_i}{q_{i,m_i}}  \leq \frac{r}{2} C_2  x_{i, m_i-1}^2.
\end{equation}

\medskip

{\em Step 8.}
Let $z\in \BB{C}$ be a point satisfying the hypothesis of the proposition. 
By Equation~\eqref{E:bound-on-diameter-multiplier}, $z$ belongs to a ball of diameter at most 
$\frac{r}{2} C_2  x_{i, m_i-1}^2$ that is tangent to the real line $p_{i,m_i}/q_{i,m_i}$. 
On the other hand, by Equation~\eqref{E:small-change}, $p_{i,m_i}/q_{i,m_i}$ is within $\frac{r}{2} C_2  x_{i, m_i-1}^2$ distance from $[\langle b_{i,j} : \gep_{i,j}\rangle_{j=1}^{m_i-1}]$. 
Hence, $|z- [\langle b_{i,j} : \gep_{i,j}\rangle_{j=1}^{m_i-1}]|\leq \frac{r}{2} C_2  x_{i, m_i-1}^2$.
By Equations~\eqref{E:bdd-deviation-from-center} and \eqref{E:ball-in-ball}, 
the above inequality implies that 
\[z\in B_i^r \ci \C{F}_{m_i-1}([\langle b_{i,j}: \gep_{i,j} \rangle_{j=1}^{m_i-1}]).\]
This finishes the proof of the proposition, by virtue of Lemma~\ref{L:F-balls-in-cone}.
\end{proof}


\subsection{Rigidity of complex quadratic polynomials}\label{SS:inp-q}
Recall the constant $r_1$ introduced in Proposition~\ref{P:sigma-fixed-point}. 

\begin{lem}\label{L:beta-condition}
There is a constant $r_5>0$ satisfying the following. 
Let $f\in \PC{A(r_1)}$ or $f\in A(r_3)\ltimes \{Q_0\}$, with $|f'(0)|\geq 1$. 
If $\gb(f)\in A(r_5)$ then $\ga(f)\in A(r_3)$.
\end{lem}

\begin{proof}
Recall the domain $W$ from Lemma~\ref{P:sigma-fixed-point} and the constant $B_4$ from 
Lemma~\ref{L:preliminary-estimate-on-Index} such that for every 
$f\in \PC{A(r_3)}$ or $f\in A(r_3)\ltimes \{Q_0\}$ we have 
\[ \frac{1}{2\pi} \Big |\int_{\partial W} \frac{1}{z-f(z)} \,dz \Big | \leq B_4.\]
By the holomorphic index formula \eqref{E:holomorphic-index-formula}, this implies that 
\[\big |  \frac{1}{1-e^{2\pi \B{i} \ga(f)}} + \frac{1}{1-e^{2\pi \B{i} \gb(f)}}\Big | \leq B_4,\]
On the other hand, since $|f'(0)|\geq 1$ and $|f'(\gs(f))|\geq 1$, we must have 
$\Im \ga(f)\leq 0$ and $\Im \gb(f)\geq 0$. 
By an elementary calculation one can verify that there is $r_5>0$ such that 
if $\gb(f) \in A(r_5)$ then $\ga(f)\in A(r_3)$. 
\end{proof}

\begin{rem}
Indeed, the proof of the above lemma implies an stronger remarkable property on the relation between 
$\ga(f)$ and $\gb(f)$. 
That is, the set of $\ga(f)$ such that $\gb(f)$ is real and belongs to $(-4_5, r_5)$, is tangent to the real 
line at $0$ with the order of the tangency being quadratic. 
One may use the pre-compactness of the $\IS$ to prove stronger bounds on the location of this curve, 
which in turn may be used to give estimates on the location of the multipliers of the dividing periodic 
points of corresponding $Q_\ga$. 
\end{rem}

The following proposition is the main statement of this section. 
Recall the integers $l_k$ introduced in Equation~\eqref{E:sub-levels}
and the map $\gk$ introduced in Equation~\eqref{E:kappa}. 

\begin{propo}\label{P:main-inclusion}
Given $r_3>0$ as in Theorem~\ref{T:Ino-Shi2} there is an integer $N>0$ such that for 
every sequence of rational numbers $\langle m_i : b_{i,j} : \gep_{i,j}\rangle \in \qg$ in the interval 
$(-1/2, 1/2]$ and every integer $k\geq 1$, 
$M_\ga(\langle m_i : b_{i,j} : \gep_{i,j}\rangle_{i=1}^{k})$ is contained in 
$\gL(\langle \gk_i \rangle_{i=1}^{l_k})$, where $\gk=\gk(\langle m_i : b_{i,j} : \gep_{i,j}\rangle)$. 
\end{propo}

The above proposition combined with Theorem~\ref{T:Ino-Shi2} provides us with 
a constant $N$ such that for every $\langle m_i : b_{i,j} : \gep_{i,j}\rangle \in \qg$
and every $\ga\in M_\ga(\langle m_i : b_{i,j} : \gep_{i,j}\rangle)$, $Q_\ga$ is infinitely
near-parabolic renormalizable of type $\gk(\langle m_i : b_{i,j} : \gep_{i,j}\rangle)$.  

\begin{proof}
Let $r_5$ be the constant obtained in Lemma~\ref{L:beta-condition}. 
Let $N_1$ be the constant obtained in Proposition~\ref{P:rationals-with-growth} with 
$r=\min \{r_5, r_3\}$, and choose $N\geq N_1$ such that 
\[\frac{\log 2}{2 \pi} \cdot \frac{1}{N} \leq \frac{r_3}{2}, \; \frac{4}{3} \cdot \frac{1}{N} \leq \frac{r_3}{2}.\]

Fix $\langle m_i : b_{i,j} : \gep_{i,j}\rangle$ in $\qg$, $k\geq 1$, and 
$\ga\in M_\ga (\langle m_i : b_{i,j} : \gep_{i,j}\rangle_{i=1}^k)$. 
Define the type $\gk= \gk(\langle m_i : b_{i,j} : \gep_{i,j}\rangle)$. 
We need to show that starting with $f_1=Q_\ga$, the sequence of maps $f_i$, for $1\leq i \leq l_k$, 
in equation~\eqref{E:sequence-of-maps} is defined, and each $\ga_i\in A(r_3)$.  
We prove this by an inductive argument. 

Recall that $\gb_1= \frac{1}{2\pi \B{i}} \log \gr_1$, Equation~\eqref{E:first-multiplier}. 
By the PLY inequality~\eqref{E:PLY-copies}, and the above condition on $N$, 
$\ga$ is contained in a disk of radius bounded by $r_3/2$ attached to the real line at 
$[\langle b_{1,i} : \gep_{1,i}\rangle_{i=1}^{m_1}]$. 
Moreover, by Equation~\ref{E:preliminary-bounds}, 
\[\big |[\langle b_{1,i} : \gep_{1,i}\rangle_{i=1}^{m_1}]\big| 
\leq \frac{4}{3} \cdot \frac{1}{N}\leq \frac{r_3}{2}.\] 
Hence, $\ga$ is contained in $A(r_3)$.
Therefore, by Theorem~\ref{T:Ino-Shi2} and Definition~\ref{D:extended-renormalization}, 
$f_1$ is near-parabolic renormalizable of type $\gk_1=b$.
That is, $\nprb{1} (Q_\ga)$ is defined. 

By the definitions, $l_2=m_1$ and $\ga_2=-1/\gb_1$.
Also, for all $i$ with $2\leq i \leq l_2$ (if there is any), we have $\gk_i=t$. 
Thus, for all such $i$, we have $\ga_{i+1}= -1/\ga_i$. 
Proposition \ref{P:rationals-with-growth}, combined with the PLY inequality, 
implies that, for every $i$ with $2\leq i \leq l_2$, $\ga_i \in A(r_3)$. 
In particular, this implies that for every $i$ with $2\leq i \leq l_2+1$, $f_i$ is defined. 
But we still don't know whether $\ga_{l_2+1}$ belongs to $A(r_3)$ or not.

Let $j$ be an integer with $1 \leq j \leq k-1$. 
For $\ga\in M_\ga(\langle m_i : b_{i,j} : \gep_{i,j}\rangle_{i=1}^{k})$, we 
want to show that if $\ga\in \gL_{r_3}(\langle \gk_i \rangle_{i=1}^{l_j})$ 
then $\ga\in \gL_{r_3}(\langle \gk_i \rangle_{i=1}^{l_{j+1}})$. 
Since $\ga \in \gL_{r_3}(\langle \gk_i \rangle_{i=1}^{l_j})$, by definition,  $\ga_{l_j}$ belongs to 
$A(r_3)$ and hence, $f_{l_j+1}$ is defined. 
However, since $\ga$ belongs to $M_\ga(\langle m_i : b_{i,j} : \gep_{i,j}\rangle_{i=1}^{j+1})$, 
PLY inequality and Proposition~\ref{P:rationals-with-growth} with $n_j=0$, imply that 
$\gb_{l_j+1}\frac{1}{2\pi \B{i}}\log \gr_{j+1}$ belongs to $A(r_5)$.
Then, by Lemma~\ref{L:beta-condition}, $\ga_{l_{j}+1}$ belongs to $A(r_3)$, and 
therefore, $\nprb{1}(f_{l_j+1})$ is defined.  
Note the choice of $N$ and $N_1$ at the beginning of the proof. 
By the definition, $\gk_{l_{j}+1}=b$ and for all $l$ with $l_j+2 \leq l \leq l_{j+1}$ (if there is any) 
$\gk_l=t$. 
That is, $\ga_{l_j+2}=-1/\gb_{l_j+1}$ and $\ga_{l+1}= -1/\ga_l$ for all $l$ with $l_{j+1}+2\leq l \leq l_{j+1}$. 
Now we use Proposition~\ref{P:rationals-with-growth} to conclude that for every $l$ with 
$l_j \leq l \leq l_{j+1}$, $\ga_l\in A(r_3)$ and $f_l+1$ is defined. 

By an inductive argument, the proposition follows from the above paragraphs. 
\end{proof}

\begin{propo}
Let $p_i/q_i$, $i\geq 1$, be a sequence of non-zero rational numbers in $(-1/2, 1/2]$ 
such that for every $c$ in $M(\langle p_i/q_i\rangle_{i=1}^\infty)$, $Q_{\ga(c)}$ is infinitely near 
parabolic renormalizable and for every $n\geq 1$ the rotation $\ga_n$ belongs to $A(r_3)$. 
Then, the nest of Mandelbrot copies $M(\langle p_i/q_i\rangle_{i=1}^n)$ shrinks to a single point. 
\end{propo}

\begin{proof}
By the hypothesis, $M_\ga(p_1/q_1, p_2/q_2, \dots)$ is contained in $\gL_{r_3}(\gk)$, where 
\[\gk= \gk (\langle p_i/q_i\rangle_{i=1}^\infty)\] 
is defined in Equation~\eqref{E:kappa}. 
By Theorem~\ref{T:Cantor-structure}, the connected set 
$M_\ga(\langle p_i/q_i\rangle_{i=1}^\infty)$ 
must be a single point.  
\end{proof}

We will not use the following proposition in this paper, but it is stated for future purposes. 

\begin{propo}\label{P:nonempty-intersection}
For every sequence of rational numbers $\langle m_i : a_{i,j} : \gep_{i,j}\rangle$  
with $a_{i,j} \geq N$, there is $\ga\in M_\ga(\langle m_i : a_{i,j} : \gep_{i,j}\rangle)$ 
such that $Q_\ga$ is infinitely near-parabolic renormalizable of  type 
$\gk(\langle m_i : b_{i,j} : \gep_{i,j}\rangle)$. 
\end{propo}

\begin{proof}
By the continuity of the relations between $\ga_n$ and $\gb_n$ as well as $\ga_{n-1}$ in terms of 
$\ga_n$ or $\gb_n$, there is $\ga\in A(r_3)$ such that $Q_\ga$ is infinitely near parabolic 
renormalization of type $\gk$.
On the other hand, if $Q_\ga$ is infinitely near-parabolic renormalizable, then the orbit of the 
critical point remains uniformly bounded in $\BB{C}$. 
This implies that $\ga$ belongs to $M_\ga$. 
It is not difficult to see that $Q_\ga$ has the correct combinatorial rotations a the dividing periodic points.
More details shall be added later.
\end{proof}

\bibliographystyle{amsalpha}
\bibliography{Data}
\end{document}